\documentclass[a4paper,10pt,twoside,reqno]{amsart}

\usepackage[utf8]{inputenc}
\usepackage{amsmath, amsfonts, amssymb,amsthm}
\usepackage[mathscr]{eucal}
\usepackage[pdftex]{graphicx}
\usepackage{color}
\usepackage{multicol}
\usepackage{hyperref}
\usepackage{multirow,booktabs}
\usepackage[format=hang, justification=centerlast, skip =  0pt, width=0.95\linewidth, font=small]{caption}

\usepackage[T1]{fontenc}
\usepackage[sc, osf]{mathpazo}
\usepackage{enumitem}
\setlist[itemize]{topsep=0.2em, itemsep=0.2em, leftmargin=2em}
\setlist[enumerate]{topsep=0.2em, itemsep=0.2em, leftmargin=2em}

\usepackage{marginfix} 

\newcommand{\R}{\mathbb{R}}
\newcommand{\C}{\mathbb{C}}
\newcommand{\s}{\mathbb{S}}
\newcommand{\h}{\mathbb{H}}
\newcommand{\E}{\mathbb{E}}
\newcommand{\M}{\mathbb{M}}

\newcommand{\Sb}{\mathbb{S}^3_{\mathrm{B}}}
\newcommand{\SL}{\widetilde{\mathrm{SL}}_2(\R)}
\newcommand{\Nil}{\mathrm{Nil}_3}

\newcommand{\X}{\mathfrak{X}}
\newcommand{\df}{\mathrm{d}}
\newcommand{\id}{\mathrm{id}}
\newcommand{\Rot}{\mathrm{Rot}}

\newcommand{\prodesc}[2]{\left\langle#1,#2\right\rangle}

\newcommand{\Iso}{\mathrm{Iso}}

\DeclareMathOperator{\arcsinh}{arcsinh}

\DeclareMathOperator{\arctanh}{arctanh}
\DeclareMathOperator{\sech}{sech}
\DeclareMathOperator{\pIm}{Im}

\DeclareMathOperator{\length}{Length}
\DeclareMathOperator{\Area}{Area}
\DeclareMathOperator{\Length}{Length}
\DeclareMathOperator{\Div}{div}

\newtheorem{theorem}{Theorem}[section]
\newtheorem{proposition}[theorem]{Proposition}

\newtheorem{lemma}[theorem]{Lemma}

\theoremstyle{definition}
  \newtheorem{definition}[theorem]{Definition}
  \newtheorem{remark}[theorem]{Remark}
  \newtheorem{example}[theorem]{Example}

\numberwithin{equation}{section}

\title{Conjugate Plateau constructions in product spaces}
  
  \author{Jesús Castro-Infantes}
  \address{Departamento de Geometría y Topología, Universidad de Granada, Spain.}
  \email{jcastroinfantes@ugr.es}

  \author{José M. Manzano}
  \address{Departamento de Matemáticas, Universidad de Jaén, Spain.}
  \email{jmprego@ujaen.es}

  \author{Francisco Torralbo}
  \address{Departamento de Geometría y Topología, Universidad de Granada, Spain.}
  \email{ftorralbo@ugr.es}

\subjclass[2020]{Primary 53A10; Secondary 53C30}

\keywords{Constant mean curvature, compact surfaces, homogeneous three-manifolds, product spaces, conjugate constructions}

\usepackage{soul}

\begin{document}

\begin{abstract}
  This survey paper investigates, from a purely geometric point of view, Daniel's isometric conjugation between minimal and constant mean curvature surfaces immersed in homogeneous Riemannian three-manifolds with isometry group of dimension four. On the one hand, we collect the results and strategies in the literature that have been developed so far to deal with the analysis of conjugate surfaces and their embeddedness. On the other hand, we revisit some constructions of constant mean curvature surfaces in the homogeneous product spaces $\mathbb{S}^2\times\mathbb{R}$, $\mathbb{H}^2\times\R$ and $\R^3$ having different topologies and geometric properties depending on the value of the mean curvature. Finally, we also provide some numerical pictures using Surface Evolver.
\end{abstract}

\maketitle

\setcounter{tocdepth}{1}
\tableofcontents


\section{Introduction}

The study of minimal and constant mean curvature surfaces ($H$-surfaces in the sequel) represents a central topic in Surface Theory with a long trajectory dating back to works of Euler and Lagrange. Minimal surfaces were popularized by Plateau's experiments on soap films, which gave rise to the so-called Plateau problem of finding the least-area surface spanning a given contour. Likewise, $H$-surfaces with $H>0$ were first investigated in connection with the isoperimetric problem of finding the least-area surface enclosing a given volume. A geometrically appealing question that have caught the attention of many geometers during the last century is to produce non-trivial examples of $H$-surfaces with a high number of symmetries. This is well exemplified by the vast number of triply periodic minimal surfaces in Euclidean space $\mathbb{R}^3$ that appear in the literature, e.g., see Schoen's report~\cite{Schoen} and Weber's Minimal Surface Archive~\cite{Weber}.

However, in the case $H>0$, very few examples of $H$-surfaces were known prior to Lawson's discovery~\cite{Law} of an isometric duality or conjugation for $H$-surfaces in space forms $\mathbb{M}^3(c)$ of constant sectional curvature $c\in\R$. It allowed him to produce the first examples of doubly-periodic embedded $1$-surfaces in $\R^3=\mathbb{M}^3(0)$ by means of minimal surfaces in the round sphere $\mathbb{S}^3=\mathbb{M}^3(1)$. The minimal surfaces he employed were indeed solutions to appropriate Plateau problems over certain spherical geodesic polygons whose sides become planar lines of curvature in the conjugate $1$-surface. Furthermore, he noticed that an extension of the minimal surface by axial symmetries about the boundary components agrees with an extension of the conjugate $1$-surface by mirror symmetry about the planes containing the conjugate boundary. This technique is known as the \emph{conjugate Plateau construction} and has been used to obtain many examples of complete $1$-surfaces in $\R^3$ as we explain below. It is important to say that Lawson's result actually yields a $2$-parameter isometric deformation in which we can change not only the ambient curvature but also a phase parameter which rotates the second fundamental form. In particular, Lawson's correspondence generalizes the classical notion of associate Bonnet family of minimal surfaces in $\R^3$, see also~\cite{KP}. However, when we speak of \emph{conjugation}, we are prescribing the phase angle equal to $\frac\pi 2$, which gives a more precise control of the corresponding surfaces, as we will explain throughout this work. 

As for the conjugation from minimal surfaces in $\s^3$ to $1$-surfaces in $\R^3$, Lawson's constructions were resumed by Karcher~\cite{K89a}, who realized that many of Schoen's triply periodic minimal surfaces in~\cite{Schoen} admit many planes of mirror symmetry and their fundamental pieces can be obtained as conjugate Plateau constructions. This allowed him to deform Schoen's minimal examples into $H$-surfaces by considering geodesic polygons in the $\s^3$ similar to those needed in $\R^3$. In his paper, Karcher also improved some of Lawson's ideas about conjugate curves and developed their connection with the different Hopf fibrations in $\s^3$, making it clear that the conjugate technique had the capability to produce beautiful highly-symmetric $H$-surfaces. Gro\ss{}e-Brauckmann~\cite{G} took it one step further and produced many interesting examples of $1$-surfaces in $\R^3$, including $k$-unduloids, whose $k\geq 2$ ends are asymptotic to the classical Delaunay unduloids. In fact, triunduloids (with $k=3$) were later proved to be properly embedded $1$-surfaces in $\R^3$ by Gro\ss{}e-Brauckmann, Kusner and Sullivan~\cite{GKS}. Gro\ss{}e-Brauckmann and Wohlgemuth~\cite{GW} also used the conjugation to prove that the triply-periodic minimal surface known as \emph{gyroid} is embedded and can be deformed into an $H$-surface (see also~\cite{Grosse-Brauckmann1997} for numerical experiments). The conjugation has also been used in different ambient spaces: Karcher, Pinkall and Sterling~\cite{KPS88} constructed minimal surfaces in $\mathbb{S}^3$ by conjugating minimal surfaces in $\mathbb{S}^3$; Karcher~\cite{K05b} and Rossman~\cite{Rossman} obtained $1$-surfaces in hyperbolic space $\mathbb{H}^3=\mathbb{M}^3(-1)$ from minimal surfaces in $\mathbb{R}^3$; and Karcher and Polthier~\cite{KP} also used related techniques to revisit conjugate minimal surfaces in $\R^3$.

There have been other approaches to produce $H$-surfaces in space forms, three of which will be highlighted here concerning compact examples. First, using implicit methods, Karcher, Pinkall and Sterling~\cite{KPS88} obtained compact embedded minimal surfaces in $\mathbb{S}^3$ with arbitrary genus, and Kapouleas~\cite{Kapouleas1991, Kapouleas1995} found compact immersed $H$-surfaces in $\mathbb{R}^3$ with arbitrary genus. The second technique is often referred to as the Dorfmeister--Pedit--Wu (DPW) method~\cite{DPW1998}, which uses integrable systems and can be thought of as a global version of the Weierstra\ss\ representation. The DPW method was pioneered by Pinkall and Sterling~\cite{PS1989} and Hitchin~\cite{Hitchin1990} to study $H$-tori. Heller, Heller and Traizet~\cite{HHT2021} (see also~\cite{BHS2021} for numerically constructed surfaces) have recently shown the existence, for large genus $g$, of a complete and smooth family of compact $H$-surfaces in the round $3$-sphere deforming the Lawson surface $\xi_{1,g}$ (see~\cite[\S6]{Law}) into a doubly covered geodesic 2-sphere. Third and last, we have gluing methods, illustrated by Kapouleas' constructions of compact minimal surfaces in $\mathbb{S}^3$ by connecting two parallel Clifford tori~\cite{Kapouleas2010} or equatorial 2-spheres~\cite{Kapouleas2017} by means of catenoidal bridges.

Back to the conjugate techniques, they have been extended recently to the case of simply connected homogeneous $3$-manifolds with isometry group of dimension $4$, which are the most symmetric spaces after the space forms. They consist in a $2$-parameter family $\E(\kappa,\tau)$, $\kappa \neq 4\tau^2$, containing the product spaces $\mathbb{M}^2(\kappa)\times\R=\E(\kappa,0)$ as well as the Lie groups $\SL$, $\mathrm{SU}(2)$ and $\Nil$ with some left-invariant metrics (see Table~\ref{tab:Ekt}). Their geometry will be discussed in~\S\ref{sec:spaces}. The cornerstone of this extension is the work of Daniel~\cite{Dan}, who found a Lawson-type correspondence within this family and connects, for a phase angle of $\frac{\pi}{2}$, minimal surfaces in $\E(4H^2+\kappa,H)$ with $H$-surfaces in $\mathbb{M}^2(\kappa)\times\R$ for all $\kappa,H\in\R$ (see Table~\ref{tab:conjugation-cases}).  Hauswirth, Sa Earp and Toubiana~\cite{HST} also found this correspondence as a particular case of the associate family for $H=0$ and arbitrary $\kappa\in\R$. We also remark that Daniel's correspondence contains other cases with phase angle $\frac\pi2$, but mirror symmetries only exist in the case of product spaces, which makes this case the most tractable one (see Lemmas~\ref{lem:horizontal-geodesics} and~\ref{lem:vertical-geodesics}). Observe that, in the case $\kappa = 0$, we have that $\E(4H^2, H) = \mathbb{M}^3(H^2)$ and Daniel's conjugation reduces to the classical Lawson's conjugation between minimal surfaces in $\mathbb{S}^3(H^2)$ and $H$-surfaces in $\mathbb{R}^3$.

The starting tools in a conjugate construction, i.e.\ the Plateau problem (and its improper version known as the Jenkins--Serrin problem) in $\E(\kappa,\tau)$ have been solved under quite general conditions~\cite{MY82b,NR,CR,MRR11,Melo,Younes} (see~\S\ref{sec:plateau} and~\S\ref{sec:JS}). They provide us with plenty of surfaces with boundary a geodesic polygon (with possibly some components at infinity in the Jenkins--Serrin case) that are at our disposal to act as initial minimal surfaces in the conjugate construction. Also, the extension by axial and mirror symmetries, as well as the absence of singularities rely on general results that also apply, see Proposition~\ref{prop:conjugation-completion-by-simmetries}. 

It is necessary to point out that the family $\mathbb{E}(\kappa,\tau)$ has fewer isometries than space forms, which forces us to consider only initial polygons consisting of horizontal and vertical geodesics (with respect to the Killing submersion $\pi: \E(\kappa, \tau) \to \mathbb{M}^2(\kappa)$, see \S\ref{sec:spaces}). This reduces both the number of possible configurations and the directions in which we are able to control the involved surfaces. 

In~\S\ref{sec:conjugation}, we will collect and present different features of Daniel's correspondence that are essential in the conjugate constructions, specially a detailed study of conjugate curves, some of them in more generality than those in the literature. We will pay special attention in~\S\ref{subsec:surfaces-preserved-by-sister-correspondence} to some classes of surfaces in which the correspondence is well understood, such as equivariant surfaces, ideal Scherk graphs or ruled minimal surfaces. In this respect, it is worth mentioning that Daniel's correspondence has been a formidable tool to analyze surfaces satisfying preserved geometric conditions and has played a key role in their classification, e.g., $H$-surfaces with zero Abresch--Rosenberg quadratic differential~\cite{AR,ER,DomMan}, $H$-surfaces with certain bounds on the intrinsic or extrinsic area growth~\cite{MN} or $H$-surfaces with constant Gauss curvature~\cite{DDV}. 

We will now give a brief overview of constructions of $H$-surfaces in $\E(\kappa,\tau)$-spaces that use conjugation, and we will begin with the minimal case. Morabito and Rodríguez~\cite{MR} used a conjugate Jenkins--Serrin construction to obtain minimal $k$-noids in $\mathbb H^2\times\R$ with genus $0$ and $k$-ends asymptotic to vertical planes, as well as minimal saddle towers in $\mathbb H^2\times\R$ similar to those in $\R^3$ obtained by Karcher~\cite{K88}. Pyo~\cite{P} also found the minimal $k$-noids independently assuming additionally that the vertical planes are disposed symmetrically. Rodríguez~\cite{Rod} extended this construction to give minimal examples in $\mathbb{H}^2\times\R$ with infinitely many ends and an arbitrary (finite or countable) number of limit ends. Martín and Rodríguez~\cite{MarR} have also used a conjugate Jenkins--Serrin construction to produce minimal embeddings of any non-simply connected planar domain in $\mathbb{H}^2\times\mathbb{R}$. Mazet, Rodríguez and Rosenberg~\cite{MRR14} used a conjugate Jenkins--Serrin construction to produce examples for their classification of doubly-periodic embedded minimal surfaces in $\mathbb{H}^2\times\mathbb{R}$. The first two authors~\cite{CM} have obtained minimal $k$-noids in $\mathbb H^2\times\R$ with genus $1$ and $k\geq 3$ ends by another conjugate Jenkins-Serrin construction, inspired by a construction of Plehnert~\cite{Ple12} of similar surfaces for $H=\frac{1}{2}$. Plehnert and the second and third authors~\cite{MPT} have also obtained embedded minimal surfaces of type Schwarz P in $\mathbb{S}^2\times\mathbb{R}$ by a conjugate Plateau construction.

In the non-minimal case, the study of conjugate constructions was initiated independently by Plehnert~\cite{Ple14} (to obtain $k$-noids with genus $0$ in $\mathbb{H}^2\times\R$ and $0<H\leq\frac{1}{2}$) and by the second and third authors~\cite{MT14} (to obtain horizontal unduloid-type $H$-surfaces in $\mathbb{H}^2\times\mathbb{R}$ and $\mathbb{S}^2\times\mathbb{R}$). The latter was subsequently improved in~\cite{MT20} to obtain horizontal nodoid-type $H$-surfaces in $\mathbb{H}^2\times\mathbb{R}$ and $\mathbb{S}^2\times\mathbb{R}$ and to determine which of these examples are embedded. The second and third authors have also provided compact orientable embedded $H$-surfaces in $\mathbb{S}^2\times\mathbb{R}$ with arbitrary genus by means of a different conjugate Plateau construction. In the case of $\mathbb{H}^2\times\R$ and $0<H\leq\frac{1}{2}$, Rodríguez and the first and second authors~\cite{CMR} have produced $k$-noids and saddle towers that extend Plehnert's, plus another family of $H$-surfaces, called $k$-nodoids, with genus $0$ and $k\geq 2$ ends that approach the asymptotic vertical planes from the convex side (unlike the $k$-noids, which lie in their concave side).

We will illustrate how the technique works by sketching five of the above constructions (see Table~\ref{tab:constructions}). A deeper motivation for each of them will be given in the corresponding sections, but we will say here that we have chosen them because we would like to cover all ranges of the mean curvature in both $\mathbb{S}^2\times\mathbb{R}$ and $\mathbb{H}^2\times\mathbb{R}$ in order to visualize the dissimilarities between the critical, supercritical and subcritical cases (see~\S\ref{subsec:multigraphs}). These constructions are also part of some ongoing open research lines to which the authors have contributed.
\begin{itemize}
  \item On the one hand, we are interested in the classification of compact embedded $H$-surfaces in $\mathbb{S}^2\times\mathbb{R}$ attending to their genus. The only non-equivariant known examples are the families of $H$-tori given by Theorem~\ref{thm:horizontal-Delaunay-embeddedness} with $H>\frac{1}{2}$ and the arbitrary genus $H$-surfaces given by Theorem~\ref{thm:arbitrary-genus-H-surface-embeddedness} with $H<\frac{1}{2}$. We expect that the former are the only embedded $H$-tori. Since they are not equivariant, their characterization should be more involved than in the case of $H$-tori in $\mathbb{S}^3$ (see Andrews and Li~\cite{AL2015}).

  \item On the other hand, we are also interested in producing examples of $H$-surfaces with finite total curvature in $\mathbb{H}^2\times\mathbb{R}$ displaying different topologies and asymptotic behaviors. In~\S\ref{sec:genusone} we discuss and collect some properties of this class of surfaces in the minimal case, which is probably the most studied class of minimal surfaces in $\mathbb{H}^2\times\mathbb{R}$. It is widely believed that $H$-surfaces in $\mathbb{H}^2\times\mathbb{R}$ with $0<H<\frac{1}{2}$ should behave similarly, but this is still an unexplored area of research. In particular, we expect that the $(H,k)$-noids and $(H,k)$-nodoids given by Theorem~\ref{th:knoids-genus-0} have finite total curvature when the mean curvature is not critical. 
\end{itemize}

\begin{table}[htbp]
  \newlength{\twocolumns}
  \setlength{\twocolumns}{0.40\linewidth}
  \addtolength{\twocolumns}{\tabcolsep}
  \newlength{\threecolumns}
  \setlength{\threecolumns}{0.60\linewidth}
  \addtolength{\threecolumns}{2\tabcolsep}
  \centering
  \small

  \begin{tabular}{cp{0.20\linewidth}p{0.20\linewidth}p{0.20\linewidth}p{0.20\linewidth}}    \toprule
  & \centering $H= 0$ &\centering  $0 < H < \tfrac{1}{2}$ & \centering $H= \tfrac{1}{2}$ &\centering  $H > \tfrac{1}{2}$     \tabularnewline \cmidrule(r){2-5}
  \multirow{5}{*}{\rotatebox[origin=c]{90}{$\mathbb{S}^2\times \mathbb{R}$}}
  & \textbf{Theorem~\ref{thm:orientable-minimal-arbitrary-genus-examples-S2xS1}}:\newline Schwarz\newline P-surfaces.
      & \textbf{Theorem~\ref{thm:arbitrary-genus-H-surface-embeddedness}}:\newline Compact embedded surfaces of arbitrary genus. 
      & &
      \textbf{Theorem~\ref{thm:horizontal-Delaunay-embeddedness}}:\newline Embedded tori.          
      \tabularnewline 
      &  
      & \multicolumn{3}{c}{
        \textbf{Theorem~\ref{thm:horizontal-Delaunay}}: Horizontal Delaunay surfaces.
      }
      \tabularnewline \cmidrule(r){2-5}
      \multirow{3}{*}{\rotatebox[origin=c]{90}{$\mathbb{H}^2\times \mathbb{R}$}}
      & \textbf{Theorem~\ref{th:knoids-genus-1}}:\newline Genus-$1$ $k$-noids.
      &\multicolumn{2}{p{\twocolumns}}{  \textbf{Theorem~\ref{th:knoids-genus-0}}: Saddle towers, $(H,k)$-noids and $(H,k)$-nodoids.}
      &\textbf{Theorem~\ref{thm:horizontal-Delaunay}}:\newline Horizontal Delaunay surfaces.           \tabularnewline 
  \bottomrule
  \end{tabular}

  \caption{Constructions sketched in the document. Note that the special value $H=\frac{1}{2}$ is the critical mean curvature in $\mathbb{H}^2\times\R$ but it is also relevant for the embeddedness of compact $H$-surfaces in $\mathbb{S}^2\times\R$, see~\S\ref{sec:delaunay} and~\S\ref{sec:genus}.}
  \label{tab:constructions}
\end{table}

In conjugate constructions, one easily reaches the central but tough question of embeddedness, which probably has not been well understood yet. We will treat it carefully throughout the constructions in this survey by emphasizing the different approaches to answer this question. In the case of horizontal Delaunay $H$-surfaces, embeddedness follows from spotting a geometric function in the common stability operator of the conjugate immersions and finding $1$-parameter groups of isometries that induce this function (see \S\ref{subsubsec:horizontal-Delaunay-embeddedness}). In the case of arbitrary genus $H$-surfaces, we rely on the estimates on the curvature of the boundary of an $H$-bigraph given by the second author in~\cite{Man13}. In the case of minimal Schwarz P-surfaces, convexity of some boundary curves come in handy along with some isoperimetric inequalities. In the case of $(H,k)$-noids and $(H,k)$-nodoids, we are able to find some embedded limits and then we can ensure that some of the examples are embedded (and some of them are not) by continuity (see Proposition~\ref{prop:continuity}). Finally, in the case of the minimal $k$-noids of genus $1$, we use the Krust type property given by Hauswirth, Sa Earp and Toubiana~\cite{HST} (see Proposition~\ref{prop:Krust}) to also show that the surfaces are embedded provided that the parameters are controlled. The Krust property is an essential tool that facilitates the conjugation of minimal surfaces in $\mathbb{H}^2\times\mathbb{R}$ but it is not true in general for other values of $H$ (see~\S\ref{subsubsec:knoids-embeddedness}).

Numerical methods are an invaluable tool in the visualization of $H$-surfaces and prove useful to forecast theoretical results. We will highlight three different approaches in this respect. First, Weierstra\ss\ representation along with numerical solutions of the period problems have been used to represent many families of minimal surfaces in $\mathbb{R}^3$ (see the Mathematica notebooks given by Weber in~\cite{Weber} for details). Second, the use of the DPW method along with a numerical adjustment of parameters has been developed to get high genus $H$-surfaces in $\mathbb{S}^3$ by Bobenko, Heller and Schmitt~\cite{HS2015, BHS2021}, where the figures in the first paper were done with the software Xlab developed by Schmitt). This approach presents the drawback that particular Weierstra\ss\ data might be difficult to derive in more complicated examples. 
A last and more \emph{variational} approach was developed by Brakke in his software Surface Evolver~\cite{Brakke1992} by minimizing energies (e.g.\ area functional) of triangulated surfaces under certain constraints (e.g.\ volume and boundary constrains, a.k.a.\ free boundary problem) to get $H$-surfaces, although the method may display some issues when approximating the solution since fundamental domains might not be stable minima.  In \S\ref{sec:numerical-examples}, we present some numerical experiments using Surface Evolver in order to visualize the minimal examples in $\mathbb{S}^2(\kappa) \times \mathbb{R}$ constructed in \S\ref{sec:periodic}.
We also remark that, to overcome some of the issues in the above methods, Pinkall and Polthier~\cite{PP1993} used discrete differential techniques to implement the conjugation directly. The philosophy of this procedure is that the Plateau solution is usually stable and easier to obtain by minimization than the free boundary solution which is more likely unstable. Examples of surfaces produced with this technique can be found in~\cite{PP1993, GP1997}, where the software \textsc{grape}~\cite{grape} is used. We would like to say that, unfortunately and as far as we know, neither Xlab nor \textsc{grape} are publicly available.

As a final comment, in the present work we have made an effort to keep the notation homogeneous by writing: a tilde for the elements of the initial minimal surface in $\E(4H^2+\kappa,H)$, no tilde for the corresponding elements of the conjugate target $H$-surface in $\mathbb{M}^2(\kappa)\times\R$, and an asterisk $^*$ to indicate the completion of the surface after successive reflections about its boundary components. Also, we have indicated as subscripts the parameters that are part of the construction (e.g., $\Sigma_{a,b}$), and in functional notation those which are auxiliary and will likely disappear after a continuity argument (e.g., $\Sigma(a,b)$). Regarding the figures, we have colored the boundary curves in red and blue for the horizontal and vertical components of the initial polygon, respectively. We hope this will help the reader to easily follow some of the geometric discussions in the document. Finally, we have deliberately not normalized the spaces by homotheties, so that we will do our constructions in $\mathbb{H}^2(\kappa)\times\R$ and $\mathbb{S}^2(\kappa)\times\R$. This is because the limit case $\kappa=0$ always gives some insight by comparing with the Euclidean counterparts.

\noindent\textbf{Acknowledgement.} The authors are supported by project \textsc{pid2019.111531ga.i00}, the first author is also partially supported by project \textsc{pid2020.117868gb.i00}, both funded by \textsc{mcin/\-aei/\-10.13039/\-501100011033}. The second author is also supported by the Ramón y Cajal programme of \textsc{mcin/aei} and by a \textsc{feder-uja} project (ref.\ 1380860); the third author is also supported by the Programa Operativo \textsc{feder} Andalucía 2014-2020, grant no.\ \textsc{e-fqm-309-ugr18}.


\section{The geometry of $\E(\kappa,\tau)$-spaces}\label{sec:spaces}

Simply-connected oriented homogeneous Riemannian 3-manifolds with 4-dim\-ensional isometry group can be arranged in a 2-parameter family $\mathbb{E}(\kappa,\tau)$, where $\kappa,\tau\in\R$ and $\kappa-4\tau^2\neq 0$. These parameters are geometrically characterized by the existence of a Riemannian submersion $\pi:\E(\kappa,\tau)\to\mathbb{M}^2(\kappa)$ whose fibers are the integral curves of a unit Killing vector field $\xi$ (also called \emph{Killing submersion}), and such that the bundle curvature is constant and equal to $\tau$, that is
\begin{equation}\label{eqn:connection-xi}
  \overline\nabla_X\xi=\tau X\times\xi
\end{equation}
holds true for all vector fields $X\in\X(\E(\kappa,\tau))$, see~\cite{Dan,Man14}. The cross product $\times$ reflects the orientation of the ambient space, being $\{u,v,u\times v\}$ a positively oriented basis for all linearly independent tangent vectors $u$ and $v$. As a matter of fact, $\E(\kappa,\tau)$ and $\E(\kappa,-\tau)$ are the same space for all $\kappa$ and $\tau$ but opposite orientations have been chosen. More generally, $\E(a^2\kappa,a\tau)$ is homothetic to $\E(\kappa,\tau)$ with a conformal factor $a^2$ for any constant $a\neq 0$.

In the case $\kappa=4\tau^2$, the above description is also valid, though the space $\E(\kappa,\tau)$ has constant sectional curvature (it is isometric to the space form $\mathbb{M}^3(\tau^2)$), whence its isometry group is 6-dimensional. We remark that hyperbolic spaces $\mathbb{H}^3(c)$ are not $\mathbb{E}(\kappa,\tau)$-spaces for any sectional curvature $c<0$ because they do not admit Killing fields of constant length.

All $\E(\kappa,\tau)$-spaces are isometric to Lie groups endowed with left-invariant metrics, except for $\E(\kappa,0)$ with $\kappa>0$, see~\cite[Thm.~2.4]{MeeksPerez}. Indeed, the condition $\tau=0$ indicates that the distribution orthogonal to $\xi$ is integrable, so $\mathbb{E}(\kappa,0)$ is better thought of as the Riemannian product space $\mathbb{M}^2(\kappa)\times\R$ (observe that there is no Lie group with underlying manifold $\mathbb{S}^2\times\mathbb{R}$). As shown in Table~\ref{tab:Ekt}, in the case $\tau\neq 0$, we encounter the universal cover of the special linear group $\SL$, the Heisenberg group $\Nil$, and the special unitary group $\mathrm{SU}(2)$, endowed with left-invariant metrics in which $\xi$ is a biinvariant vector field, and hence Killing. In the case of $\mathrm{SU}(2)$, these spaces are known as Berger spheres and will be denoted by $\Sb(\kappa,\tau)$ in the sequel. Observe that the Lie groups $\mathrm{SU}(2)$ and $\SL$ also admit left-invariant metrics with 3-dimensional isometry group that will not be considered here.

\begin{table}[ht]
\begin{tabular}{c|ccc}
    \toprule
            & $\kappa>0$    &   $\kappa=0$  &   $\kappa<0$ \\\cmidrule{2-4}
$\tau=0$    & $\mathbb{S}^2\times \mathbb{R}$   &$\mathbb{R}^3$   & $\mathbb{H}^2\times\mathbb{R}$ \\
$\tau \neq 0$   &   $\mathrm{SU}(2)$  &   $\Nil$    &    $\SL$ \\
\bottomrule
\end{tabular}
\caption{Different geometries in $\mathbb{E}(\kappa,\tau)$-spaces.}\label{tab:Ekt}\end{table}

As Killing submersions, the $\E(\kappa,\tau)$-spaces admit natural notions of vertical and horizontal directions, defined as those tangent and orthogonal to the unit Killing vector field $\xi$, respectively. Also, the isometries spanned by $\xi$, called \emph{vertical translations}, form a $1$-parameter group of isometries $\{\Phi_t\}_{t\in\R}$ that plays a fundamental role in the geometry of these spaces.

\subsection{Geodesics}
There are two distinguished types of geodesics in $\E(\kappa,\tau)$, namely the vertical ones (fibers of the submersion over $\mathbb{M}^2(\kappa)$), and horizontal ones (horizontal lifts of geodesics of $\M^2(\kappa)$). More generally, if $\tau\neq 0$, then any non-vertical geodesic $\gamma$ of $\E(\kappa,\tau)$ with unit speed projects onto a curve $\pi\circ\gamma$ of $\M^2(\kappa)$ of constant curvature, say $c$, and meets the Killing direction with constant angle $\langle\gamma',\xi\rangle=\frac{c}{\sqrt{4\tau^2+c^2}}$, see~\cite[Prop.\ 3.6]{Man14} and Figure~\ref{fig:geodesics}. In the case $\tau=0$, any non-vertical geodesic of $\E(\kappa,\tau)=\M^2(\kappa)\times\R$ factors as the product of geodesics of each factor, so it always projects onto a geodesic of $\M^2(\kappa)$. 

\begin{figure}[ht]
{\centering\includegraphics[width=0.9\textwidth]{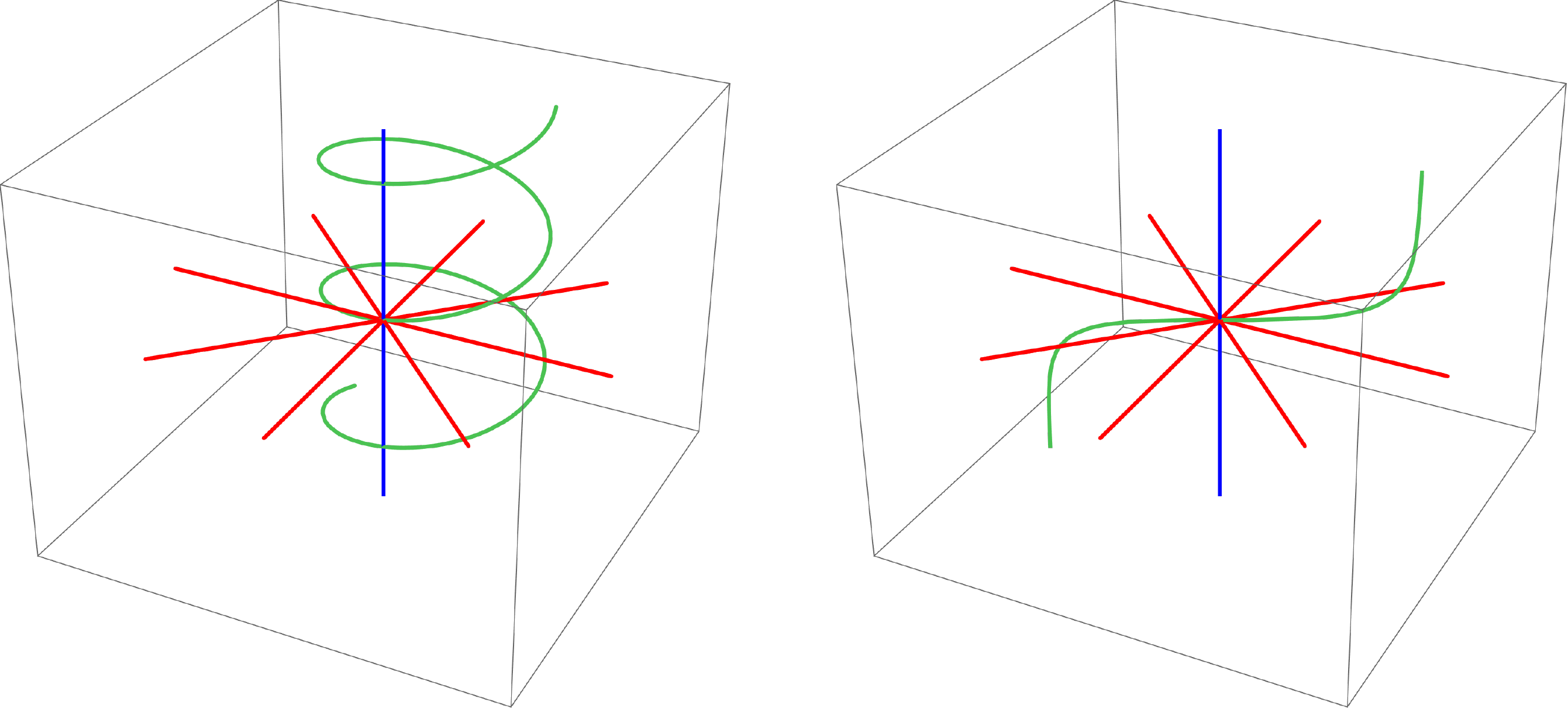}}
\caption{A vertical geodesic in blue, four horizontal geodesics in red, and a oblique geodesic in green. They are represented in $\Nil$ (left) and in $\mathbb{H}^2\times\R$ (right) in the Cartan model, see~\S\ref{subsec:working-coordinates}.}\label{fig:geodesics}
\end{figure}

Horizontal and vertical geodesics have infinite length if $\kappa\leq 0$. However, if $\kappa>0$ and $\tau\neq 0$, the length of all vertical geodesics is $\frac{8\tau\pi}{\kappa}$, whereas the length of all horizontal geodesics is $\frac{4\pi}{\sqrt{\kappa}}$. Note that this is twice the length of a great circle of $\s^2(\kappa)$ because horizontal geodesics in Berger spheres project two-to-one onto great circles. This contrasts the case $\kappa>0$ and $\tau=0$ because horizontal geodesics of $\s^2(\kappa)\times\R$ project one-to-one onto great circles of $\s^2(\kappa)$. 

Explicit parametrizations of all geodesics of $\E(\kappa,\tau)$, $\kappa\leq 0$, as well as some discussion of their minimization properties can be found in~\cite[\S2]{MN}. The case $\kappa > 0$ and $\tau \neq 0$ was studied by Rakotoniaina in~\cite{Rakotoniaina1985}. A discussion of Jacobi fields along any geodesic of any $\E(\kappa,\tau)$ can be found in~\cite[\S4]{DomMan}.

\subsection{Isometries}
An isometry $f\in\Iso(\E(\kappa,\tau))$ induces another isometry $h\in\Iso(\M^2(\kappa))$ such that $\pi\circ f=h\circ\pi$. Conversely, given $h\in\Iso(\mathbb{M}^2(\kappa))$, there is an orientation-preserving isometry $f_+\in\Iso(\E(\kappa,\tau))$ such that $\pi\circ f_+=h\circ\pi$. If $\tau=0$, then there is also a orientation-reversing isometry $f_-\in\Iso(\E(\kappa,\tau))$ such that $\pi\circ f_-=h\circ\pi$. Both $f_+$ and $f_-$ are unique up to composition with vertical translations. Moreover, all the isometries of $\E(\kappa,\tau)$ are of these types provided that $\kappa-4\tau^2\neq 0$; in particular, there are no orientation-reversing isometries in $\E(\kappa,\tau)$ if $\kappa-4\tau^2\neq 0$ and $\tau\neq 0$. We refer to~\cite[Thm.\ 2.8]{Man14} for a classification of the isometries that preserve the Killing direction in a general Killing submersion.

In $\E(\kappa,\tau)$, there are rotations of any angle about any vertical geodesic, which are recovered as orientation-preserving lifts of rotations in $\mathbb{M}^2(\kappa)$ with center the point onto which the vertical geodesic projects. In the case of horizontal geodesics, there are axial symmetries (i.e., rotations of angle $\pi$) about them, recovered as orientation-preserving lifts of axial symmetries with respect to their projections, but there are no rotations of arbitrary angle about a horizontal geodesic if $\kappa-4\tau^2\neq 0$. These are the only orientation-preserving isometries that keep all points of a certain geodesic \emph{fixed}. 

\begin{lemma}\label{lem:translations}
Given a non-vertical geodesic $\gamma:\R\to\E(\kappa,\tau)$, there is a unique $1$-parameter group $\{T_t\}_{t\in\R}$ of isometries of $\E(\kappa,\tau)$ such that $T_t(\gamma(s))=T_{t+s}(\gamma(0))$ for all $t,s\in\R$.
\end{lemma}

\begin{proof}
Since $\pi\circ\gamma$ is a curve of constant curvature parametrized with constant speed, we can consider a $1$-parameter group $\{h_t\}_{t\in\R}$ of isometries of $\mathbb{M}^2(\kappa)$ such that $h_t(\pi(\gamma(0))=\pi(\gamma(t))$. Therefore, we can find a unique orientation-preserving lift $T_t\in\Iso(\E(\kappa,\tau))$ of $h_t$ such that $T_t(\gamma(0))=\gamma(t)$. By the uniqueness of the lift and the uniqueness of a geodesic given initial conditions, it easily follows that $\{T_t\}_{t\in\R}$ is a group satisfying the desired condition.
\end{proof}

The elements of $\{T_t\}_{t\in\R}$ will be called translations along $\gamma$; if $\gamma$ is vertical, then such a group is not unique but translations along $\gamma$ are defined as the vertical translations. In general, a 1-parameter group $\{f_t\}_{t\in\R}$ of isometries is not necessarily the group of translations along a geodesic, but translations will allow us to interpret the rest of them as their screw-motions (i.e., a composition with a group of vertical translations).
\begin{itemize}
    \item If $\tau\neq 0$ and the induced group $\{h_t\}_{t\in\R}$ of isometries of $\M^2(\kappa)$ has a fixed point $p$, then $\{f_t\}_{t\in\R}$ is a group of rotations about the vertical geodesic $\pi^{-1}(p)$ or their screw motions. If $\{h_t\}_{t\in\R}$ has no fixed points, then $\kappa\leq 0$ and we have two possibilities:
    \begin{itemize}
        \item If $\{h_t\}_{t\in\R}$ consists of translations along a geodesic $\alpha$ of $\R^2$ or $\h^2(\kappa)$, then $\{f_t\}_{t\in\R}$ is a group of (hyperbolic) translations along a horizontal geodesic of $\E(\kappa,\tau)$ projecting to $\alpha$ or their screw motions.
        \item The second possibility only occurs if $\kappa<0$ and $\{h_t\}_{t\in\R}$ are parabolic translations. In this case, we can choose a oblique geodesic $\gamma$ projecting to an orbit of $\{h_t\}_{t\in\R}$, which is a horocycle of $\h^2(\kappa)$. Then, $\{f_t\}_{t\in\R}$ is a group of (parabolic) translations along $\gamma$ or their screw motions.
    \end{itemize}

    \item If $\tau=0$, then 1-parameter groups of isometries of $\E(\kappa,\tau)$ factor as the product of 1-parameter groups of isometries of $\mathbb{M}^2(\kappa)$ and $\mathbb{R}$. This makes it easier to classify such a 1-parameter group as rotations, translations, parabolic translations or hyperbolic translations (when they leave the factor $\R$ fixed) and their screw motions (otherwise).
\end{itemize}

If $\tau=0$, then there are also mirror symmetries with respect to vertical planes (i.e., preimages of geodesics of $\M^2(\kappa)$ under $\pi$) or with respect to horizontal slices (i.e., surfaces of the form $\mathbb{M}^2(\kappa)\times\{t_0\}$ for some $t_0\in\R$). These symmetries are orientation-reversing lifts of axial symmetries and of the identity in $\M^2(\kappa)$, respectively. In particular, vertical planes and horizontal slices are totally geodesic surfaces in $\E(\kappa,0)$, yet there are no totally umbilical surfaces in $\E(\kappa,\tau)$ if $\tau\neq 0$, as shown by Souam and Toubiana~\cite{ST}.

\subsection{Working in coordinates}\label{subsec:working-coordinates}
The following common framework for all $\E(\kappa,\tau)$-spaces was originally introduced by Cartan~\cite[\S296]{Cartan}. We consider
\[\Omega_\kappa=\{(x,y)\in\R^2:\lambda_\kappa(x,y)>0\}, \text{ where } \lambda_\kappa(x,y)=\frac{1}{1+\tfrac{\kappa}{4}(x^2+y^2)}. \]
Therefore, $\Omega_\kappa$ is a disk of radius $\tfrac{2}{\sqrt{-\kappa}}$ if $\kappa<0$ or the whole $\R^2$ if $\kappa\geq 0$. Then, we define $M(\kappa,\tau)$ as $\Omega_\kappa\times\mathbb{R}\subseteq\R^3$ endowed with the Riemannian metric
\[\mathrm{d} s^2=\lambda_\kappa^2(\mathrm{d} x^2+\mathrm{d} y^2)+(\mathrm{d} z+\lambda_\kappa\tau(y\mathrm{d} x-x\mathrm{d} y))^2.\]
Moreover, the orientation is chosen such that
\begin{equation}\label{eqn:standard-frame}
\begin{aligned}
 E_1&=\tfrac{1}{\lambda_\kappa}\partial_x-\tau y\,\partial_z,& 
 E_2&=\tfrac{1}{\lambda_\kappa}\partial_y+\tau x\,\partial_z,&
 E_3&=\partial_z
\end{aligned}\end{equation}
is a global positively oriented orthonormal frame. It easily follows that $\pi(x,y,z)=(x,y)$ is a Killing submersion from $M(\kappa,\tau)$ to $(\Omega_{\kappa},\df s_\kappa^2)$ with constant bundle curvature $\tau$ and unit Killing vector field $\xi=E_3$, where the Riemannian metric $\df s_\kappa^2=\lambda_\kappa^2(\df x^2+\df y^2)$ has constant curvature $\kappa$. Therefore, $M(\kappa,\tau)$ is a global model of $\E(\kappa,\tau)$ if $\kappa\leq 0$, but fails to be complete otherwise. More precisely, if $\kappa>0$, then $M(\kappa,\tau)$ is isometric to the universal cover of $\mathbb{E}(\kappa,\tau)$ minus a vertical fiber, as we will discuss shortly.

In the frame~\eqref{eqn:standard-frame}, the Levi-Civita connection $\overline\nabla$ of $\E(\kappa,\tau)$ reads
\begin{equation}\label{eqn:levi-civita}
\begin{aligned}
\overline\nabla_{E_1}E_1&=\tfrac{\kappa y}{2}E_2,&\overline\nabla_{E_1}E_2&=-\tfrac{\kappa y}{2}E_1+\tau E_3,&\overline\nabla_{E_1}E_3&=-\tau E_2,\\
\overline\nabla_{E_2}E_1&=-\tfrac{\kappa x}{2}E_2-\tau E_3,&\overline\nabla_{E_2}E_2&=\tfrac{\kappa x}{2}E_1,&\overline\nabla_{E_2}E_3&=\tau E_1,\\
\overline\nabla_{E_3}E_1&=-\tau E_2,&\overline\nabla_{E_3}E_2&=\tau E_1,&\overline\nabla_{E_3}E_3&=0.
\end{aligned}
\end{equation}
By evaluating at the frame, it is not difficult to check that the Riemann curvature tensor $\overline R$ of $\E(\kappa,\tau)$ is given by
\begin{equation}\label{eqn:curvature-tensor}
\begin{aligned}
\overline R(X,Y,Z,W)&=\langle\overline\nabla_X\overline\nabla_YZ-\overline\nabla_X\overline\nabla_YZ-\overline\nabla_{[X,Y]}Z,W\rangle\\
&=-\tau^2\langle X\times Y,Z\times W\rangle-(\kappa-4\tau^2)\langle X\times Y,\xi\rangle\langle Z\times W,\xi\rangle.
\end{aligned}\end{equation}
Note that our sign convention for $\overline R$ is the opposite to Daniel's in~\cite[Prop.~2.1]{Dan}. Equation~\eqref{eqn:curvature-tensor} implies that the sectional curvature in $\E(\kappa,\tau)$ equals $\tau^2$ for vertical planes and $\kappa-3\tau^2$ for horizontal planes. It follows that the sectional curvature is constant if and only if $\kappa-4\tau^2=0$ as already discussed. On the other hand, it also follows that $\E(\kappa,\tau)$ has constant scalar curvature $2(\kappa-\tau^2)$.

\begin{remark}
Using the model $M(\kappa,\tau)$ we can understand why the geometry of $\E(\kappa,\tau)$ twists in the presence of bundle curvature, which leads to a very different behaviour with respect to product spaces ($\tau=0$). Assume that $\alpha:[0,\ell]\to\Omega_\kappa$ is a piecewise-$\mathcal{C}^1$ Jordan curve enclosing a relatively compact region $U$ with the interior of $U$ to the left when traveling on $\alpha$. Let $\overline\alpha:[0,\ell]\to M(\kappa,\tau)$ be a horizontal lift of $\alpha$ (i.e., $\pi\circ\overline\alpha=\alpha$ and $\overline\alpha$ is everywhere orthogonal to $\xi$), which is unique up to vertical translations. The signed vertical distance from $\overline\alpha(0)$ to $\overline\alpha(\ell)$ is given by $2\tau\mathrm{Area}(U)$, see~\cite[Prop.~1.6.2]{KIAS} and~\cite[Prop.~3.3]{Man14}.
\end{remark}

\subsubsection{A global model for Berger spheres}
If $\kappa>0$ and $\tau \neq 0$, then $\E(\kappa, \tau)$ or $\Sb(\kappa, \tau)$ is a Berger sphere modelled in complex coordinates as the usual $3$-sphere $\s^3 = \{(z,w)\in \C^2:\, |z|^2 +|w|^2 = 1\}$ equipped with the Riemannian metric
\[
  \df s^2(X, Y) = \tfrac{4}{\kappa}\left[\prodesc{X}{Y} + \tfrac{16\tau^2}{\kappa^2}\bigl(\tfrac{4\tau^2}{\kappa} - 1\bigl)\prodesc{X}{\xi}\prodesc{Y}{\xi} \right],
\] 
being $\prodesc{\cdot}{\cdot}$ the usual inner product in $\mathbb{C}^2\equiv\mathbb{R}^4$. The vector field $\xi$ is defined by $\xi_{(z, w)} = \frac{\kappa}{4\tau}(iz, iw)$ and the Killing submersion is the Hopf fibration 
\[\pi: \Sb(\kappa, \tau) \rightarrow \s^2(\kappa) \subset \mathbb{C} \times  \mathbb{R} \equiv\mathbb{R}^3,\qquad \pi(z, w) = \tfrac{2}{\sqrt{\kappa}}\bigl(z\bar{w}, \tfrac{1}{2}(|z|^2 - |w|^2) \bigr),\]
see~\cite[\S2]{Tor12}. A local isometry between $\Sb(\kappa,\tau)$ and $M(\kappa,\tau)$ is given by the Riemannian covering map $\Theta:M(\kappa,\tau)\to\Sb(\kappa,\tau)-\{(e^{i \theta},0)\colon  \theta \in \mathbb{R}\}$, where
\begin{equation}\label{eq:local-isometry-Daniel-Berger}
\begin{aligned}
\Theta(x, y, z) &= \frac{1}{\sqrt{1 + \tfrac{\kappa}{4}(x^2 + y^2))}} \left(\tfrac{\sqrt{\kappa}}{2}(y + ix) \exp(i \tfrac{\kappa}{4\tau}z), \exp(i \tfrac{\kappa}{4\tau} z)\right).
\end{aligned}
\end{equation}

\subsubsection{The half-space model}\label{subsubsec:halfspace-model}
In the case $\kappa<0$, another model of $\E(\kappa,\tau)$ is the half-space model given by $\{(x,y,z)\in\R^2:y>0\}$ endowed with the metric
\begin{equation}\label{eqn:halfspace-metric}\frac{\df x^2+\df y^2}{-\kappa y^2}+\left(\df z+\frac{2\tau}{\kappa y}\df x\right)^2.
\end{equation}
Notice that the conformal factor $\frac{1}{y\sqrt{-\kappa}}$ defines a metric of constant curvature $\kappa$ in the upper half-plane, and the orientation is chosen such that
\begin{equation}\label{eqn:halfspace-frame}
\begin{aligned}
  E_1&=y\sqrt{-\kappa}\,\partial_x+\tfrac{2\tau}{\sqrt{-\kappa}}\partial_z,&
  E_2&=y\sqrt{-\kappa}\,\partial_y,&
  E_3&=\partial_z,
\end{aligned}\end{equation}
is a positively oriented orthonormal frame. The Killing submersion is again $\pi(x,y,z)=(x,z)$ with unit Killing vector field $\xi=E_3$. A global isometry from $M(\kappa,\tau)$ to the half-space model is given by
\begin{align*}
\Theta(x,y,z)=\left(\tfrac{\frac{4}{\sqrt{-\kappa }} y}{\bigl(\frac{2}{\sqrt{-\kappa }}+x\bigr)^2+y^2},\tfrac{-\frac{4}{\kappa
   }-x^2-y^2}{\bigl(\frac{2}{\sqrt{-\kappa }}+x\bigr)^2+y^2},z+\tfrac{4 \tau}{\kappa}\arccos\!\left(\tfrac{y}{\sqrt{\bigr(\frac{2}{\sqrt{-\kappa
   }}+x\bigl)^2+y^2}}\right)\!\!\right).
\end{align*}
The following expression for the Levi-Civita connection in the global frame~\eqref{eqn:halfspace-frame} can be deduced directly from~\cite[Eq.~(5--1)]{Man14} for $\lambda=\frac{1}{y\sqrt{-\kappa}}$:
\begin{equation}\label{eqn:levi-civita-halfspace}
\begin{aligned}
\overline\nabla_{E_1}E_1&=\sqrt{-\kappa}E_2,&\overline\nabla_{E_1}E_2&=-\sqrt{-\kappa}E_1+\tau E_3,&\overline\nabla_{E_1}E_3&=-\tau E_2,\\
\overline\nabla_{E_2}E_1&=-\tau E_3,&\overline\nabla_{E_2}E_2&=0,&\overline\nabla_{E_2}E_3&=\tau E_1,\\
\overline\nabla_{E_3}E_1&=-\tau E_2,&\overline\nabla_{E_3}E_2&=\tau E_1,&\overline\nabla_{E_3}E_3&=0.
\end{aligned}
\end{equation}

\subsection{Fundamental data}\label{subsec:fundamental-data}

Let $\phi:\Sigma\to\E(\kappa,\tau)$ be an isometric immersion of an orientable Riemannian surface $\Sigma$ with global smooth unit normal $N$. The shape operator of $\phi$ with respect to $N$ can be seen as a symmetric $(1,1)$-tensor $A$ identified with the smooth field of self-adjoint linear operators
\[A_p:T_p\Sigma\to T_p\Sigma,\qquad A_p(v) = -\overline\nabla_{\df\phi_p(v)}N_p,\quad \text{for all } v\in T_p\Sigma.\]
We can also consider the angle function $\nu\in\mathcal C^\infty(\Sigma)$ and the tangent part of the Killing field $T\in\X(\Sigma)$ defined by $\nu(p)=\langle\xi_{\phi(p)},N_p\rangle$ and $\df\phi_p(T_p)=\xi_{\phi(p)}-\nu(p)N_p$, respectively, for all $p\in\Sigma$. The orientation in $\E(\kappa,\tau)$ and the choice of $N$ induce an orientation in $\Sigma$ expressed in terms of a $\frac{\pi}{2}$-rotation $J$ in the tangent bundle (that is, $J$ is a smooth field of linear operators such that $J^2=-\mathrm{id}$) defined by assuming that $\{\df\phi_p(u),\df\phi_p(Ju),N_p\}$ is positively oriented, or equivalently
\begin{equation}\label{eqn:orientation-J}
\df\phi_p(Ju)=N_p\times\df\phi_p(u),
\end{equation}
for all non-zero $u\in T_p\Sigma$ and $p\in\Sigma$.

Daniel~\cite{Dan} showed that the quadruplet $(A,T,J,\nu)$, also called the \emph{fundamental data} of the immersion, satisfies the following equations for all $X,Y\in\X(\Sigma)$:
\begin{enumerate}[label=(\roman*)]
  \item $K=\det(A)+\tau^2+(\kappa-4\tau^2)\nu^2$,
  \item $\nabla_XAY-\nabla_YAX-A[X,Y]=(\kappa-4\tau^2)(\langle Y,T\rangle X-\langle X,T\rangle Y)\nu$,
  \item $\nabla_X T=(AX-\tau JX)\nu$
  \item $\nabla\nu=-AT-\tau JT$,
  \item $\|T\|^2+\nu^2=1$,
\end{enumerate}
where $K$ is the Gau\ss\ curvature of $\Sigma$ and $\nabla$ is its Levi-Civita connection. The identities (i) and (ii) are nothing but the Gau\ss\ and Codazzi equations. Observe that we have considered $J$ as a fundamental datum to state explicitly that the orientation plays a key role; equivalently, Daniel assumes that $\Sigma$ is oriented. 

Conversely, any simply connected Riemannian surface $\Sigma$ carrying a quadruplet $(A,T,J,\nu)$ that satisfies the above conditions (i)--(v) can be isometrically immersed in $\E(\kappa,\tau)$ with shape operator $A$, tangent part of the Killing $T$, and angle function $\nu$, being $A$ and $\nu$ defined with respect to the normal $N$ compatible with $J$, in the sense that~\eqref{eqn:orientation-J} holds true. Such an immersion is determined up orientation-preserving isometries that also preserve the orientation of vertical fibers, see~\cite[Thm.~4.3]{Dan}. A similar discussion of the fundamental equations in terms of a conformal parameter on $\Sigma$ is given in~\cite{FM07}.

\begin{remark}\label{rmk:data-transformations}
To understand the uniqueness, it is important to mention how other geometric transformations affect the fundamental data:
\begin{itemize}
    \item A change of the sign of $N$ results in $(A,T,J,\nu)\mapsto (-A,T,-J,-\nu)$.
    \item A composition with an orientation-preserving isometry that reverses the orientation of the fibers gives $(A,T,J,\nu)\mapsto (A,-T,J,-\nu)$.
    \item If $\tau=0$, a composition with an orientation-reversing isometry that preserves the orientation of the fibers gives $(A,T,J,\nu)\mapsto (A,T,-J,\nu)$.
\end{itemize}
These transformations are also discussed in~\cite[Rmk.~4.12]{Dan} and~\cite[Rmk.~3.4]{GMM}.\end{remark}

\subsection{Cylinders and multigraphs}\label{subsec:multigraphs}
A vertical cylinder in $\E(\kappa,\tau)$ is a surface $\Sigma$ invariant under vertical translations, so it is foliated by vertical geodesics and its angle function vanishes identically. This amounts to say that $\Sigma$ is the preimage $\pi^{-1}(\beta)$ of a curve $\beta\subset\M^2(\kappa)$, from where it easily follows that the mean curvature of $\Sigma$ is half of the curvature of $\beta$ as a curve of $\M^2(\kappa)$. In particular, we define the $H$-cylinders as the preimages under $\pi$ of complete curves of constant curvature $2H$ in $\M^2(\kappa)$. This implies a different behavior if they project onto circles ($4H^2+\kappa>0$), straight lines ($H=\kappa=0$), horocycles ($4H^2+\kappa=0$ and $\kappa<0$) or curves of $\h^2(\kappa)$ equidistant to a geodesic ($4H^2+\kappa<0$).

On the opposite side of vertical cylinders, another distinguished class of surfaces in $\E(\kappa,\tau)$ are the so-called \emph{vertical multigraphs}, which are surfaces everywhere transverse to $\xi$. The angle function of a multigraph $\Sigma$ never vanishes, whence the projection $\pi_{|\Sigma}$ is a local diffeomorphism. If $\pi_{|\Sigma}$ is also one-to-one, the surface is called a \emph{vertical graph} over the domain $\pi(\Sigma)\subseteq\M^2(\kappa)$. Such a graph is said to be \emph{entire} if $\pi_{|\Sigma}$ is a global diffeomorphism onto $\M^2(\kappa)$. 

Smooth graphs over $\Omega\subseteq\M^2(\kappa)$ can be parametrized by smooth functions $u\in\mathcal C^\infty(\Omega)$ as $F_u:\Omega\to\E(\kappa,\tau)$ given by $F_u(q)=\Phi_{u(q)}{F_0(q)}$ for all $q\in\Omega$, where $F_0:\Omega\to\E(\kappa,\tau)$ is a zero section and $\{\Phi_t\}_{t\in\R}$ is the group of vertical translations. If $\tau=0$, then $\M^2(\kappa)\times\{0\}$ is the natural zero section, but it is important to remark that there is no natural zero section whenever $\tau\neq0$. Working in the model $M(\kappa,\tau)$, it is common to consider $F_0(x,y)=(x,y,0)$ as zero section and then parametrize any graph over $\Omega\subseteq\Omega_\kappa$ as $F_u(x,y)=(x,y,u(x,y))$. The mean curvature of this parametrization is given in divergence form as
\begin{equation}\label{eqn:H-graphs}
H=\frac12\Div\left(\frac{Gu}{\sqrt{1+\|Gu\|^2}}\right),
\end{equation}
where the divergence and norm are computed in the geometry of $\M^2(\kappa)$. The vector field $Gu\in\X(\Omega)$, given by $Gu=(u_x+\tau\lambda_\kappa y)\partial_x+(u_y-\tau\lambda_\kappa x)\partial_y$, is also known as the \emph{generalized gradient} of $u$. It does not depend on the choice of the zero section, and $\sqrt{1+\|Gu\|^2}$ is the area element of the surface, see~\cite[\S3]{LerMan}.

Any complete $H$-multigraph in $\E(\kappa,\tau)$ is either a horizontal slice in $\s^2(\kappa)\times\R$ (with $H=0$) or a graph (with $4H^2+\kappa\leq 0$) over a simply connected domain of $\mathbb{M}^2(\kappa)$ bounded by curves of constant curvature $\pm 2H$ where the function defining the graph tends to $\pm\infty$, see~\cite{ManRod}. Such a complete $H$-graph must be entire if $\kappa+4H^2=0$, see~\cite[Cor.~4.6.3]{KIAS} and the references therein. Note that, by the Fernández and Mira's solution to the Bernstein problem~\cite{FM} (see also~\cite{Man19}), there are plenty of entire graphs with critical mean curvature, generically a two-parameter family of them for each choice of a holomorphic quadratic differential on the complex plane $\C$ (not identically zero) or on the unit disk $\mathbb{D}\subseteq\C$.

If $\kappa+4H^2\leq 0$, by a standard application of the maximum principle for $H$-surfaces, the existence of entire $H$-graphs prevents the existence of compact immersed $H$-surfaces. However, if $\kappa+4H^2>0$, then there do exist compact $H$-surfaces in $\E(\kappa,\tau)$, e.g., the rotationally invariant $H$-spheres. As a matter of fact, Abresch and Rosenberg~\cite{AR} solved the Hopf problem in $\E(\kappa,\tau)$ by showing that these are the only immersed $H$-spheres in $\E(\kappa,\tau)$. Note that some of these $H$-spheres are non-embedded~\cite[Thm.~1]{Tor10}. We will show in \S\ref{sec:compact} that there are plenty of compact embedded $H$-surfaces in $\mathbb S^2(\kappa)\times\R$ which are not equivariant. It is important to remark that the classical Alexandrov problem concerning the classification of compact embedded $H$-surfaces has not been hitherto settled in $\E(\kappa,\tau)$-spaces other than $\mathbb{H}^2(\kappa)\times\R$ (see~\cite{HH1989}) and $\R^3$, where Alexandrov's moving-plane technique applies. In $\mathbb{S}^2(\kappa)\times \mathbb{R}$ and $\mathbb{S}^3_B(\kappa,\tau)$ there do exist non-spherical $H$-surfaces, e.g., the rotational embedded $H$-tori, see~\cite{PR, Tor10}, or the non-equivariant examples constructed in~\S\ref{sec:delaunay} and~\S\ref{sec:genus}.

All the above discussions reveal that the value of $H$, if any, such that $4H^2+\kappa=0$ is geometrically relevant, and it is called the \emph{critical mean curvature} in $\E(\kappa,\tau)$. Accordingly, $H$-surfaces with $4H^2+\kappa>0$ and $4H^2+\kappa<0$ will be said to have \emph{supercritical} and \emph{subcritical} mean curvature, respectively.


\section{The conjugate construction}\label{sec:conjugation}

Let $\kappa,\tau,H,\widetilde\kappa,\widetilde\tau,\widetilde H\in\mathbb{R}$ be constants such that $\kappa-4\tau^2=\widetilde\kappa-4\widetilde\tau^2$ and $\tau+iH=e^{i\theta}(\widetilde\tau+i\widetilde H)$ for some $\theta\in[0,2\pi)$. Given a simply-connected Riemannian surface $\Sigma$, there is an isometric correspondence between $\widetilde H$-immersions $\widetilde\phi:\Sigma\to\mathbb{E}(\widetilde\kappa,\widetilde\tau)$ and $H$-immersions $\phi:\Sigma\to\mathbb{E}(\kappa,\tau)$ by means of the following transformation of fundamental data:
\begin{equation}\label{eqn:fundamental-rotation}
(A, T, J,\nu)=\bigl(\Rot_\theta\circ(\widetilde A-\widetilde H\,\id)+H\,\id,\Rot_\theta(\widetilde T),\widetilde J,\widetilde\nu\bigr),
\end{equation}
where $\Rot_\theta=\cos(\theta)\id+\sin(\theta)J$ is a rotation of angle $\theta$ in the tangent bundle of $\Sigma$. The immersions $\widetilde\phi$ and $\phi$ are called \emph{sister immersions}, and determine each other up to orientation-preserving isometries that also preserve the orientation of the fibers, see~\cite[Prop.~4.1]{Dan}. We would like to remark that the simple connectedness of $\Sigma$ is not an essential assumption here, for the correspondence can be applied after considering the Riemannian universal cover of $\Sigma$.

\begin{remark}[Lawson correspondence]
If $\kappa-4\tau^2=0$, then Daniel correspondence reduces to a correspondence between $\widetilde H$-immersions in $\M^3(\widetilde\tau^2)$ and $H$-immersions in $\M^3(\tau^2)$ such that $\widetilde H^2+\widetilde\tau^2=H^2+\tau^2$ in which the shape operator is rotated by an angle $\theta$. This is a particular case of the 2-parameter correspondence given by Lawson~\cite[Thm.~8]{Law}, see also~\cite{K89a,G} for a more geometric description. If $\widetilde\kappa=\widetilde\tau=\widetilde H=0$, then $\kappa=\tau=H=0$ and Daniel correspondence reduces to the classical Bonnet associate family of minimal surfaces in $\R^3$.
\end{remark}

Using the uniqueness, we can prove a result similar to~\cite[Prop.~2.12]{Law}:

\begin{proposition}\label{prop:equivariant}
Let $\widetilde\phi:\Sigma\to\mathbb{E}(\widetilde\kappa,\widetilde\tau)$ and $\phi:\Sigma\to\mathbb{E}(\kappa,\tau)$ be sister immersions. Then $\phi$ is equivariant if and only if $\widetilde\phi$ is equivariant.
\end{proposition}

\begin{proof}
Assume that $\{S_t\}_{t\in\R}$ is a 1-parameter group of isometries of $\E(\kappa,\tau)$ under which $\phi$ is invariant. Since this is a continuous group, it must preserve both the orientation and the orientation of the fibers, and induces a 1-parameter group $\{R_t\}_{t\in\R}$ of isometries of $\Sigma$ such that $S_t\circ\phi=\phi\circ R_t$. The fundamental data of the immersions $\phi\circ R_t:\Sigma\to\E(\kappa,\tau)$ do not depend on $t$, whence their sister immersions have the same fundamental data as $\widetilde\phi\circ R_t:\Sigma\to\E(\widetilde\kappa,\widetilde\tau)$ (which does not depend on $t$ either) in view of~\eqref{eqn:fundamental-rotation}. By the uniqueness of the correspondence, we can find isometries $\widetilde S_t$ of $\E(\widetilde\kappa,\widetilde\tau)$ preserving the orientation and the orientation of the fibers such that $\widetilde S_t\circ\widetilde\phi=\widetilde S_t\circ\widetilde\phi\circ R_0=\widetilde\phi\circ R_t$ for all $t\in\R$. It follows that
\[\widetilde S_{s+t}\circ\widetilde\phi=\widetilde\phi\circ R_{s+t}=\widetilde\phi\circ R_{s}\circ R_t=\widetilde S_s\circ\widetilde\phi\circ R_t=\widetilde S_s\circ\widetilde S_t\circ\widetilde\phi.\]
This implies that $\widetilde S_{s+t}=\widetilde S_s\circ\widetilde S_t$ unless $\widetilde\phi$ is part of vertical cylinder or a horizontal slice. However, were it the case, then $\nu\equiv 0$ or $\nu\equiv\pm1$ and it follows that $\phi$ is also part of a vertical cylinder or a horizontal slice, so the statement holds true.
\end{proof}

The continuity of the sister correspondence also follows from its uniqueness. The proof is a direct generalization of the particular case in~\cite[Prop.~2.3]{CMR}.

\begin{proposition}[continuity]\label{prop:continuity}
Let $\Sigma$ be a smooth surface and let $\widetilde\phi_n:(\Sigma,\mathrm{d} s^2_n)\to\mathbb{E}(\widetilde\kappa,\widetilde\tau)$ be a sequence of isometric $\widetilde H$-immersions that converge on the $\mathcal{C}^m$-topology on compact subsets (for all $m\geq 0$) to an isometric $\widetilde H$-immersion $\widetilde\phi_\infty:(\Sigma,\mathrm{d} s_\infty^2)\to\mathbb{E}(\widetilde\kappa,\widetilde\tau)$. Given $\theta\in\R$ not depending on $n$, the sister $H$-immersions $\phi_n$ (up to suitable isometries of $\mathbb{E}(\kappa,\tau)$) converge in the same mode to the sister $H$-immersion $\phi_\infty$.
\end{proposition}

Observe that $\E(\kappa,\tau)$ has bounded geometry in view of~\eqref{eqn:curvature-tensor}. If a sequence of $H$-surfaces in $\E(\kappa,\tau)$ has uniformly bounded second fundamental form in a neighborhood (of uniform intrinsic radius) around an accumulation point, then there is a subsequence that converges in the topology $\mathcal C^m$ for all $m$ to some $H$-surface. The condition on the accumulation point amounts to translate the surfaces appropriately using the homogeneity, whereas the uniform bound of the second fundamental form usually follows from stability, as shown by Rosenberg, Souam and Toubiana~\cite{RST}. If the surfaces in the sequence are complete and stable, then the limit $H$-surface is also complete.

\begin{remark}[rescaling]\label{rmk:rescaling}
Consider homotheties $\widetilde \rho_\mu:\E(\widetilde\kappa,\widetilde\tau)\to \E(\mu^2\widetilde\kappa,\mu\widetilde\tau)$ and $\rho_\mu:\E(\kappa,\tau)\to \E(\mu^2\kappa,\mu\tau)$ that multiply lengths by a constant factor $\mu^{-1}$. In the Cartan model, these transformations are nothing but $(x,y,z)\mapsto(\frac{x}{\mu},\frac{y}{\mu},\frac{z}{\mu^2})$.

Given an isometric $\widetilde H$-immersion $\widetilde\phi:(\Sigma,\df s_\Sigma^2)\to \E(\widetilde\kappa,\widetilde\tau)$ with fundamental data $(\widetilde A,\widetilde T,J,\nu)$, we have that $\rho_\mu\circ\widetilde\phi:(\Sigma,\mu^{-2}\df s_\Sigma^2)\to \E(\mu^2\widetilde\kappa,\mu\widetilde\tau)$ is an isometric $(\mu\widetilde H)$-immersion with fundamental data $(\mu \widetilde A,\mu \widetilde T,J,\nu)$. Given a sister $H$-immersion $\phi:\Sigma\to\E(\kappa,\tau)$ for some phase angle $\theta\in[0,2\pi)$ and reasoning likewise, the fundamental data of $\rho_\mu\circ\phi$ is $(\mu A,\mu T,J,\nu)$. By looking at~\eqref{eqn:fundamental-rotation}, we deduce that $\widetilde\rho_\mu\circ\widetilde\phi$ is the sister of $\rho_\mu\circ\phi$ for the same value of $\theta$, that is, the sister correspondence commutes with rescaling the metrics.
\end{remark}

Other geometric objects in the theory of $H$-surfaces in $\E(\kappa,\tau)$ also behave nicely with respect to the correspondence. For instance, at the conformal level, the Abresch--Rosenberg holomorphic quadratic differentials~\cite{AR} of sister immersions are related by $\widetilde Q=e^{-2i\theta}Q$. Also, the harmonic Gau\ss\ maps into hyperbolic space spanning these differentials in the case of multigraphs of critical mean curvature are associate for sister immersions, see~\cite[Prop.~5.6]{DFM}. We also point out that sister immersions define the same stability operator given by
\begin{equation}\label{eqn:stability-op}
L=\Delta-2K+4H^2+\kappa+(\kappa-4\tau^2)\nu^2,
\end{equation}
where $\Delta$ stands for the Laplacian on $\Sigma$, see~\cite[Prop.~5.11]{Dan}.

As for the construction of $H$-surfaces, we will focus on the case $\theta=\frac{\pi}{2}$ and $\widetilde H=0$, which gives $\widetilde\kappa=4H^2+\kappa$ and $\widetilde\tau=H$, i.e., we can associate an $H$-immersion $\phi:\Sigma\to\M^2(\kappa)\times\R$ to any minimal immersion $\widetilde\phi:\Sigma\to\E(4H^2+\kappa,H)$. Under these assumptions, we will call $\phi$ and $\widetilde\phi$ \emph{conjugate} immersions in the sequel. They fulfill some additional properties that will allow us to control the geometry of $\phi$ in terms of the geometry of $\widetilde\phi$, and will be discussed in \S\ref{sec:conjugate-curves}. Recall that we will use the tilde notation for the minimal surface (the minimal surface we will begin with), because $\phi$ (without tilde) will stand for our target surface in a product space. Conjugation is characterized in terms of the fundamental data $(A,T,J,\nu)=(J\circ\widetilde A+H\,\id,J\widetilde T,J,\widetilde\nu)$, plus the orientations satisfy the compatibility relations $\df\phi_p(Ju)=N\times\df\phi_p(u)=N^*\times\df\phi^*_p(u)$ for all $u\in T_p\Sigma$.

The different possible configurations are shown in Table~\ref{tab:conjugation-cases}. That table explains why $H$-surfaces with critical, supercritical and subcritical mean curvature in $\M^2(\kappa)\times\R$ display a qualitatively different behaviour, as discussed in \S\ref{subsec:multigraphs}. Their minimal isometric conjugate surfaces lie in Berger spheres (supercritical case), in the Heisenberg group (critical case), in $\SL$ (subcritical case), or in $\h^2\times\R$ and $\s^2\times\R$ (minimal case). In this last case, there is a notion of associate family of minimal immersions, discovered by Hauswirth, Sa Earp and Toubiana~\cite[Cor.~10]{HST}. They found that conjugate surfaces come from certain conjugate harmonic maps as in the aforesaid case of critical mean curvature. This justifies the use of the term \emph{conjugation} in the general context.

\begin{remark}
 We could also have chosen $\theta=-\frac{\pi}{2}$ but it would lead to an isometric surface in $\E(\kappa,-\tau)$. Note that composition with an isometry from $\E(\kappa,\tau)$ to $\E(\kappa,-\tau)$ that preserves the orientation of the fibers produces the change of fundamental data $(A, T, J,\nu)\mapsto(-A,-T, J,\nu)$, and also the change of parameters $(\kappa,\tau,H)\mapsto(\kappa,-\tau,-H)$. This is equivalent to adding $\pm\pi$ to $\theta$ in~\eqref{eqn:fundamental-rotation}. 
\end{remark} 

\begin{table}[hbtp]
\begin{center}
\begin{tabular}{c|ccc}
\toprule
 \multirow{2}{*}{A minimal surface in}   &   \multicolumn{3}{c}{produces an $H$-surface in} \\\cmidrule{2-4}
&  $\mathbb{S}^2(\kappa)\times \mathbb{R}$ &   $\mathbb{H}^2(\kappa)\times\mathbb{R}$ & $\mathbb{R}^3$ \\\midrule
$\Sb(4H^2+\kappa, H)$ &   $H > 0$    &   $H > \sqrt{-\kappa}/2$ &   $H>0$\\
$\text{Nil}_3$    &   --- &   $H = \sqrt{-\kappa}/2$ &   ---\\
$\SL(4H^2+\kappa, H)$  &   --- &   $0 < H < \sqrt{-\kappa}/2$ &   ---\\
$\mathbb{H}^2(\kappa)\times\mathbb{R}$ &  --- &   $H=0$ &   ---\\
$\mathbb{S}^2(\kappa)\times \mathbb{R}$ & $H=0$ &   --- &   ---\\
$\mathbb{R}^3$ &   --- &   --- &  $H=0$\\
\bottomrule
\end{tabular}\end{center}
\caption{Possible configurations for conjugate surfaces in product spaces.}\label{tab:conjugation-cases}
\end{table}

\subsection{Conjugate curves}\label{sec:conjugate-curves}

Let $\widetilde\phi:\Sigma\to\E(4H^2+\kappa,H)$ be a minimal immersion and let $\phi:\Sigma\to\M^2(\kappa)\times\R$ be its conjugate $H$-immersion. We recall that the shape operators $A$ and $\widetilde{A}$ of $\phi$ and $\widetilde{\phi}$ are related by $A = J\widetilde{A} + H\,\mathrm{id}$ (see \eqref{eqn:fundamental-rotation}).

Given $\alpha:[0,\ell]\to\Sigma$ a unit-speed regular curve, we shall investigate the relation between conjugate curves $\widetilde\gamma=\widetilde\phi\circ\alpha$ and $\gamma=\phi\circ\alpha$.

\subsubsection{Vertical and horizontal geodesics}
We will begin by collecting some results to understand the geometry of $\gamma$ whenever $\widetilde\gamma$ is a vertical or a horizontal geodesic. The proof of the following two lemmas, sometimes in particular cases, is scattered across the references~\cite{Tor12,Ple14,MT14,MPT,CM,CMR}.

\begin{lemma}[Horizontal geodesics]\label{lem:horizontal-geodesics} 
  If $\widetilde\gamma$ is a horizontal geodesic segment, then $\gamma$ is contained in a totally geodesic vertical plane $P$ that the immersion meets orthogonally .
\begin{enumerate}[label=\emph{(\alph*)}]
     \item If $\gamma=(\beta,z)\in\M^2(\kappa)\times\mathbb{R}$ is decomposed component-wise, then $|z'|=\sqrt{1-\nu^2}$ and $\|\beta'\|=|\nu|$. In particular, $z$ and $\beta$ might fail to be locally one-to-one only around points where $\nu=\pm1$ and $\nu=0$, respectively (see Figure~\ref{fig:orientation-hor}).
     \item Write $\widetilde N_{\gamma(t)}=\cos(\theta(t)) X(t)+\sin(\theta(t))\widetilde\xi_{\widetilde\gamma(t)}$ for some $\theta\in\mathcal C^\infty([0,\ell])$, where $\{\widetilde\gamma',\widetilde\xi,X\}$ is a positively oriented orthonormal frame along $\widetilde\gamma$. The curve $\gamma$ has geodesic curvature $\kappa_g^P=\theta'$ as a curve of $P$ with respect to $N$ as conormal.
 \end{enumerate}
\end{lemma}

\begin{proof}
We will provide the proof of item (b), which is not in the literature. Since $\overline\nabla_{\widetilde\gamma'}\widetilde\xi =H\widetilde\gamma'\times\widetilde\xi= HX$ and $\overline\nabla_{\widetilde\gamma'}X=\overline\nabla_{\widetilde\gamma'}(\widetilde\gamma'\times\widetilde\xi)=-H\widetilde\xi$ (see \eqref{eqn:connection-xi}), we easily compute $\overline\nabla_{\widetilde\gamma'}\widetilde N=(\theta'-H)\widetilde N\times\widetilde\gamma'$. This implies that
\begin{align*}
\theta'-H&=\langle\overline\nabla_{\widetilde\gamma'}\widetilde N,\widetilde N\times\widetilde\gamma'\rangle=-\langle\widetilde A\alpha',J\alpha'\rangle=\langle JA\alpha'-HJ\alpha',J\alpha'\rangle\\
&=\langle A\alpha',\alpha'\rangle-H=-\langle\overline\nabla_{\gamma'}N,\gamma'\rangle-H=\langle\overline\nabla_{\gamma'}\gamma',N\rangle-H=\kappa_g^P-H.\qedhere
\end{align*}
\end{proof}

\begin{figure}[htb]
\begin{center}
\includegraphics[width=0.85\textwidth]{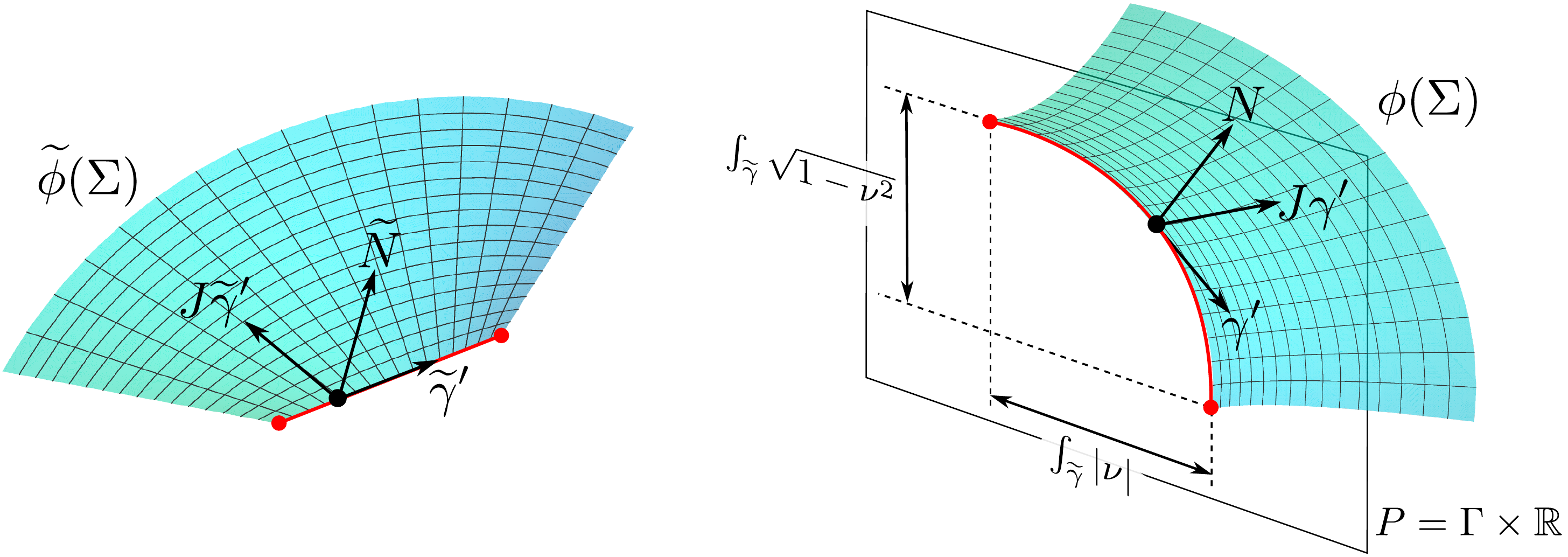}
\caption{Conjugate surfaces $\widetilde\phi(\Sigma)$ and $\phi(\Sigma)$ in a neighborhood of a horizontal geodesic. The quantities $\int_{\widetilde\gamma}|\nu|$ and $\int_{\widetilde\gamma}\sqrt{1-\nu^2}$ indicate the lengths of the projections of $\gamma$ to the factors $\M^2(\kappa)$ and $\R$, respectively (with multiplicity).}\label{fig:orientation-hor} 
\end{center}\end{figure}

\begin{lemma}[Vertical geodesics]\label{lem:vertical-geodesics}
If $\widetilde\gamma$ is a vertical segment, then $\gamma$ is contained in a totally geodesic horizontal slice $P=\M^2(\kappa)\times\{t_0\}$ that the immersion meets orthogonally. 

Assume that $\widetilde\gamma'=\widetilde\xi$ and write $\widetilde N_{\gamma(t)}=\cos(\theta(t)) X_1(t)+\sin(\theta(t))X_2(t)$ for some $\theta\in\mathcal C^\infty([0,\ell])$, where $\{X_1,X_2,\widetilde\xi\}$ is a positively oriented orthonormal frame along $\widetilde\gamma$ such that $X_1$ and $X_2$ project to constant tangent vectors of $\M^2(\kappa)$.
\begin{enumerate}[label=\emph{(\alph*)}]
    \item The curve $\gamma$ has geodesic curvature $\kappa_g^P=2H-\theta'$ as a curve of $P$ with respect to $N$ as conormal.

    \item Assume that $\alpha\subset\partial R$ for some region $R\subseteq\Sigma$ on which $\nu>0$, and let $\widetilde\Omega\subset\M^2(4H^2+\kappa)$ and $\Omega\subset\M^2(\kappa)$ be the (possibly non-embedded) domains over which $\widetilde\phi(R)$ and $\phi(R)$ are multigraphs, respectively.
    \begin{itemize}
        \item If $\theta'>0$, then $J\widetilde\gamma'$ (resp.\ $J\gamma'=\xi$) is a unit outer conormal to $\partial\widetilde\Omega$ (resp.\ $\partial\Omega$) along $\widetilde\gamma$ (resp.\ $\gamma$), $N$ points to the interior of $\Omega$ along $\gamma$, and $\phi(R)$ lies in $\M^2(\kappa)\times(-\infty,t_0]$ locally around $\gamma$ (see Figure~\ref{fig:orientation}, top). 
        
        \item If $\theta'<0$, then $J\widetilde\gamma'$ (resp.\ $J\gamma'=\xi$) is a unit inner conormal to $\partial\widetilde\Omega$ (resp.\ $\Omega$) along $\widetilde\gamma$ (resp.\ $\gamma$), $N$ points to the exterior of $\Omega$ along $\gamma$, and $\phi(R)$ lies in $\M^2(\kappa)\times[t_0,+\infty)$ locally around $\gamma$ (see Figure~\ref{fig:orientation}, bottom).
    \end{itemize}
\end{enumerate}
\end{lemma}

\begin{figure}[hbp]
\begin{center}
\includegraphics[width=0.85\textwidth]{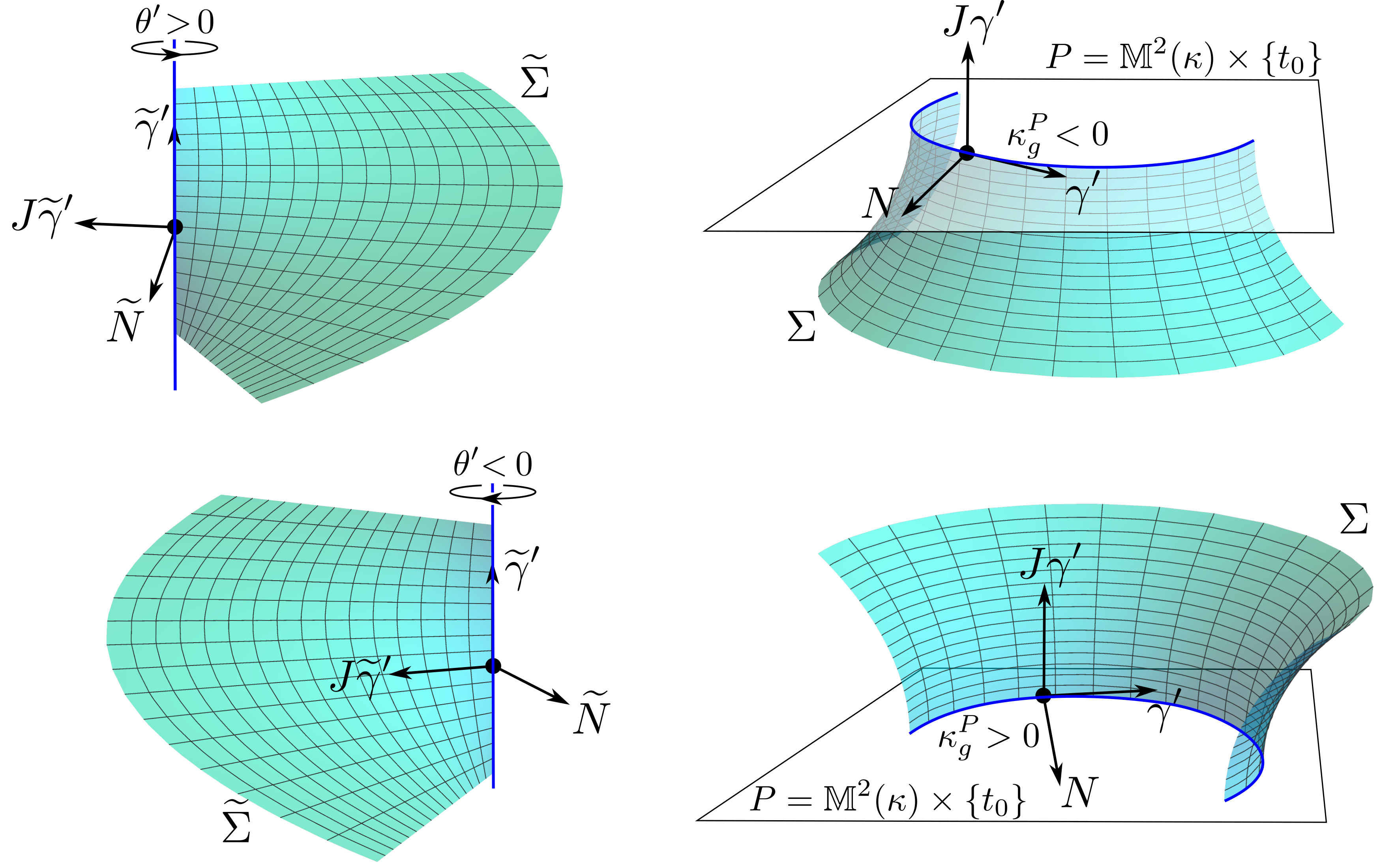}
\caption{Orientation of conjugate surfaces $\widetilde\phi(\Sigma)$ and $\phi(\Sigma)$ according to the direction of rotation of $\widetilde N$ along a vertical geodesic $\widetilde\gamma$ such that $\widetilde\gamma'=\widetilde\xi$. Recall that this is true on the boundary of a region where $\nu>0$.}\label{fig:orientation}
\end{center}\end{figure}

The function $\theta\in\mathcal{C}^{\infty}([0,\ell])$ in Lemmas~\ref{lem:horizontal-geodesics} and~\ref{lem:vertical-geodesics} will be called the \emph{angle of rotation} of $\widetilde N$ along $\widetilde\gamma$. In general, it is impossible to obtain $\theta$ explicitly, but $\theta(\ell)-\theta(0)=\int_0^\ell\theta'(t)\df t$ gives useful information because it is related to $\int_\gamma\kappa_g^P$, the total geodesic curvature of $\gamma$ by means of the formulas in the above statements. It is also important to remark that if $P$ is a vertical plane (Lemma~\ref{lem:horizontal-geodesics}), then $P$ is flat and $\int_\gamma\kappa_g^P$ is the total rotation of the normal of $\gamma$ as a curve of $P$. However, if $P$ is a horizontal slice (Lemma~\ref{lem:vertical-geodesics}), then $P$ has constant curvature $\kappa$, and there is no clear geometric interpretation of $\int_\gamma\kappa_g^P$ if $\kappa\neq 0$. Gau\ss--Bonnet theorem offers some information, as we will discuss in the constructions.

\begin{remark}\label{rmk:horocycle-rotation}
If $\widetilde\gamma$ is vertical, the curvature $\kappa_g^P$ can be related to the rotation angle of $\gamma$ with respect to foliations of $\mathbb{M}^2(\kappa)$ by curves. This idea was devised by Plehnert~\cite[Lem.~4.7]{Ple12} for a foliation of $\mathbb{H}^2$ by horocycles, but alike formulas show up for other foliations of $\mathbb{M}^2(\kappa)$ by curves of constant curvature.

In the halfspace model of $\mathbb{H}^2(\kappa)\times\mathbb{R}$, we can use the frame $\{E_1,E_2,E_3\}$ given by~\eqref{eqn:halfspace-frame} to express
\begin{equation}\label{eqn:rotation-horocycles}
 \gamma'(t)=\cos(\psi(t))E_1+\sin(\psi(t))E_2,
\end{equation}
where we call $\psi\in\mathcal C^\infty([0,\ell])$ the \emph{angle of rotation} of $\gamma$ with respect to a foliation by horocycles (these horocycles are the integral curves of $E_1$). Using~\eqref{eqn:levi-civita-halfspace} and taking derivatives in~\eqref{eqn:rotation-horocycles}, we get $\overline\nabla_{\gamma'}\gamma'=(\psi'+\sqrt{-\kappa}\cos(\psi))(-\sin(\psi)E_1+\cos(\psi)E_2)$. On the other hand, as the surface meets $P$ orthogonally, we infer that $N=\gamma'\times J\gamma'=\gamma'\times E_3=\sin(\psi)E_1-\cos(\psi)E_2$, so we reach the desired formula 
\[\kappa_g^P=\langle\overline\nabla_{\gamma'}\gamma', N\rangle=-\psi'-\sqrt{-\kappa}\cos(\psi).\]
This also makes sense in the limit case $\kappa=0$, in which $\psi$ becomes the angle rotation of the normal of $\gamma$ as a plane curve, whose derivative is well known to agree with the curvature of $\gamma$.
\end{remark}

\subsubsection{The control of the angle function}\label{subsubsec:angle-control}
Lemma~\ref{lem:horizontal-geodesics} exposes the convenience of knowing at which points the angle function takes values $0$ or $\pm 1$ in order to obtain geometric information about the conjugate curves. We collect here some ideas that prove useful when dealing with this type of analysis in different constructions. They are complementary to the maximum principles (in the interior or in the boundary) for minimal surfaces. To simplify the notation, we will denote by $\widetilde\Sigma$ and $\Sigma$ the (immersed) conjugate surfaces.
\begin{itemize}
  \item \textbf{Strategy 1} (points where $\nu=0$). If $p\in\widetilde\Sigma$ is an interior point with $\nu(p)=0$, then compare with the vertical cylinder $T_p$ tangent to $\Sigma$ at $p$.

  The intersection of minimal surfaces consists of a set of regular curves meeting transversally at some isolated points, with at least four rays emanating from tangency points. More generally, there are $2n+2$ rays emanating from a point with order of contact $n$, i.e., where all derivatives of both surfaces (as graphs over their tangent planes) coincide up to order $n$, see~\cite[Lem.~2]{MY82b}. Therefore, this strategy consists in reaching a contradiction if there are too many rays emanating from $p$ and ending in $\partial\widetilde\Sigma\cap T_p$, which generically consists of very few points. However, we must ensure that they cannot enclose regions either of $\widetilde\Sigma$ or in $T_p$ by some maximum principle at the interior or at infinity. Let us also comment about two variants of this strategy:
  \begin{itemize}
    \item Comparing with $T_p$ also gives information if $p\in\partial\widetilde\Sigma$ by applying the boundary maximum principle or by considering $p$ as an interior point after extending the surface across that boundary.
    \item In the case there are two or more curves where $\nu=0$ meeting at $p$, then $\nabla\nu(p)=0$ and there are at least six rays in $\widetilde\Sigma\cap T_p$ emanating from $p$, see~\cite[Lem.~3.5]{MT20}. Note that $\nu$ lies in the kernel of the stability operator given by~\eqref{eqn:stability-op}, which is a second order linear elliptic operator, and hence its zeros also form set of regular curves intersecting transversally at some isolated points, see~\cite{Bers} and the references therein.
  \end{itemize}
  \item \textbf{Strategy 2} (points where $\nu=\pm 1$). The is pretty similar to Strategy 1, but we compare with the umbrella $\mathcal U_p$ centered at an interior or boundary point $p\in\widetilde\Sigma$ where $\nu(p)=\pm 1$, see Example~\ref{ex:umbrellas}. Since $\mathcal U_p$ is tangent to $\widetilde\Sigma$, at least four rays in $\mathcal U_p\cap\widetilde\Sigma$ arise from $p$ and we have to figure out where they end.

  If $p$ belongs to a horizontal boundary component $\widetilde h\subset\partial\widetilde\Sigma$, then it is often useful to compare with the invariant surface $\mathcal I_{\widetilde h}$ defined in Example~\ref{ex:invariant-I}, whose angle is constant $1$ along $\widetilde h$. In particular, some interior arc in $\mathcal I_{\widetilde h}\cap\widetilde\Sigma$ emanates from each point of $\widetilde h$ where $\nu=\pm 1$.

  \item\textbf{Strategy 3} (comparison along the boundary). Sometimes it is possible to find a minimal surface $B$ (also called a \emph{barrier}) that contains a horizontal component $\widetilde h\subset\partial\widetilde\Sigma$ and stays (locally around $\widetilde h$) at one side of $\widetilde\Sigma$. The boundary maximum principle imposes a restriction on the normal $\widetilde N$ of $\widetilde\Sigma$ along $\widetilde h$. If we are able to control the points where $\nu=0$ or $\nu=\pm 1$, the existence of such a barrier often translates into estimates for the angle function $\nu$.

  This might become a bit subtle. For instance, if we know that $\widetilde\Sigma$ and $B$ are multigraphs with angle functions $\nu>0$ and $\nu_B>0$, respectively, and $\Sigma$ stays above $B$ along $\widetilde h$, then this does not necessarily imply that $\nu_B<\nu$ or $\nu<\nu_B$. However, if we know additionally that the invariant surface $\mathcal I_{\widetilde h}$ stays below $B$ (resp. above $\widetilde\Sigma$), then it follows that $0<\nu<\nu_B$ (resp. $\nu_B<\nu<1$).

  Note that this strategy also applies to vertical geodesics, in which case we obtain an estimate for the angle of rotation of the normal (instead of the angle function), which is bounded by the angle of rotation of the barrier.
\end{itemize}

\subsubsection{Completion and embeddedness}\label{sec:completion-embeddedness}

A minimal surface in $\E(\kappa,\tau)$ containing a horizontal or a vertical geodesic segment in the interior must be axially symmetric with respect to that segment, which follows from the work of Leung~\cite{Leung}. Sa Earp and Toubiana~\cite{SaT} have also proved that a minimal surface which has a vertical or horizontal geodesic segment in the boundary and is of class $\mathcal C^1$ up to that boundary can be analytically extended by axial symmetry. By comparing the fundamental data of axially symmetric surfaces and mirror symmetric surfaces (see Remark~\ref{rmk:data-transformations}), it is not difficult to show that a minimal surface in $\E(4H^2+\kappa,H)$ axially symmetric with respect to a geodesic segment $\gamma$ has a conjugate $H$-surface in $\mathbb{M}^2(\kappa)\times\R$ mirror symmetric with respect to the totally geodesic surface $P$ containing the conjugate curve $\gamma$. Also, if two boundary components of $\partial\Sigma$ meet in a vertex with angle $\frac\pi k$ for some $k\geq 2$, then we can apply successive reflections, and the possible singularity at the vertex can be removed by means of a result of Choi and Schoen~\cite[Prop.~1]{CS}. All in all, we can state the following result.

\begin{proposition}\label{prop:conjugation-completion-by-simmetries}
Let $\widetilde\phi:\Sigma\to\E(4H^2+\kappa,H)$ be a minimal immersion with boundary a piecewise regular geodesic polygon consisting of vertical and horizontal geodesic arcs (of finite or infinite length). Assume that $\widetilde\phi$ is of class $\mathcal C^1$ up to the interior of each boundary component and the interior angle at each vertex is an integer divisor of $\pi$. 
\begin{enumerate}[label=\emph{(\alph*)}]
     \item The immersion $\widetilde\phi$ can be extended to a complete minimal immersion $\widetilde\phi^*$ by successive axial symmetries about the boundary components.
     \item The conjugate immersion $\phi:\Sigma\to\M^2(\kappa)\times\R$ can be extended to a complete $H$-immersion $\phi^*$ by successive mirror symmetries about the boundary components.
 \end{enumerate}
\end{proposition}

Once a complete surface is obtained by successive mirror symmetries, it is natural to investigate whether it is embedded or not. This is one of the toughest problems in conjugate constructions, because it must be solved just in terms of the prescribed geodesic polygon, since none of the immersions $\widetilde\phi$ and $\phi$ are explicitly known. It is often convenient to relax this assumption and consider \emph{Alexandrov-embeddedness} instead, which means that the immersed surface can be recovered as the boundary of a $3$-manifold immersed (but possibly non-embedded) in $\E(\kappa,\tau)$. To simplify future discussions, we will distinguish the two possible situations where non-embeddedness occurs in conjugate constructions:
\begin{itemize}
    \item \textbf{Type I self-intersections:} the immersion $\phi$ is not an embedding, i.e., the fundamental piece has self-intersections.
    \item \textbf{Type II self-intersections:} the immersion $\phi$ is an embedding but the completion $\phi^*$ has self-intersections (see Figures~\ref{fig:horizontal-Delaunay-embeddedness-north-pole} and~\ref{fig:non-emb}).
\end{itemize}

On the one hand, self-intersections of type I can be prevented if the boundary of $\phi$ projects one-to-one to $\mathbb{M}^2(\kappa)$ by a standard application of the maximum principle (see also Proposition~\ref{prop:uniqueness}). Convexity has also something to say in this respect; e.g., the following result for multigraphs follows from the maximum principle and from Gau\ss--Bonnet theorem and assumes that the angle at one vertex is convex to prove that the conjugate vertical geodesic is embedded.

\begin{proposition}[{\cite[Lem.~2.1]{CMR}}]\label{prop:rotation-curvature}
Let $\widetilde\phi:\Sigma\to\mathbb{E}(4H^2+\kappa,H)$ and $\phi:\Sigma\to\mathbb{H}^2(\kappa)\times\mathbb{R}$ be conjugate multigraphs, where $\kappa<0$ and $4H^2+\kappa\leq 0$. Assume that $\nu>0$ and $\widetilde\gamma$ is a vertical geodesic such that $\widetilde\gamma'=\xi$. If $\theta'>0$ and $\int_{\widetilde\gamma}\theta'\leq\pi$, then $\gamma$ is embedded.
\end{proposition}

In the case $H=0$, Hauswirth, Sa Earp and Toubiana~\cite[Thm.~14]{HST} extended a theorem of Krust in $\R^3$ (unpublished, see \cite[Theorem~2.4.1]{K89b}) exploiting the convexity of the domain. Notice that a similar property does not hold true in general if $H>0$, as we shall discuss in \S\ref{sec:knoids}.

\begin{proposition}[Krust property]\label{prop:Krust}
Let $\widetilde\phi:\Sigma\to\mathbb{M}^2(\kappa)\times\mathbb{R}$ and $\phi:\Sigma\to\mathbb{M}^2(\kappa)\times\mathbb{R}$ be sister minimal immersions. If $\kappa\leq 0$ and $\widetilde\phi$ is a graph over some convex domain $\Omega\subset\mathbb{M}^2(\kappa)$, then $\phi$ is a graph (and hence embedded).
\end{proposition}

On the other hand, self-intersections of type II are usually treated by an \emph{a priori} subdivision of the target space $\mathbb{M}^2(\kappa)\times\R$ in disjoint congruent regions bounded by horizontal and vertical planes, so that the fundamental piece fits in one of those regions with its boundary lying in the boundary planes. If the desired symmetries and the tessellation can be chosen properly, then there will be one copy of the fundamental piece on each region, so that no self-intersections will be produced by reflection. The constructions in~\S\ref{sec:compact} and~\S\ref{sec:non-compact} follow this philosophy.

\subsection{Some classes of surfaces preserved by the sister correspondence}\label{subsec:surfaces-preserved-by-sister-correspondence}
Daniel correspondence behaves really well with respect to many geometric conditions, which enables the translation of some problems for $H$-surfaces in some $\E(\kappa,\tau)$ to the same problem in another space. The equivariance (see Prop.~\ref{prop:equivariant}) and the stability (see~\eqref{eqn:stability-op}) are good examples of this. Note that any condition that is purely intrinsic to the $H$-surface is also preserved because the correspondence is isometric, e.g., the total curvature or the area growth of intrinsic balls.

\subsubsection{Cylinders and multigraphs}
Vertical $H$-cylinders and vertical $H$-multigraphs are characterized by the conditions $\nu\equiv 0$ and $\nu\neq 0$, respectively. Since $\nu$ is preserved by the sister correspondence, then so are these families of $H$-surfaces. 

Complete multigraphs in $\E(\kappa,\tau)$ are indeed graphs~\cite[Thm.~1]{ManRod}, so in particular complete graphs also form a preserved class. However, it is not true in general that a sister surface of a graph is a graph, as we will discuss later in \S\ref{sec:knoids}. It is worth highlighting that any sister surface of an entire graph with critical mean curvature is again an entire graph~\cite[Cor.\ 4.6.4]{KIAS}, though this is not true in general in the subcritical case, as the following example shows.

\begin{example}
Let $0\leq H<\frac{1}{2}$ and consider the half-space model for $\h^2\times\R$. The surface $P_0$ given by $x^2+y^2=1$ is a vertical plane that can be parametrized isometrically as $(r,h)\mapsto(\tanh(r),\sech(r),h)$, where $r$ is the hyperbolic distance to $(1,0)$ in $\mathbb{H}^2$ and $h$ is the projection onto the factor $\R$. Consider the surface $\Sigma$ parametrized by $(t,r)\mapsto\left(e^t\tanh(r),e^t\sech(r),h(r)\right)$, where
\[h(r)=\frac{1}{\sqrt{1-4H^2}}\arcsinh\left(\frac{\sqrt{1-4H^2}-2H\sinh(r)}{2H+\sqrt{1-4H^2}\sinh(r)}\right)+\frac{2Hr}{\sqrt{1-4H^2}}\]
for $t\in\R$ and $r>-2\arctanh(2H)$. This surface is invariant by the hyperbolic translations $(x,y,z)\mapsto(e^tx,e^ty,z)$ and has constant mean curvature $H$. Since $h(r)$ tends to $+\infty$ and $0$ as $r\to-2\arctanh(2H)$ and $h\to+\infty$, respectively, we deduce that $\Sigma$ is a complete graph over the subset of $\h^2$ of points on the concave side of a complete curve $\alpha$ of constant curvature $2H$ (if $H=0$, then $\Sigma$ is the graph over a half-plane discovered by Sa Earp~\cite[Eqn.~(32)]{Sa08}). Also, $\Sigma$ is mirror symmetric about the vertical planes $P_t$ of equation $x^2+y^2=e^{2t}$ for all $t\in\R$, whence $\Sigma$ is foliated by the congruent curves $\gamma_t=P_t\cap\Sigma$. 

The conjugate minimal surface $\widetilde\Sigma$ in $\E(4H^2-1,H)$ is therefore foliated by the horizontal geodesics $\widetilde\gamma_t$. Note that two curves $\gamma_t$ and $\gamma_s$ lie at bounded distance from each other, and this distance is realized asymptotically at their endpoints in $\alpha\times\{+\infty\}$. We deduce that the $1$-parameter group of hyperbolic translations that leave $\Sigma$ invariant corresponds to a $1$-parameter group of parabolic screw motions that leave $\widetilde\Sigma$ invariant. In particular, we deduce that $\widetilde\Sigma$ is the entire minimal graph associated with the function $u(x,y)=lx$ for some $l\in\R$, also called a \emph{tilted plane}, see~\cite[\S3]{Castro}.

We remark that, if $H>0$, then changing $H$ to $-H$ also gives a non-congruent $H$-graph on the convex side of the equidistant curve $\alpha$. Up to ambient isometries, there are only three tilted planes according to $l\in\{-1,0,1\}$. The complete $H$-graphs that project onto the concave and convex sides of $\alpha$ should correspond to the cases $l=\pm 1$, because the tilted plane with $l=0$ is the parabolic helicoid $P_{0,4H^2+\kappa,H}$ so its conjugate $H$-surface in $\h^2\times\R$ is the entire graph $P_{H,-1,0}$. These last examples will be discussed in~\S\ref{sec:examples-zeroAR}. 
\end{example}

Given an immersion $\phi:\Sigma\to\E(\kappa,\tau)$, the Jacobian determinant of the projection $\pi|_{\Sigma}:\Sigma\to\mathbb{M}^2(\kappa)$ satisfies $|\mathrm{Jac}(\pi|_\Sigma)|=|\nu|$. By a simple change of variable, if $\nu\in L^1(\Sigma)$, this implies that $\int_\Sigma|\nu|$ is the area of $\pi(\Sigma)$ taking into account its possible multiplicity. In particular, this area is preserved by the correspondence. 

\begin{example}[{\cite[\S4.3]{MN}}]
Assume that $\kappa+4H^2<0$. An ideal Scherk graph in $\mathbb{E}(\kappa,\tau)$ is an $H$-graph defined over an ideal polygon $\Omega\subset\mathbb{H}^2(\kappa)$ whose boundary consists of $2n$ curves (with common ideal endpoints) and constant curvature alternatively equal to $2H$ and $-2H$ with respect to the inward-pointing normal to $\partial\Omega$. Some additional Jenkins--Serrin conditions on $\Omega$ are needed to ensure that such a graph exists, but this is not our point here. Ideal Scherk graphs can be characterized as the only complete $H$-multigraphs whose projection has finite area~\cite[Prop.~3]{MN}, in which case we have $\Area(\Omega)=\frac{2(n-1)\pi}{-\kappa-4H^2}$. Therefore, the class of ideal Scherk graphs is preserved by the sister correspondence, and so is the number $2n$ of boundary components.
\end{example}

\subsubsection{Surfaces with zero Abresch--Rosenberg differential}\label{sec:examples-zeroAR}
If $\Sigma$ is topologically a sphere, then it is simply-connected and the correspondence applies globally~\cite[Ex.~5.16]{Dan}. This means that $H$-spheres in $\E(\kappa,\tau)$, all of which are equivariant~\cite{AR}, form a preserved class of surfaces. More generally, a surface $\Sigma$ in $\E(\kappa,\tau)$ has zero Abresch--Rosenberg differential if and only if the following function $q\in\mathcal C^\infty(\Sigma)$ identically vanishes, see~\cite[Lem.~2.2]{ER}:
\begin{equation}\label{eqn:q}
\frac{q}{\kappa-4\tau^2}=\frac{1}{4}\left(\frac{4H^2+\kappa}{\kappa-4\tau^2}-\nu^2\right)\left(4H^2+\kappa+3(\kappa-4\tau^2)\nu^2-4K\right)-\|\nabla\nu\|^2,
\end{equation}
whence $q$ is trivially preserved by the correspondence. In~\cite{AR,Leite,ER} it is shown that surfaces with $q\equiv 0$ are equivariant, and in~\cite[Prop.~2.4]{DomMan} the sister correspondence was employed to prove that, except for vertical cylinders with critical mean curvature, they belong to one of the following three families:
\begin{itemize}
   \item The rotationally invariant surfaces $S_{H,\kappa,\tau}$, locally given in $M(\kappa,\tau)$ by 
   \begin{equation}\label{eqn:param-S}
   X(u,v)=\left(v\cos(u),v\sin(u),\int_0^v\frac{-4Hs\sqrt{1+\tau^2s^2}\,\df s}{(4+\kappa s^2)\sqrt{1-H^2s^2}}\right).
   \end{equation}
   This is half of an $H$-sphere if $4H^2+\kappa>0$ or an entire $H$-graph otherwise. 

   \item The screw-motion invariant surfaces $C_{H,\kappa,\tau}$ in the case $4H^2+\kappa<0$. They are complete $H$-surfaces globally parametrized in the model $M(\kappa,\tau)$ by
   \begin{equation}\label{eqn:param-C}
   \qquad X(u,v)\!=\!\left(v\cos(u),v\sin(u),\frac{4\tau}{\kappa}u\!\pm\!\!\int_{\frac{4H}{|\kappa|}}^v\!\frac{16H\sqrt{16\tau^2+\kappa^2s^2}\,\df s}{\kappa s(4+\kappa s^2)\sqrt{\kappa^2s^2-16H^2}}\right)\!,
   \end{equation}
   The family $C_{H,\kappa,\tau}$ contains helicoids ($H=0$) and rotational catenoids ($\tau=0$).

   \item The parabolic helicoids $P_{H,\kappa,\tau}$ in the case $4H^2+\kappa<0$. They are the entire $H$-graphs given in the halfspace model of $\mathbb{E}(\kappa,\tau)$ by the global parametrization
   \begin{equation}\label{eqn:param-P}
   X(u,v)=\left(u,v,a\log(v)\right),\qquad\text{with } a=\frac{2H\sqrt{-\kappa+4\tau^2}}{-\kappa\sqrt{-4H^2-\kappa}}.
   \end{equation}
\end{itemize}
Each of the families $S_{H,\kappa,\tau}$, $C_{H,\kappa,\tau}$, and $P_{H,\kappa,\tau}$ is preserved by the correspondence, see~\cite[Rmk.~2.5]{DomMan}. On the other hand, the parabolic helicoids $P_{H,\kappa,\tau}$ along with vertical planes and horizontal slices are the only elements of the following classes:
\begin{itemize}
  \item isoparametric $H$-surfaces,
  \item $H$-surfaces with constant principal curvatures,
  \item $H$-surfaces which are homogeneous by ambient isometries, and
  \item $H$-surfaces with constant angle function (see~\cite[Thm.2.2]{ER}).
\end{itemize}
Therefore, each of these geometric conditions is preserved by the correspondence. We remark that the family of $H$-surfaces with constant Gau\ss\ curvature is also preserved (the Gau\ss\ curvature is intrinsic), but it contains a few inhomogeneous examples, see~\cite[Thm.~3.3 and~4.6]{DDV}.

\subsubsection{Ruled minimal surfaces}\label{sec:spherical-helicoids}

Given a minimal immersion $\widetilde\phi:\Sigma\to\E(4H^2+\kappa,H)$ ruled by horizontal geodesics, Lemma~\ref{lem:horizontal-geodesics} implies that the conjugate $H$-immersion $\phi:\Sigma\to\mathbb{M}^2(\kappa)\times\mathbb{R}$ is foliated by lines of mirror symmetry lying in vertical planes. It easily follows that $\phi$ and $\widetilde{\phi}$ are equivariant, which establishes an isometric conjugation between minimal surfaces ruled by horizontal geodesics and $H$-surfaces in product spaces invariant by $1$-parameter groups of isometries acting in the horizontal direction (i.e., such that they fix the factor $\R$). Here we will focus on a few examples that will play the role of barriers in our constructions and that illustrate how conjugation works. Ruled minimal surfaces in any $\mathbb{E}(\kappa,\tau)$-space have been classified by Kim et al.\ in~\cite{KKLSY}.

\begin{example}[The umbrella $\mathcal U_p$]\label{ex:umbrellas}
Given $p\in\mathbb{E}(\widetilde\kappa,\widetilde\tau)$, the umbrella $\mathcal U_p$ is the union of all horizontal geodesics through $p$. It follows that $\mathcal U_p$ coincides with the rotationally invariant surface $S_{0,\widetilde\kappa,\widetilde\tau}$ defined in \S\ref{sec:examples-zeroAR}, whence it is an entire minimal graph if $\widetilde\kappa\leq 0$ and a minimal sphere if $\widetilde\kappa>0$. Note that $\mathcal U_p$ is everywhere horizontal if $\widetilde\tau=0$, it is horizontal only horizontal at its center $p$ if $\widetilde{\kappa} \leq 0$ and it is horizontal at $p$ and its \emph{antipodal} point if $\widetilde{\kappa} > 0$. In the model $M(\widetilde\kappa,\widetilde\tau)$, the umbrella centered at $p=(0,0,0)$ contains the graph of the function $u(x,y)=0$, though this is not the complete umbrella if $\widetilde\kappa>0$.

From~\S\ref{sec:examples-zeroAR}, we deduce that the conjugate $H$-surface of an umbrella in $\E(4H^2+\kappa,H)$ is the rotationally invariant surface $S_{H,\kappa,0}$ in $\mathbb{M}^2(\kappa)\times\R$, which is an $H$-sphere if $4H^2+\kappa>0$ or an entire $H$-graph otherwise.
\end{example}

\begin{example}[The invariant surface $\mathcal I_{\widetilde\gamma}$]\label{ex:invariant-I}
Given a horizontal geodesic $\widetilde\gamma$ in $\E(\widetilde\kappa,\widetilde\tau)$, let $\mathcal I_{\widetilde\gamma}$ be the surface formed by all horizontal geodesics orthogonal to $\widetilde\gamma$. It follows that $\mathcal I_{\widetilde\gamma}$ is an entire minimal graph if $\widetilde\kappa\leq 0$ or $\widetilde\tau=0$; otherwise, it is a minimal embedded torus that will be described as a spherical helicoid in Example~\ref{ex:spherical-helicoids}. The surface $\mathcal I_{\widetilde\gamma}$ is invariant by translations along $\widetilde\gamma$, as in Lemma~\ref{lem:translations}. If $\widetilde\gamma$ is the $x$-axis in the model $M(\widetilde\kappa,\widetilde\tau)$ and $\widetilde\kappa\leq 0$, then $\mathcal I_{\widetilde\gamma}$ is the entire graph of the function
 \[u(x,y)=\begin{cases}
  \widetilde\tau xy&\text{\ if }\widetilde\kappa=0,\\
  \frac{2\widetilde\tau}{\widetilde{\kappa}}\arctan\frac{2xy}{\frac{4}{\widetilde\kappa}+x^2-y^2}&\text{\
    if }\widetilde\kappa<0.
  \end{cases}\]
If we consider the surface $\mathcal{I}_{\widetilde\gamma}$ in $\E(4H^2+\kappa,H)$, then the horizontal geodesic $\widetilde\gamma$ becomes a curve $\gamma$ in a vertical plane of symmetry of the conjugate $H$-surface in $\mathbb{M}^2(\kappa)\times\R$. Since the angle function is $1$ along $\widetilde\gamma$, it follows that $\gamma$ is also a horizontal geodesic, whence the $H$-surface is invariant by translations along $\gamma$. It is an $H$-cylinder or $H$-torus if $4H^2+\kappa>0$ (explicit parametrizations can be found in~\cite[\S2]{Man13}), or an entire $H$-graph otherwise.
\end{example}

\begin{example}[Spherical helicoids~{\cite[\S4]{MT14}}]\label{ex:spherical-helicoids}
Lawson~\cite[Prop.~7.2]{Law} defined the spherical helicoid of pitch $c\in\R$ in the round $3$-sphere $\s^3$ as the immersion
\[
\begin{split}
\widetilde\phi_c: \R^2 &\rightarrow \s^3\subset\C^2 \\
(u, v) &\mapsto \bigl( \cos(u) e^{icv}, \sin(u) e^{iv} \bigr).
\end{split}
\]
These immersions are characterized by being minimal if we substitute $\s^3$ with $\Sb(\widetilde\kappa,\widetilde\tau)$ for any $\widetilde\kappa>0$ and $\widetilde\tau\neq 0$, see~\cite{Tor12}. They are also equivariant by screw motions and ruled by horizontal geodesics in all Berger metrics. The surface parametrized by $\widetilde\phi_c$ is topologically a torus or a Klein bottle if $c\in\mathbb{Q}-\{0\}$ (these are Lawson's examples $\tau_{m,n}$, see~\cite[\S7]{Law}), a sphere if $c=0$, and a dense cylinder otherwise; however, it is embedded if and only if $c\in\{-1,0,1\}$. Note that $\widetilde\phi_1$ and $\widetilde\phi_{-1}$ are congruent Clifford tori in the round metric $\mathbb{S}^3$ but they are not quite alike in the Berger case: the torus $\widetilde\phi_1$ is flat and everywhere vertical (i.e., the preimage by $\pi$ of a great circle of $\mathbb{S}^2(\widetilde\kappa)$), yet $\widetilde\phi_{-1}$ is not flat and it is vertical just along a couple of horizontal geodesics. As a matter of fact, $\widetilde\phi_{-1}$ is congruent to the invariant surface $\mathcal I_{\widetilde\gamma}$ we discussed in Example~\ref{ex:invariant-I} in the case of Berger spheres, where $\widetilde\gamma(t)=\frac{1}{\sqrt{2}}(e^{-it}, e^{it})$. We also remark that $\widetilde\phi_c$ and $\widetilde\phi_{1/c}$ are congruent surfaces for all $c\neq 0$, so we can assume that $c\in[-1,1]$ not losing any generality.

The spherical helicoids $\widetilde\phi_c$ in $\Sb(4H^2+\kappa,H)$, where $4H^2+\kappa>0$, have the rotational Delaunay $H$-surfaces as conjugate $H$-immersions $\phi_c$ in $\M^2(\kappa)\times\R$. This was proved in~\cite[Thm.~2.1]{G} for $\kappa=0$ (Lawson correspondence) and in~\cite[\S4]{MT14} in the general case. By comparing the rotation of the normal along the rulings of $\widetilde\phi_c$ with the curvature of the curve that generates the Delaunay surfaces by means of Lemma~\ref{lem:horizontal-geodesics}, we obtain the following list of conjugate surfaces:
\begin{enumerate}[label=(\alph*)]
  \item $\phi_1$ is the vertical $H$-cylinder over a curve of $\M^2(\kappa)$ of curvature $2H$,
  \item $\phi_c$ is a rotationally invariant $H$-unduloid if $0<c<1$,
  \item $\phi_0$ is the rotationally invariant $H$-sphere,
  \item $\phi_c$ is a rotationally invariant $H$-nodoid if $-1<c<0$,
  \item $\phi_{-1}$ is an $H$-torus ($\kappa>0$) or an $H$-cylinder ($\kappa\leq 0$) invariant by translations along a horizontal geodesic.
\end{enumerate}
We refer to~\cite[Lem~1.3]{PR} and~\cite[\S2]{Man13} for a more specific description of Delaunay $H$-surfaces. It is worth emphasizing that different directions of rotation of the normal along the axis of $\widetilde\phi_c$ output very different surfaces, as already discussed in Lemma~\ref{lem:vertical-geodesics}. This will be further explored in \S\ref{sec:knoids}. 
\end{example}


\section{Dirichlet problems for $H$-surfaces in $\E(\kappa,\tau)$}

The initial minimal surfaces in our conjugate constructions will be solutions of a Plateau or a Jenkins--Serrin problem in $\E(4H^2+\kappa,H)$, both of which will be thought of as Dirichlet problems for the minimal surface equation. Needless to say that not all geodesic polygons span graphical surfaces; if that is not the case, additional work will be required to find the desired minimal surface.

\subsection{The Plateau problem}\label{sec:plateau}

We will solve a Plateau problem in the more general scenario of arbitrary Killing submersions in the sense of~\cite{LerMan}. This is motivated by the fact that some of our minimal surfaces will be graphs in $\E(4H^2+\kappa,H)$ with respect to other (non-unitary) Killing directions. In the sequel, we will assume that $\pi:\E\to M$ is a Riemannian submersion whose fibers have infinite length and coincide with the integral curves of a nowhere vanishing Killing field $\xi$.

\begin{definition}\label{def:nitsche}
A Nitsche contour in $\E$ is a pair $(\Omega,\Gamma)$, where $\Omega\subset M$ is a precompact open domain and $\Gamma\subset\E$ is a Jordan curve admitting a piecewise-regular parametrization $\gamma:[a,b]\to\Gamma$ satisfying the following conditions:
\begin{enumerate}[label=(\alph*)]
 \item There is a partition $a=t_1<s_1\leq t_2<\ldots\leq t_r<s_r\leq t_{r+1}=b$ such that $\gamma(a)=\gamma(b)$ and, for any $j\in\{1,\ldots,r\}$, the component $\gamma|_{[t_j,s_j]}$ is a nowhere vertical curve and $\gamma|_{[s_j,t_{j+1}]}$ a vertical segment;
 \item The projection $\pi\circ\gamma$ parametrizes $\partial\Omega$ injectively except at vertical segments.
\end{enumerate}
This means that $\partial\Omega$ is a regular curve except possibly at the points $\pi(t_i)$, that will be called the vertexes of $\Omega$.
\end{definition}

Graphical surfaces with prescribed mean curvature and boundary a Nitsche contour $(\Omega,\Gamma)$ satisfy a maximum principle which has been proved in the case of unitary Killing submersions~\cite[Prop.~3.8]{Man12}, but also holds in the general case. This relies on the fact that the mean curvature of Killing graphs can be written in divergence form similar to~\eqref{eqn:H-graphs}, see~\cite[\S3]{LerMan}. We remark that in this case, the Killing graph of a function $u\in\mathcal{C}^\infty(\Omega)$ is defined in the usual way with respect to a global section $F_0:M\to\E$ as $F_u(p)=\phi_{u(p)}(F_0(p))$. Note that a global section exists by the assumption on the infinite length of the fibers~\cite[Thm.~12.2]{Ste}.

\begin{proposition}\label{prop:uniqueness}
Let $(\Omega,\Gamma)$ and $(\Omega,\Gamma')$ be Nitsche contours in $\E$ over the same domain $\Omega$ with the same set of vertexes $V\subset\partial\Omega$. Assume that $u,v\in\mathcal{C}^\infty(\Omega)$ verify:
\begin{itemize}
 \item[a)] the graphs $F_u$ and $F_v$ have the same (possibly non-constant) mean curvature over $\Omega$,
 \item[b)] $u$ and $v$ extend continuously to $\overline\Omega-V$ giving rise to surfaces with boundaries $\Gamma$ and $\Gamma'$, respectively.
\end{itemize}
If $u\leq v$ in $\partial\Omega-V$, then $u\leq v$ in $\Omega$.
\end{proposition}

In particular, given a Nitsche contour $(\Omega,\Gamma)$, there is at most one minimal graph over $\Omega$ with boundary $\Gamma$. Although uniqueness holds in general, we will consider some additional conditions that ensure that $\pi^{-1}(\overline\Omega)\subset\E$ is a $3$-dimensional mean-convex body in the sense of Meeks and Yau~\cite{MY82}, so there exists at least one minimal surface with boundary $\Gamma$. It follows from a standard application of the maximum principle and from uniqueness that the solution is a graph, using the same argument as in~\cite[Thm.~3.11]{Man12} (see also~\cite[Prop.~2]{MT14}). All in all, we have the following result:

\begin{proposition}\label{prop:plateau-existencia}
Let $(\Omega,\Gamma)$ be a Nitsche contour such that $\Omega$ is simply connected, $\pi^{-1}(\pi(\alpha))$ is a minimal surface for each non-vertical component $\alpha\subset\Gamma$, and the interior angle at each vertex of $\Omega$ is at most $\pi$. There is a unique minimal surface $\Sigma\subset\pi^{-1}(\overline\Omega)$ with boundary $\Gamma$ and the interior of $\Sigma$ is a Killing graph over $\Omega$.
\end{proposition}

Observe that the solution $\Sigma$ must be invariant by any ambient isometry that preserves both the Killing submersion and the Nitsche contour $(\Omega,\Gamma)$ because of the uniqueness. We also observe that Meeks and Yau's solution to the Plateau problem is of class $\mathcal C^1$ up to the boundary, see~\cite[Rmk.~3.4]{SaT}, which enables analytic continuation of the surface across vertical and horizontal geodesics in $\Gamma$ by axial symmetry in the sense of Proposition~\ref{prop:conjugation-completion-by-simmetries}.

\subsection{The Jenkins--Serrin problem}\label{sec:JS}
The non-parametric Plateau problem we have discussed in Proposition~\ref{prop:plateau-existencia} amounts to obtaining a function $u$ whose graph is minimal over $\Omega$ and extends continuously to $\overline\Omega-V$, being $V$ the set of vertexes of $\Omega$ with the prescribed values on $\partial\Omega-V$ induced by $\Gamma$. The Jenkins--Serrin problem also allows $\pm\infty$ values along the components of $\partial\Omega$, and has been studied in $\E(\kappa,\tau)$-spaces with $\kappa<0$ respect to the usual Killing direction by also allowing the domain $\Omega\subset\h^2(\kappa)$ to be unbounded. 

We will first deal with the case $\kappa<0$ and $\tau=0$, in which the solution has been proved under very general assumptions by Mazet, Rodríguez and Rosenberg~\cite{MRR11}. We will suppose that $\partial\Omega$ is piecewise regular and contains finitely-many components of three distinct types: $A_1,\ldots,A_{n_1}\subset\h^2(\kappa)$, where the value $+\infty$ is prescribed, $B_1,\ldots,B_{n_2}\subset\h^2(\kappa)$, where the value $-\infty$ is prescribed, and $C_1,\ldots,C_{n_3}\subset\h^2(\kappa)$, where continuous finite boundary values are prescribed. We will also assume that $\partial_\infty\Omega$ consists of finitely many ideal segments $D_1,\ldots,D_{n_4}\subset\partial_\infty\h^2(\kappa)$ where continuous finite boundary values are also prescribed. The endpoints of the segments $A_i$, $B_i$, $C_i$ and $D_i$ will be called the vertexes of $\Omega$. There are two necessary conditions for the existence of solution: on the one hand, each component $A_i$ or $B_i$ must be a geodesic segment; on the other hand, there cannot be two consecutive components of type $A_i$ meeting at an interior angle less than $\pi$, and likewise for those of type $B_i$. Additionally, it is assumed that the arcs $C_i$ are convex with respect to the inner conormal to $\Omega$ along $C_i$. Under these conditions, we will say that $\Omega$ is a \emph{general Jenkins--Serrin domain}.

A polygonal domain $\mathcal P$ in $\mathbb{H}^2(\kappa)$ is a domain whose boundary $\partial\mathcal{P}$ consists of finitely-many geodesic segments. We will say that it is inscribed in a general Jenkins--Serrin domain $\Omega$ if $\mathcal P\subset\overline\Omega$ and all vertexes of $\mathcal P$ are vertexes of $\Omega$. Assume that the ideal vertexes of $\Omega$ are $p_1,\ldots,p_m\in\partial_\infty\h^2(\kappa)$ and, for each $i\in\{1,\ldots,m\}$, let $H_i\subset\h^2(\kappa)$ be a domain with boundary a horocycle asymptotic to $p_i$, small enough so that it only intersects the components of $\partial\Omega$ with endpoint at $p_i$ and $H_i\cap H_j=\emptyset$ for all $j\neq i$. Under these assumptions, the following finite lengths characterize the existence of solution:
\begin{align*}
\alpha(\mathcal P)&=\sum_{i=1}^{n_1}\length(\cup_{i=1}^{n_1}(A_i\cap\partial\mathcal{P})-\cup_{j=1}^mH_j),\\
\beta(\mathcal P)&=\sum_{i=1}^{n_2}\length(\cup_{i=1}^{n_2}(B_i\cap\partial\mathcal{P})-\cup_{j=1}^mH_j),\\
\gamma(\mathcal P)&=\length(\partial\mathcal P-\cup_{j=1}^mH_j).
\end{align*}

\begin{theorem}[{\cite[Thm.~4.9 and~4.12]{MRR11}}]\label{thm:general-JS}
Let $\Omega\subset\h^2(\kappa)$ be a general Jenkins--Serrin domain on which we consider the above Dirichlet problem.
\begin{enumerate}[label=(\alph*)]
    \item If $n_3=n_4=0$, then the problem has a solution if and only if $\alpha(\Omega)=\beta(\Omega)$ and $\max\{\alpha(\mathcal P),\beta(\mathcal P)\}<\frac{1}{2}\gamma(\mathcal P)$ for all polygonal domains $\mathcal P\neq\overline\Omega$ inscribed in $\Omega$.
    \item Otherwise, the problem has a solution if and only if $\max\{\alpha(\mathcal P),\beta(\mathcal P)\}<\frac{1}{2}\gamma(\mathcal P)$ for all polygonal domains $\mathcal P$ inscribed in $\Omega$.
\end{enumerate}
\end{theorem}

Nelli and Rosenberg~\cite{NR} proved this result for convex relatively compact domains, and Collin and Rosenberg~\cite{CR} extended it for unbounded convex domains all of whose vertexes are ideal and with $n_4=0$. It is worth mentioning that convexity at the vertexes is a condition that ensures that the surface is of class $\mathcal C^1$ up to the boundary and enables the completion of the surface given by Proposition~\ref{prop:conjugation-completion-by-simmetries}, as proved by Sa Earp and Toubiana~\cite{SaT2}. This also implies that solutions to Jenkins--Serrin problems over convex domains have finite radial limits on the corners (see~\cite[Lem.~4]{CM}). In the non-convex case, the solution given by Theorem~\ref{thm:general-JS} is a smooth graph over the interior of $\Omega$, but might fail to be of class $\mathcal C^1$ at a vertical geodesic projecting to a non-convex vertex. This was noticed by Finn~\cite[Thm.~3]{Finn} and Eclat and Lancaster~\cite{ElcLan} for minimal surfaces in $\mathbb{R}^3$.

A very general maximum principle for solutions of the  Jenkins--Serrin problem in $\mathbb{H}^2(\kappa)\times\R$ can be found in~\cite[Thm.~4.16]{MRR11}. However, we will consider the following earlier version of Collin and Rosenberg, which is enough when the prescribed boundary data comes from a geodesic polygon with vertical and horizontal (possibly ideal) components. It clearly implies that the solution given by Theorem~\ref{thm:general-JS} in that case is unique if there are non-ideal horizontal geodesics (which are components of type $C_i$). On the other hand, the solution can be shown to be unique up to vertical translations if there are no such horizontal geodesics~\cite[Thm.~4.12]{MRR11}. Note that there are no components of type $D_i$ if the boundary is a geodesic polygon.

\begin{proposition}[{\cite{CR}}]\label{prop:uniqueness-JS}
Let $\Omega\subset\h^2(\kappa)$ be a general Jenkins--Serrin domain with $n_3\neq 0$ and $n_4=0$, and such that any two components of $\partial\Omega$ meeting at a common ideal vertex are asymptotic to each other at that vertex (in hyperbolic distance). 

Assume that $u,v\in\mathcal C^\infty(\Omega)$ span minimal graphs over $\Omega$ and extend continuously to $\overline\Omega-V$ (where $V$ is the set of vertexes of $\Omega$) with possible infinite values along some of the $A_i$ or $B_i$. If $u\leq v$ on $\partial\Omega-V$, then $u\leq v$ on $\Omega$.
\end{proposition}

Theorem~\ref{thm:general-JS} and Proposition~\ref{prop:uniqueness-JS} are believed to hold true also in the case $\tau\neq 0$ (observe that both statements translate literally to this more general case and all the needed tools seem to be available in $\SL$). However, they have not hitherto been proved in this generality. The two results in this direction so far have been given by Younes~\cite{Younes}, who has proved existence and uniqueness in the case of a bounded convex domain, and by Melo~\cite{Melo}, who proves only existence over convex unbounded domains having only ideal vertexes and no arcs of type $D_i$. Note also that Younes and Melo's surfaces can be used as barriers to give \emph{ad hoc} solutions to more general Jenkins--Serrin problems as in~\cite[Lem.~3.2]{CMR}.

As a final remark, it is worth noticing that the solution $\Sigma$ to any of the aforesaid Jenkins--Serrin problems can be obtained as the limit of a double sequence of minimal surfaces $\Sigma_n^k$. The surface $\Sigma_n^k$ can be taken as the solution of a Plateau problem over a Nitsche contour $(\Omega_n,\Gamma_n^k)$, where the domains $\Omega_n\subset\h^2(\kappa)$ are bounded and converge to $\Omega$ and the curves $\Gamma_n^k$ replace the target boundary components with prescribed values $\pm\infty$ with a sequence of horizontal geodesics that diverges in the desired direction. The surface $\Sigma_n^k$ is a solution to a Plateau problem in the sense of Proposition~\ref{prop:plateau-existencia} and converges to a minimal graph $\Sigma^k$ over $\Omega$ as $n\to\infty$. The surfaces $\Sigma^k$ in turn converge to the solution $\Sigma$ as $k\to\infty$. Since the convergence is of class $C^m$ on compact subsets for all $m$, we can use the description given by Lemmas~\ref{lem:horizontal-geodesics} and~\ref{lem:vertical-geodesics} plus the continuity of the conjugation in Proposition~\ref{prop:continuity} to analyze the ideal geodesics as limits of non-ideal geodesics. In particular, we obtain the following result:

\begin{proposition}[{\cite[Cor.~2.4]{CMR}}]\label{prop:conjugation-JS}
Assume that $4H^2+\kappa<0$ and let $\widetilde\Sigma\subset\E(4H^2+\kappa,H)$ be the solution of a Jenkins--Serrin problem with boundary a polygon consisting of vertical and horizontal (possibly ideal) geodesics. Let $\Sigma\subset\h^2(\kappa)\times\R$ be the conjugate (possibly non-embedded) $H$-multigraph.
\begin{enumerate}[label=\emph{(\alph*)}]
\item Ideal vertical geodesics in $\partial_\infty\widetilde\Sigma$, if any, become ideal horizontal curves in $\partial_\infty\Sigma$ of constant curvature $\pm 2H$ at height $\pm\infty$.
\item Ideal horizontal geodesics in $\partial_\infty\widetilde\Sigma$ become ideal vertical geodesics of $\partial_\infty\Sigma$.
\end{enumerate}
\end{proposition}


\section{Compact $H$-surfaces in $\mathbb{S}^2(\kappa)\times\mathbb{R}$}\label{sec:compact}

Now we will focus on constructions where the ideas we have discussed in the previous sections apply. The first group of examples we will deal with concern the construction of compact or periodic $H$-surfaces such that the fundamental piece is compact and comes from the solution of a Plateau problem. As pointed out in the introduction, we are interested in providing embedded examples with different topologies in $\mathbb{S}^2(\kappa)\times\R$ or in a quotient by a vertical translation.

\subsection{Horizontal Delaunay $H$-surfaces}\label{sec:delaunay}
This section is devoted to the construction of $H$-surfaces in $\mathbb{S}^2(\kappa)\times \mathbb{R}$ and $\mathbb{H}^2(\kappa)\times \mathbb{R}$ provided that $4H^2+\kappa > 0$ and $H>0$ by conjugating minimal surfaces in $M(4H^2+\kappa, H)$. The name \emph{horizontal Delaunay} is motivated by the fact that these $H$-surfaces resemble classical Delaunay surfaces in Euclidean space. Although they are not equivariant if $\kappa\neq 0$ (in contrast with the vertical case discussed in Example~\ref{ex:spherical-helicoids}), they are invariant under a discrete group of translations along a horizontal geodesic called the axis of the surface. Horizontal Delaunay surfaces comprise a deformation of the $H$-spheres into the $H$-cylinders by means of unduloid-type surfaces, and they also contain non-embedded  surfaces of nodoid-type. The next result describes the moduli space of horizontal Delaunay $H$-surfaces:

\begin{theorem}[\cite{MT14,MT20}]\label{thm:horizontal-Delaunay}
  Fix $\kappa\in\R$ and a horizontal geodesic $\Lambda\subset\mathbb{M}^2(\kappa)\times\{0\}$. There exists a family $\Sigma_{\lambda,H}^*$, parametrized by $\lambda\geq 0$ and $H>0$ such that $4H^2+\kappa>0$, of complete $H$-surfaces in $\mathbb{M}^2(\kappa)\times \R$, invariant under a discrete group of translations along $\Lambda$ with respect to which they are cylindrically bounded. They are also symmetric about the totally geodesic surfaces $\mathbb{M}^2(\kappa)\times\{0\}$ and $\Lambda\times\mathbb R$. Moreover:
\begin{enumerate}[label=(\roman*)]
  \item $\Sigma_{0,H}^*$ is the $H$-cylinder ($H$-torus if $\kappa>0$) invariant under the continuous $1$-parameter group of translations along $\Lambda$;
  \item $\Sigma_{\lambda,H}^*$ is a unduloid-type surface if $0<\lambda<\frac{\pi}{2}$;
  \item $\Sigma_{\frac{\pi}{2},H}^*$ is a stack of tangent rotationally invariant $H$-spheres centered on $\Lambda$;
  \item $\Sigma_{\lambda,H}^*$ is a nodoid-type surface if $\lambda>\frac{\pi}{2}$. 
\end{enumerate}
\end{theorem}

\begin{figure}[htbp]
  \centering
  \includegraphics[width=\linewidth]{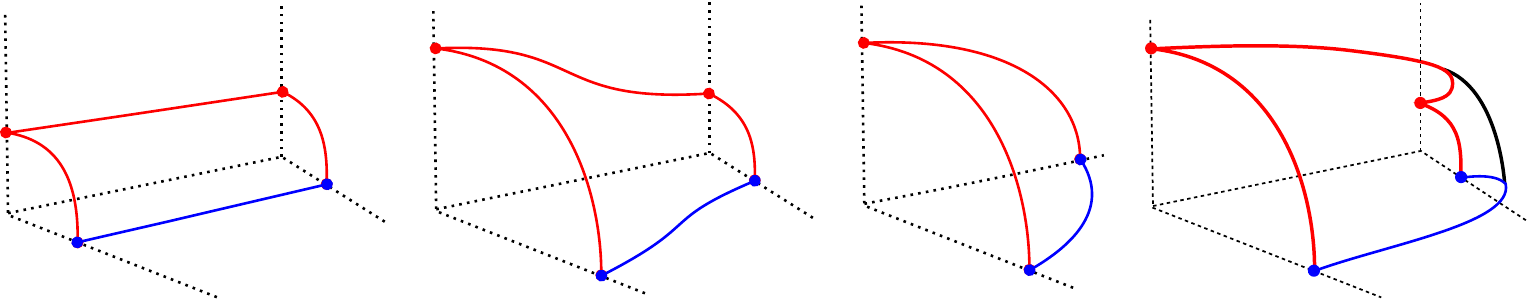}
  \caption{Schematic representation of the fundamental piece of horizontal Delaunay $H$-surfaces in $\mathbb{M}^{2}(\kappa) \times \mathbb{R}$ (see Theorem~\ref{thm:horizontal-Delaunay}). From left to right: cylinder (torus if $\kappa > 0$), unduloid, sphere, and nodoid}
  \label{fig:}
\end{figure}

In the case $\kappa>0$, we find (among the horizontal unduloids) many families of embedded $H$-tori which continuously deform stacks of tangent $H$-spheres evenly distributed along a horizontal geodesic $\Lambda$ into an equivariant $H$-torus. The embeddedness is shown by spotting a function in the kernel of the common stability operator of the conjugate surfaces that is induced simultaneously by two $1$-parameter groups of isometric deformations: one in the initial space $M(4H^2+\kappa, H)$ and the other one in the target space $\mathbb{M}^2(\kappa)\times \mathbb{R}$. We will see later that this function carries insightful geometric information that also implies that horizontal unduloids are properly embedded if $\kappa\leq 0$, see~\cite[Prop.~4.4]{MT20}.

\begin{theorem}[{\cite[Thm.~1.2]{MT20}}]\label{thm:horizontal-Delaunay-embeddedness}
  Fix $\kappa>0$. For each integer $m\geq 2$, there is a family $\mathcal T_m$ of embedded $H$-tori in $\mathbb{S}^2(\kappa)\times\mathbb R$ (see Figure~\ref{fig:horizontal-Delaunay-compact-embedded-moduli-space}) parametrized as
\[\mathcal T_m=\left\{\Sigma_{\lambda_m(H),H}^*:\cot(\tfrac{\pi}{2m})<\tfrac{2H}{\sqrt{\kappa}}\leq\sqrt{m^2-1}\right\},\]
where $H\mapsto\lambda_m(H)$ is a continuous strictly decreasing function ranging from $\frac{\pi}{2}$ to $0$.
\begin{enumerate}
  \item The family $\mathcal T_m$ is a continuous deformation (in which $H$ varies) from a stack of $m$ tangent spheres evenly distributed along $\Lambda$ to an equivariant torus.
  \item The surfaces $\Sigma_{\lambda_m(H),H}^*$, along with $H$-spheres $\Sigma_{\pi/2,H}^*$ and $H$-cylinders $\Sigma_{0,H}^*$ for all $H>0$, are the only compact embedded $H$-surfaces among all $\Sigma_{\lambda,H}^*$ (for all $\kappa\in\R$).
\end{enumerate}
\end{theorem}

\begin{figure}[htb]
\includegraphics{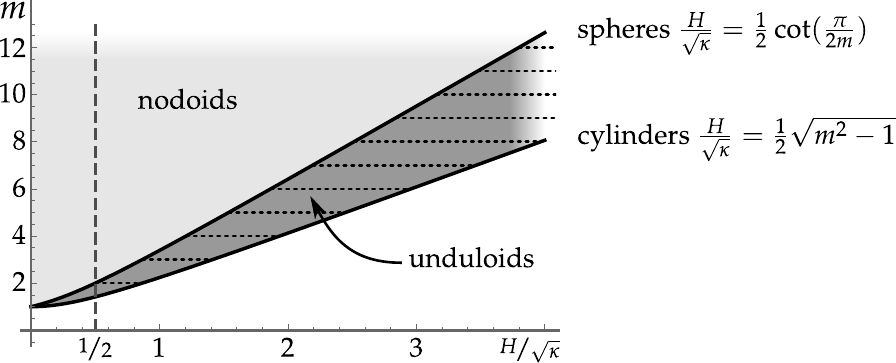}
\caption{The darker region is the moduli space of $\Sigma_{\lambda,H}^*\subset\mathbb{S}^2(\kappa)\times\mathbb{R}$, $\lambda\in[0,\frac\pi2]$, in terms of $\frac{H}{\sqrt\kappa}$ and $m$. Dotted horizontal segments represent the families $\mathcal{T}_m$ of embedded tori. The vertical dashed line indicates that no such tori exist if $H>\frac{\sqrt\kappa}{2}$.}\label{fig:horizontal-Delaunay-compact-embedded-moduli-space} 
\end{figure}

\subsubsection{Construction of the minimal surface in $M(4H^2 + \kappa, H)$}\label{subsubsec:horizontal-Delaunay-construction-minimal-piece}
Assume that $H,\kappa\in\mathbb{R}$ are fixed such that $H>0$ and $4H^2+\kappa>0$ (we will omit the dependence on $H$ in the sequel unless otherwise stated). To obtain an $H$-surface in $\mathbb{M}^2(\kappa)\times\mathbb{R}$ by means of the correspondence, we will begin by constructing a compact minimal surface $\widetilde\Sigma_\lambda$, $\lambda\in\R$, with boundary a polygon $\widetilde\Gamma_\lambda\subset M(4H^2 + \kappa, H)$ consisting of three horizontal geodesics $\widetilde{h}_0$, $\widetilde{h}_1$ and $\widetilde{h}_2$, and one vertical geodesic $\widetilde v$, whose vertexes will be labeled as $\widetilde 1$, $\widetilde 2$, $\widetilde 3$ and $\widetilde 4$, as shown in Figure~\ref{fig:horizontal-Delaunay-polygon-Berger}. More precisely,
\begin{itemize}
  \item $\widetilde{h}_0$ is a quarter of a horizontal geodesic, which can be parametrized, up to an ambient isometry, by $\widetilde{h}_0(s)=\tfrac{2}{\sqrt{4H^2+\kappa}}\left(0,  \tfrac{\cos(2s)}{1 + \sin(2s)}, 0\right)$, for  $s \in \bigl[0, \tfrac\pi2\bigr]$.
  
  \item $\widetilde{h}_1$ and $\widetilde{h}_2$ are horizontal geodesics starting orthogonally at the endpoints of $\widetilde{h}_0$ with signed lengths $\tfrac{1}{2\sqrt{4H^2+\kappa}}(\lambda - \frac{\pi}{2})$ and $\tfrac{1}{2\sqrt{4H^2+\kappa}}(\lambda + \frac{\pi}{2})$ in the directions $\widetilde{h}_0'(0)\times\xi$ and $-\widetilde{h}_0'(\frac{\pi}{2})\times\xi$, respectively. In particular, $\widetilde h_1$ goes in opposite directions according to the sign of this length.

  \item $\widetilde{v}$ is the vertical geodesic joining the endpoints of $\widetilde{h}_1$ and $\widetilde{h}_2$.
\end{itemize}
Observe that $\widetilde{\Gamma}_\lambda$ and $\widetilde{\Gamma}_{-\lambda}$ are congruent for all $\lambda>0$ by the isometry $(x, y, z) \mapsto (-x,y,-z)$, so we can assume that $\lambda \geq 0$ without loss of generality.

\begin{figure}[htbp]
\centering
\includegraphics{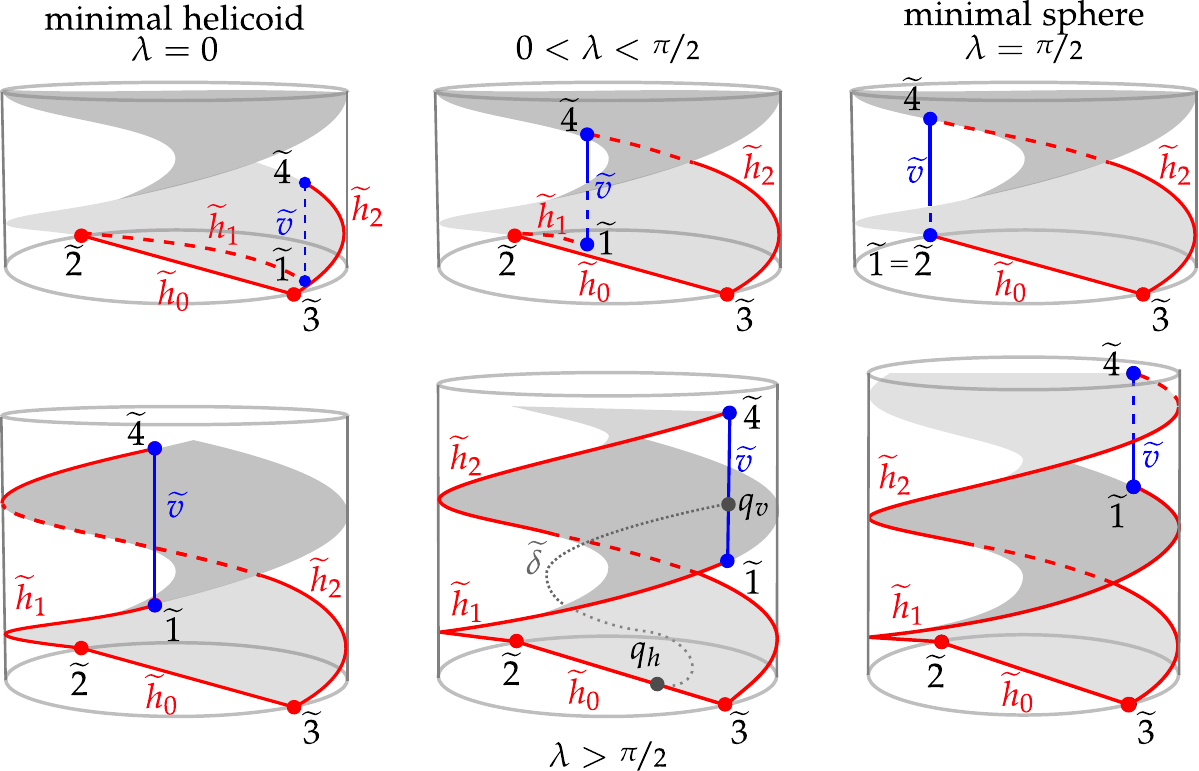}
\caption{A faithful representation of the polygon $\widetilde{\Gamma}_\lambda$ for different values of $\lambda$. The barriers $T$ (vertical cylinder) and $S$ (helicoid, in gray) demarcate the mean convex solid $\Omega$. Top row: polygons $\widetilde{\Gamma}_\lambda$ for $\lambda \leq \frac{\pi}{2}$ whose conjugate surfaces are unduloids and limit cases $\lambda = 0$ (left), which corresponds to the $H$-cylinder, and $\lambda = \frac{\pi}{2}$ (right), which corresponds to the $H$-sphere. Bottom row: polygon $\widetilde{\Gamma}_\lambda$ for $\lambda > \frac{\pi}{2}$ whose conjugate $H$-surfaces are nodoids. The dotted line in the central figure represents the curve $\widetilde{\delta}$ of zeros of the angle function, see Proposition~\ref{prop:horizontal-Delaunay-angle}.}
\label{fig:horizontal-Delaunay-polygon-Berger}
\end{figure}

\begin{remark}\label{rmk:polygon-case-Michigan}
The polygon $\Theta(\widetilde{\Gamma}_\lambda)$, where $\Theta:M(4H^2 +\kappa, H)\to\Sb(4H^2 + \kappa, H)$ is the local isometry given by~\eqref{eq:local-isometry-Daniel-Berger}, has self-intersections if $\lambda\geq\frac{7\pi}{2}$, so the resulting Plateau problem is ill-posed in $\Sb(4H^2+\kappa,H)$. This is the reason why the locally isometric Cartan model $M(4H^2+\kappa,H)$ is used throughout this section.
\end{remark}

The polygon $\widetilde{\Gamma}_\lambda$ was first consider in~\cite{MT14} for $0 \leq \lambda \leq \frac{\pi}{2}$ and then extended for $\lambda \geq \frac{\pi}{2}$ in~\cite{MT20}. We remark that the mean convex body used to solve the Plateau problem in~\cite{MT14} is no longer valid for $\lambda>\frac{\pi}{2}$ and a different approach was developed to show the existence of $\widetilde{\Sigma}_\lambda$, which we present here.

Consider the following two minimal surfaces in $M(4H^2 + \kappa, H)$: the vertical cylinder $T$ of equation $x^2 + y^2 = \tfrac{4}{4H^2 + \kappa}$ and the surface $S = \mathcal{I}_{\widetilde{h}_0}$ (see Example~\ref{ex:invariant-I}) of equation $x \cos(\tfrac{4H^2+\kappa}{2H} z) + y \sin(\tfrac{4H^2+\kappa}{2H} z)=0$. Observe that $\Theta(T)$ is congruent to the Clifford torus $\widetilde\phi_1$ in Example~\ref{ex:spherical-helicoids} and $\Theta(S)$ is congruent to $\widetilde\phi_{-1}$ in Example~\ref{ex:spherical-helicoids} (see the shaded surface in Figure~\ref{fig:horizontal-Delaunay-compact-embedded-moduli-space}), where $\Theta$ is the local isometry defined in~\eqref{eq:local-isometry-Daniel-Berger}. The surface $S$ divides the interior domain of the cylinder $T$ in two connected components.  The connected component $W$ that contains $\widetilde{v}$ is a mean-convex solid in the sense of~\cite{MY82}, so there exists an embedded minimal disk $\widetilde{\Sigma}_\lambda \subset W$ solution to the Plateau problem with boundary $\partial \widetilde{\Sigma}_\lambda = \widetilde{\Gamma}_\lambda \subset W$.

If $0 \leq \lambda \leq \frac{\pi}{2}$, then $\widetilde{\Gamma}_\lambda$ is a Nitsche contour with respect to the usual Killing submersion, but this is not true in general. To overcome this issue, consider the Killing vector field $\widetilde X$ associated with the group of screw-motions that leave the surface $S$ invariant. Then $\widetilde{X}$ has no zeroes and gives rise to a Killing submersion $\pi_0:M(\kappa,\tau)\to(\R^2,\df s^2)$ in the sense of~\cite{LerMan} (note that the metric $\df s^2$ has not constant curvature, $\widetilde X$ has not constant length, and the bundle curvature is not constant). The curves $\widetilde h_0$ and $\widetilde v$ are transversal to the fibers of $\pi_0$ whereas $\widetilde h_1$ and $\widetilde h_2$ become vertical for $\pi_0$ (they are integral curves of $\widetilde{X}$). Hence, $\widetilde\Gamma_\lambda$ is a Nitsche contour with respect to $\pi_0$ in the sense of Definition~\ref{def:nitsche} for all $\lambda\geq 0$. Proposition~\ref{prop:plateau-existencia} implies that $\widetilde{\Sigma}_\lambda$ is unique surface in $W$ with boundary $\widetilde\Gamma_\lambda$, and it is everywhere transversal to $\widetilde X$. Uniqueness in turn implies that $\widetilde\Sigma_\lambda$ depends continuously on $\lambda\geq 0$ (and so does its sister surface $\Sigma_\lambda$ thanks to Proposition~\ref{prop:continuity}).

We highlight the following special cases, depicted in Figure~\ref{fig:horizontal-Delaunay-polygon-Berger} (see left and right drawings in the top row):
\begin{itemize}
  \item If $\lambda = 0$, then $\widetilde{\Sigma}_0$ is part of a spherical helicoid with axis $\widetilde{v}$, that is, $\Theta(\widetilde{\Sigma}_{0})$ is congruent to part of the spherical helicoid $\widetilde\phi_{-1}$ of equation $\pIm(z^2 + w^2) = 0$ (see Figure~\ref{fig:horizontal-Delaunay-polygon-Berger} top left).
  \item If $\lambda = \frac{\pi}{2}$, then $\widetilde{\Sigma}_{\frac{\pi}{2}}$ is an open piece of a minimal sphere. More precisely, $\Theta(\widetilde{\Sigma}_{\frac{\pi}{2}})$ is part of the minimal sphere $\pIm(z - w) = 0$ (see Figure~\ref{fig:horizontal-Delaunay-polygon-Berger} top right).
\end{itemize}

\begin{remark}\label{rmk:round-sphere1}
  If $\kappa=0$, then $M(4H^2+\kappa,H)$ is locally isometric to the three-sphere $\mathbb{S}^3(H^2)$ and the completion $\widetilde{\Sigma}_\lambda^*$ is a spherical helicoid  (see Example~\ref{ex:spherical-helicoids}) for all $\lambda \geq 0$ (observe that $\widetilde{v}$ and $\widetilde{h}_0$ have the same length and $\widetilde{\Sigma}_\lambda$ is invariant under the composition of translations and suitable rotations about $\widetilde{h}_2$). This is a consequence of the isotropy of the three-sphere $\mathbb{S}^3(H^2)$ that allows rotations around any geodesic. This argument obviously fails if $\kappa \neq 0$ (the segments $\widetilde{v}$ and $\widetilde{h}_0$ no longer have the same length and there is no screw motion group with axis $\widetilde{h}_2$).
\end{remark}

\subsubsection{Analysis of the angle function}\label{subsubsec:analysis-angle-function}

As pointed out in \S\ref{sec:conjugation}, the angle function $\nu_\lambda:\widetilde{\Sigma}_\lambda\to[-1,1]$ gives precious information about the shape of the boundary curves of the conjugate contour $\Gamma_\lambda$. We are mainly interested in the points where $\nu_\lambda$ has values $\pm 1$ or $0$, if any. We will also establish the monotonicity properties of $\nu_\lambda$ as a function of $\lambda$ along the horizontal geodesic arcs of $\widetilde{\Gamma}_\lambda$ in the common boundary of intersection $\widetilde{\Gamma}_{\lambda_1}\cap \widetilde{\Gamma}_{\lambda_2}$, $\lambda_1 \neq \lambda_2$. In the sequel, we will choose the normal $\widetilde N$ to $\widetilde\Sigma_\lambda$ such that $\nu_\lambda(\widetilde{3}) =-1$.

\begin{proposition}[{\cite[Lem.~3]{MT14} and \cite[Prop.~3.3]{MT20}}]\label{prop:horizontal-Delaunay-angle}
Let $\nu_\lambda$ be the angle function of the compact minimal disk $\widetilde{\Sigma}_\lambda$ spanning $\widetilde{\Gamma}_\lambda$ such that $\nu_\lambda(\widetilde 3)=-1$.
\begin{enumerate}[label=(\alph*)]
  \item The only points in which $\nu_\lambda$ takes the values $\pm1$ are $\widetilde 2$ and $\widetilde 3$. More precisely, if $0<\lambda<\frac{\pi}{2}$, then $\nu_\lambda(\widetilde{2}) = -1$; if $\lambda > \frac{\pi}{2}$, then $\nu_\lambda(\widetilde{2}) = 1$. 

  \item \label{prop:angle:item:zeroes-angle} The set of points in which $\nu_\lambda$ vanishes consists of $\widetilde v$ and, in the case $\lambda > \frac{\pi}{2}$, also of a certain interior regular curve $\widetilde{\delta}\subset\widetilde\Sigma_\lambda$ with endpoints in $\widetilde v$ and $\widetilde h_0$ (see Figure~\ref{fig:horizontal-Delaunay-polygon-Berger}).

  \item Given $p$ in the horizontal boundary of $\widetilde{\Gamma}_\lambda$, the function $\lambda\mapsto\nu_\lambda(p)$ is continuous in the interval where it is defined.
  \begin{itemize}
    \item It is strictly increasing (possibly changing sign) if $p \in\widetilde h_0$ for all $\lambda>0$.
    \item It is positive and strictly increasing if $p \in\widetilde h_1$ and $\lambda>\frac{\pi}{2}$.
    It is negative and strictly increasing if $p \in\widetilde h_1$ and $0<\lambda<\frac{\pi}{2}$.
    \item It is negative and strictly increasing if $p \in\widetilde h_2$ for all $\lambda>0$. 
  \end{itemize}
\end{enumerate}
\end{proposition}

Although we will not provide a full proof of Proposition~\ref{prop:horizontal-Delaunay-angle}, it essentially relies on strategies 1, 2 and 3 discussed in \S\ref{subsubsec:angle-control}. We will illustrate this by sketching the analysis of the interior zeros of $\nu$. To this end, we intersect $\widetilde\Sigma_\lambda$ and the tangent vertical cylinder $T_p$ at some interior point $p$ with $\nu(p)=0$. 

If $0<\lambda<\frac\pi 2$, then the (at least) four rays in $T_p\cap\widetilde\Gamma_\lambda$ emanating from $p$ end up either in $\widetilde v$ or in $\widetilde h_0\cup\widetilde h_1\cup\widetilde h_2$. If two of the rays reach $\widetilde v$, they enclose a region of $\widetilde\Sigma_\lambda$ (along with a segment of $\widetilde v$); otherwise, there are two of the rays that reach the same point of $\widetilde h_0\cup\widetilde h_1\cup\widetilde h_2$ and they also enclose a region of $\widetilde\Sigma_\lambda$. Either way, such a region has boundary on the vertical cylinder $T_p$ and we easily find a contradiction with the maximum principle with respect to other vertical cylinders~\cite[Lem.~1.6]{MT20}. This means that there are no interior points with $\nu=0$ if $0<\lambda<\frac\pi 2$.

Assume now that $\lambda>\frac{\pi}{2}$ and apply the same reasoning. However, since $\Sigma_\lambda\subset W$ the rays emanating from $p$ lie in the connected component of $(T_p\cap W)-S$ containing $p$, which is a vertical quadrilateral with boundary in $S\cup T_p$: three of its sides lie in $S$ if $T_p'$ contains the $z$-axis (see Figure~\ref{fig:intersection-helicoid-torus} center), otherwise only two of the sides lie in $S$ (see Figure~\ref{fig:intersection-helicoid-torus} left and right). It is not difficult to realise that if more than four rays arise from $p$, then there will be enclosed regions in contradiction with the maximum principle as discussed above. This implies that there are no points with $\nu(p)=\nabla\nu(p)=0$, whence the curves of $\nu=0$ do not bifurcate. The argument can be further refined to prove that there is exactly one regular curve $\widetilde\delta$ where $\nu=0$, but we will not include the details here.

\begin{figure}[htbp]
  \centering
  \includegraphics{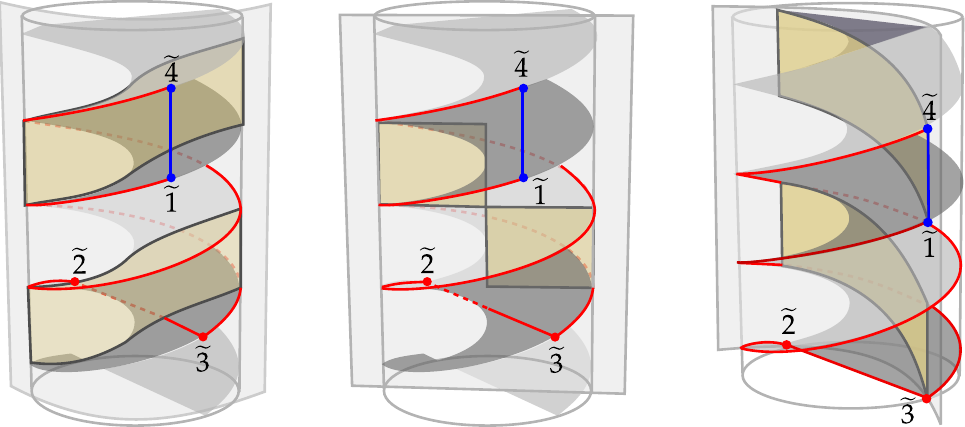}
  \caption{Each figure indicates two connected components of $(T_p \cap W)-S$ for a Clifford torus $T_p$ inside the mean convex solid $W$. From left to right: a general case, a case in which $T_p$ contains the $z$-axis, and a case in which $T_p$ contains $\widetilde{v}$.} 
  \label{fig:intersection-helicoid-torus}
\end{figure}

As for item (c), we will also say some words about the application of the maximum principle in the boundary. Assume first that $0<\lambda_1<\lambda_2<\frac{\pi}{2}$. This case is easier since $\widetilde\Sigma_{\lambda_1}$ is a barrier from above for $\widetilde\Sigma_{\lambda_2}$ along the common horizontal boundary, and this enables a direct comparison of the normals. It is important to mention that the surface $\mathcal I_{\widetilde h_0}$, whose angle function is equal to $-1$ along $\widetilde h_0$, see Example~\ref{ex:invariant-I}, is nothing but $\widetilde\Sigma_0$ and acts as a barrier to $\widetilde\Sigma_{\lambda_2}$ from below. This sandwich between $\widetilde\Sigma_0$ and $\widetilde\Sigma_{\lambda_1}$yields the sign and monotonicity stated in Proposition~\ref{prop:horizontal-Delaunay-angle}. In the case $\frac\pi2<\lambda_1<\lambda_2$, the discussion is similar but we have to observe that $\widetilde\Sigma_{\lambda_1}$ acts as a barrier from below for $\widetilde\Sigma_{\lambda_2}$ as graphs in the direction of the helicoidal Killing vector field $\widetilde X$ along $\widetilde h_1$ and $\widetilde h_2$.

\subsubsection{The conjugate $H$-immersion}\label{subsubsec:conjugate-horizontal-Delaunay}

Let $\Sigma_\lambda \subset \mathbb{M}^2(\kappa) \times \mathbb{R}$ be the conjugate $H$-surface of $\widetilde{\Sigma}_\lambda$ constructed in the previous section. By Lemmas~\ref{lem:horizontal-geodesics} and~\ref{lem:vertical-geodesics}, $\Sigma_\lambda$ is a compact $H$-surface whose boundary $\Gamma_\lambda$ consists of three curves $h_0$, $h_1$ and $h_2$, contained in vertical planes $P_{23}$, $P_{12}$ and $P_{34}$, respectively, and a curve $v$ lying in a slice $P_{14}$, which will be assumed to be $\mathbb{M}^2(\kappa)\times \{0\}$ after a vertical translation. The vertexes of $\Sigma_\lambda$ will be denoted by $1$-$4$ in correspondence with $\widetilde{1}$-$\widetilde{4}$ (see Figure~\ref{fig:horizontal-Delaunay-conjugate-polygon}). All interior angles of $\Gamma_\lambda$ are equal to $\frac\pi2$, so Proposition~\ref{prop:conjugation-completion-by-simmetries} gives a complete $H$-surface $\Sigma_\lambda^*$ after successive mirror symmetries about the boundary components. 

\begin{figure}[htbp]
  \centering
  \includegraphics{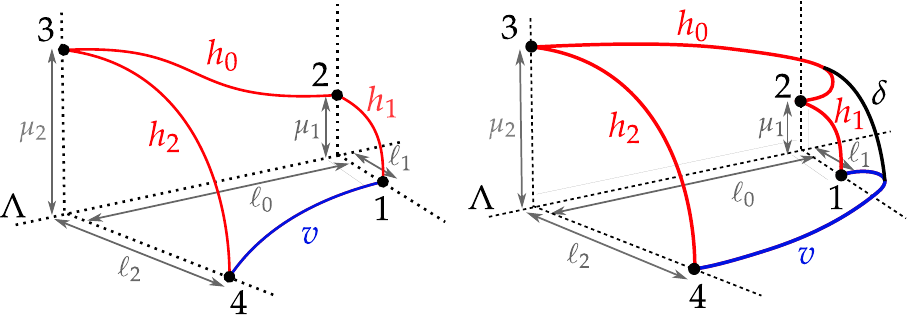}
  \caption{Conjugate contour $\Gamma_\lambda$ for $0<\lambda < \frac{\pi}{2}$ (left) and $\lambda > \frac{\pi}{2}$ (right). }
  \label{fig:horizontal-Delaunay-conjugate-polygon}
\end{figure}

Our analysis of the angle function $\nu_\lambda$ indeed implies that Figure~\ref{fig:horizontal-Delaunay-conjugate-polygon} is a faithful depiction of $\Gamma_\lambda$, at least concerning the horizontal components. By decomposing $h_i=(\beta_i,z_i)\in\mathbb{M}^2(\kappa)\times\mathbb R$ component-wise for $i\in\{0,1,2\}$, Lemma~\ref{lem:horizontal-geodesics} and Proposition~\ref{prop:horizontal-Delaunay-angle} reveal that $\beta_1$, $\beta_2$, $z_0$, $z_1$ and $z_2$ are one-to-one. If $0 \leq \lambda \leq \frac{\pi}{2}$, then $\beta_0$ is also one-to-one, whilst it consists of two one-to-one subcurves if $\lambda>\frac\pi2$ (by splitting at the point where the angle function $\nu_\lambda$ changes sign). 

To understand the dependence of $\Sigma_\lambda$ on the parameter $\lambda$, we shall also consider the following quantities that represent the (algebraic) lengths of $\beta_i$ and $z_i$, respectively, as shown in Figure~\ref{fig:horizontal-Delaunay-conjugate-polygon}:
\begin{align}\label{eqn:definition-ell-mu}
\ell_i(\lambda)&=-\int_{\widetilde h_i}\nu_\lambda,&
\mu_i(\lambda)&=\int_{\widetilde h_i}\sqrt{1-\nu_\lambda^2}
\end{align}
Proposition~\ref{prop:horizontal-Delaunay-angle} also reveals that the functions $\lambda\mapsto\ell_i(\lambda)$ satisfy the following monotonicity properties:
 \begin{enumerate}[label=(\alph*)]
   \item[(a)] $\lambda\mapsto\ell_0(\lambda)$ is strictly decreasing and positive on $[0,+\infty).$
   \item[(b)] $\lambda\mapsto\ell_1(\lambda)$ is strictly decreasing on $[0,+\infty)$ with $\ell_1(\frac{\pi}{2})=0$.
   \item[(c)] $\lambda\mapsto\ell_2(\lambda)$ is strictly increasing  and positive on $[0,+\infty)$.
 \end{enumerate}

\begin{remark}\label{rmk:round-sphere2}
If $\kappa=0$, then Remark~\ref{rmk:round-sphere1} ensures that $\widetilde\Sigma_\lambda$ is a spherical helicoid and hence $\Sigma_\lambda^*$ is a Delaunay surface (see Example~\ref{ex:spherical-helicoids}). Due to the above geometric description, $\Sigma_\lambda^*$ stays at bounded distance from the straight line $\Lambda=P_{23}\cap(\mathbb{R}^2\times\{0\})$ and it is symmetric with respect to $P_{23}$ and $\mathbb{R}^2\times\{0\}$, so in particular $\Sigma_\lambda^*$ is rotationally invariant about $\Lambda$. 
\end{remark}

\subsubsection{Compactness}\label{subsubsec:horizontal-Delaunay-compactness}
Since the vertical planes $P_{12}$ and $P_{34}$ are orthogonal to $P_{23}$, it follows that $\Sigma_\lambda^*$ is invariant under horizontal translations of length $2\ell_0(\lambda)$ along the horizontal geodesic $\Lambda = P_{23} \cap \mathbb{M}^2(\kappa) \times \{0\}$ (see Figures~\ref{fig:horizontal-Delaunay-conjugate-polygon} and~\ref{fig:horizontal-Delaunay-fundamental-annulus}).

Assuming that $\kappa>0$, the surface $\Sigma_\lambda^*$ is compact if and only if $\ell_0(\lambda)$ is a rational multiple of $\frac{2\pi}{\sqrt\kappa}$, the length of a great circle of $\mathbb{S}^2(\kappa)$. Since $\lambda\mapsto\ell_0(\lambda)$ is a positive continuous strictly decreasing function (see~\cite[Cor.~3.7]{MT20} for the details), it follows that there are many compact examples in the family $\Sigma_\lambda^*$ for $\lambda\geq0$. Observe that the monotonicity of $\lambda\mapsto\ell_0(\lambda)$ evidences that $P_{12}$ and $P_{34}$ never coincide. Hence, if $\kappa\leq 0$, $\Sigma_\lambda^*$ is a proper non-compact $H$-surface for all $\lambda\geq 0$.

\begin{remark}
The \emph{maximum height} over the slice $\mathbb{M}^2(\kappa)\times\{0\}$ of horizontal unduloids $\Sigma_\lambda^*$, $0 \leq \lambda \leq \frac{\pi}{2}$, is given by $\mu_2(\lambda)$ and varies continuously from the height of the $H$-tori $\Sigma_0^*$ to the height of the sphere $\Sigma_{\frac{\pi}{2}}^*$ (see \cite[Prop.~4.5]{MT20}). Notice that $\Sigma_\lambda^*$ is singly periodic in a horizontal direction, and the monotonicity properties of Proposition~\ref{prop:horizontal-Delaunay-angle} show that the maximum height occurs at the vertex $3$ (see Figure~\ref{fig:horizontal-Delaunay-conjugate-polygon}). Now, given $\lambda_1 < \lambda_2$, a comparison between $\Sigma_{\lambda_2}$ and a $(\lambda_2-\lambda_1)$-translated copy of $\Sigma_{\lambda_1}$ using the flow of the helicoidal Killing vector field $\widetilde{X}$ (see \S\ref{subsubsec:horizontal-Delaunay-construction-minimal-piece}) enables a comparison of the angle functions of both surfaces along their common boundary. Then, formula~\eqref{eqn:definition-ell-mu} ensures that $\lambda\mapsto\mu_2(\lambda)$ is strictly increasing.

Aledo, Espinar, and Gálvez~\cite{AEG08} proved that a $H$-graph, with $4H^2+\kappa>0$, over a compact open domain whose boundary lies in the slice $\mathbb{M}^2(\kappa)\times \{0\}$ can reach at most the height of the $H$-sphere and equality holds if and only if the surface is a rotationally invariant hemisphere. Horizontal unduloids provide the first $H$-graphs with height in between those of the cylinder and the sphere.
\end{remark}

\subsubsection{Embeddedness}\label{subsubsec:horizontal-Delaunay-embeddedness}

We will finish our sketch of the proof of Theorems~\ref{thm:horizontal-Delaunay} and~\ref{thm:horizontal-Delaunay-embeddedness} by saying which horizontal Delaunay surfaces are embedded. Observe first that our previous description of $\Gamma_\lambda$ implies that horizontal nodoids ($\lambda > \frac{\pi}{2}$) are not even Alexandrov-embedded, so we can focus on unduloids with $0<\lambda<\frac{\pi}{2}$. To this end, consider the \emph{fundamental annulus} $A_\lambda$ defined as the $H$-annulus in $\mathbb M^2(\kappa)\times\mathbb R$ that consists of four copies of $\Sigma_\lambda$ obtained by mirror symmetries about $P_{23}$ and $\mathbb{M}^2(\kappa)\times\{0\}$ (see Figure~\ref{fig:horizontal-Delaunay-fundamental-annulus}). 

To capture some global properties of $A_\lambda$, throughout this section we will leave the model $M(4H^2+\kappa,H)$ aside and assume that $\widetilde\Sigma_\lambda$ is immersed in $\Sb(4H^2+\kappa,H)\subset\C^2$ via the local isometry $\Theta$ given by~\eqref{eq:local-isometry-Daniel-Berger}. Under this assumption, the interior of $\widetilde\Sigma_\lambda\subset\Sb(4H^2+\kappa,H)$ is transversal to the Killing field $\widetilde X_{(z,w)}=\frac{i}{2}(-z,w)$ that defines a global Killing submersion $\pi_0$ we have used in the solution of the Plateau problem (in the Berger model, this submersion is topologically the Hopf fibration in another non-vertical direction). Therefore, the smooth function $\widetilde u=\langle\widetilde X,\widetilde N\rangle$ lies in the kernel of the stability operator $L$ of $\widetilde\Sigma_\lambda^*$.

The function $\widetilde u$ is positive in the interior of $\widetilde\Sigma_\lambda$ and vanishes along $\widetilde h_1$ and $\widetilde h_2$ (recall that $\widetilde\Sigma_\lambda$ is a Killing graph in the direction of $\widetilde X$ and both $\widetilde h_1$ and $\widetilde h_2$ are integral curves of $\widetilde{X}$). Furthermore, $\widetilde u$ is preserved by the axial symmetries about $\widetilde{v}$ and $\widetilde{h}_0$ and changes sign by those about $\widetilde{h}_1$ and $\widetilde{h}_2$. This follows from a careful analysis of the relation between the axial symmetries about the boundary curves of $\widetilde{\Gamma}_\lambda$ and the screw-motion group associated to $\widetilde{X}$ (see \cite[Prop.~1.6]{MT20}). Since these axial symmetries correspond to mirror symmetries of $A_\lambda$ in $\mathbb{M}^2(\kappa)\times \mathbb{R}$, it follows that $\widetilde u > 0$ in the interior of $A_\lambda$ and vanishes identically along $\partial A_\lambda$. Hence, by classical elliptic theory, the first eigenvalue of the stability operator is $\lambda_1(A_\lambda) = 0$ and $\lambda_1(D)< 0$ for any open domain $D \subset \Sigma^*_\lambda$ containing $A_\lambda$. In other words, the annulus $A_\lambda$ is a maximal stable domain of $\Sigma_\lambda^*$ (for all $\lambda > 0$).

\begin{figure}[htbp]
  \centering
  \includegraphics{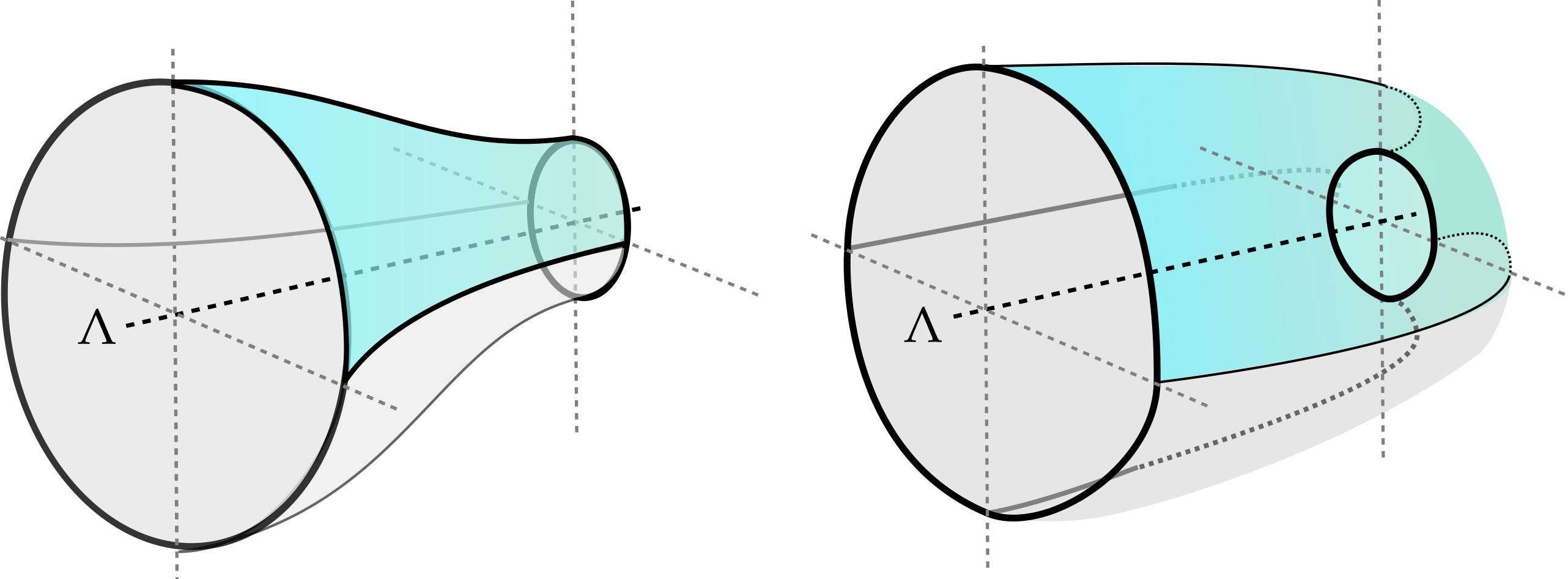}
  \caption{Fundamental annulus $A_\lambda$ constructed by extending the fundamental piece $\Sigma_\lambda$ by means of mirror symmetries across $P_{23}$ and $\mathbb{M}^2(\kappa)\times \{0\}$. The horizontal geodesic $\Lambda = P_{23} \cap(\mathbb{M}^2(\kappa)\times\{0\})$ is called the \emph{axis} of $\Sigma_\lambda^*$.}
  \label{fig:horizontal-Delaunay-fundamental-annulus} 
\end{figure}

On the other hand, we can consider the function $u = \langle X, N\rangle$, where $X$ is now the Killing field in $\mathbb{M}^2(\kappa)\times\R$ associated with the group of translations along the axis $\Lambda = P_{23}\cap(\mathbb M^2(\kappa)\times\{0\})$. Since $Lu=0$ and $u$ vanishes at $\partial A_\lambda$, we infer that $u$ belongs to the eigenspace associated to $\lambda_1(A_\lambda) = 0$. This eigenspace is $1$-dimensional, so there exists $a_\lambda \in \mathbb{R}$ such that $u = a_\lambda\widetilde u$. Notice that $u$ is only identically zero if $\lambda = 0$ because it is the only case where $\Sigma_\lambda^*$ is equivariant. Hence, for $\lambda > 0$ we have that $u$ does not vanish on the interior of $A_\lambda$, that is, the interior of $A_\lambda$ is transversal to $X$, in particular $h_0$ and $v$ are also transversal to $X$ since they lie in the interior of $A_\lambda$.

\begin{enumerate}
  \item The fundamental annulus $A_\lambda$ is an $H$-graph in the direction of $X$ (i.e., it intersects each integral curve of $X$ at most once), and hence embedded, provided that $\kappa\leq 0$ or $\kappa>0$ and $H\geq\frac{\sqrt{\kappa}}{2}$, see~\cite[\S4.3]{MT20}.

  Notice that embeddedness finds an essential obstruction whenever $\Sigma_\lambda$ runs over any of the poles defined by the great circle $\Lambda$ (see Figure~\ref{fig:horizontal-Delaunay-embeddedness-north-pole}). This situation is prevented by assuming that $H > \frac{\sqrt{\kappa}}{2}$ thanks to the monotonicity properties in \S\ref{subsubsec:horizontal-Delaunay-compactness} because $\Sigma_{\frac{\pi}{2}}$ is the $H$-sphere whose \emph{radius} $\ell_2(\frac{\pi}{2})$ is at most a quarter of the length of a great circle of $\mathbb{S}^2(\kappa)$ if $H\geq\frac{\sqrt{\kappa}}{2}$.

\begin{figure}[tbp]
  \centering
  \includegraphics{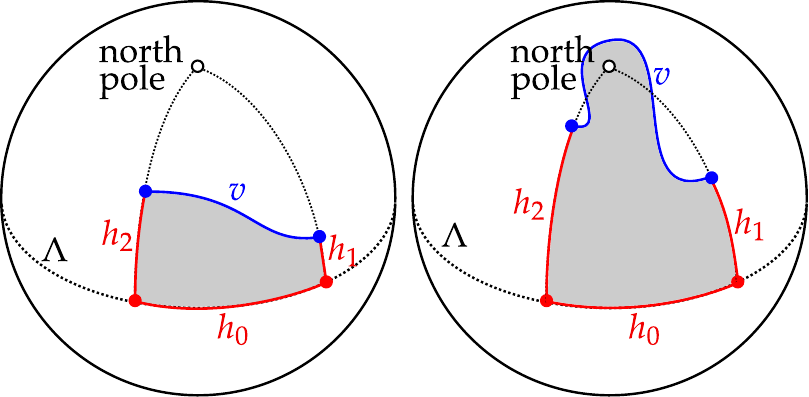}
  \caption{Different possible projections of $\Sigma_\lambda$ to $\mathbb{S}^2(\kappa)\times \{0\}$. In the first case the completion $\Sigma_\lambda^*$, $0 \leq \lambda \leq \frac{\pi}{2}$, is embedded. The second case is not possible if $H > \frac{\sqrt{\kappa}}{2}$.}
  \label{fig:horizontal-Delaunay-embeddedness-north-pole}
\end{figure}

  \item It remains to analyze the possible type II self-intersections (after completing the fundamental annulus by further reflections); observe that they actually happen for nodoids with $\lambda>\frac{\pi}{2}$. If $0 \le \lambda \leq \frac{\pi}{2}$, the geodesic curvature $\kappa_g^P$ of $v$ (computed as a curve of the slice $P=\mathbb{M}^2(\kappa)\times \{0\}$ with respect to its normal vector field $N$ pointing inside the domain of the multigraph) admits the upper bound $\kappa_g^P \leq (4H^2-\kappa)/4H$ (i.e., it is bounded by the geodesic curvature of the equator of the $H$-sphere, see~\cite[Thm.~3.3]{Man13}). This implies that $A_\lambda$ lies in the wedge bounded by $P_{12}$ and $P_{34}$: if $v$ escaped this region, then its length should be larger than it actually is to be able to meet $P_{12}$ and $P_{34}$ orthogonally from inside the wedge, see~\cite[p.~714]{MT14}.
\end{enumerate}
The argument sketched here yields the embeddedness of the unduloids if $\kappa\leq 0$, see~\cite[Prop.~4.4]{MT20}. However, if $\kappa>0$, then we need to guarantee that $\ell_0(\lambda)$ is not only a rational multiple of $\frac{2\pi}{\sqrt{\kappa}}$ (so $\Sigma^*_\lambda$ is compact, see \S\ref{subsubsec:horizontal-Delaunay-compactness}) but also that $\ell_0(\lambda) = \frac{\pi}{m\sqrt{\kappa}}$ for some $m\in\mathbb{N}$, i.e., $\Sigma^*_\lambda$ consists of $2m$ copies of $A_\lambda$ and closes its period in one turn around $\Lambda$. The monotonicity properties in \S\ref{subsubsec:horizontal-Delaunay-compactness} and the well-known behavior of the cases $\lambda = 0$ and $\lambda = \frac{\pi}{2}$ (see \S\ref{subsubsec:horizontal-Delaunay-construction-minimal-piece}) give the estimate
\begin{equation}\label{thm:embeddedness:eqn1}
\tfrac{2}{\sqrt{ \kappa}}\arctan\tfrac{\sqrt{\kappa}}{2H}=\ell_0(\tfrac{\pi}{2})<\ell_0(\lambda)< \ell_0(0)=\tfrac{\pi}{\sqrt{4H^2+\kappa}}.
\end{equation}
This easily implies that $H > \frac{\sqrt{\kappa}}{2}$ (see Figure~\ref{fig:horizontal-Delaunay-compact-embedded-moduli-space}). For a fixed integer $m\geq 2$, the inequality~\eqref{thm:embeddedness:eqn1} holds true if and only if $\frac{2H}{\sqrt{\kappa}}\in(\cot(\tfrac{\pi}{2m}),\sqrt{m^2-1})$, in which case there is a unique $\lambda=\lambda_m(H)$ such that $\ell_0(\lambda_m(H))=\frac{\pi}{m\sqrt{\kappa}}$ because $\lambda\mapsto\ell_0(\lambda)$ is continuous and strictly decreasing. This gives rise to the family $\mathcal T_m$ in the statement of Theorem~\ref{thm:horizontal-Delaunay-embeddedness}, with the following limit cases:
\begin{itemize}
  \item If $\frac{2H}{\sqrt{\kappa}}=\cot(\tfrac{\pi}{2m})$, then $m=\frac{\pi}{2\arctan(\frac{\sqrt{\kappa}}{2H})}$, and hence $\ell_0(\frac\pi2)=\ell_0(\lambda)$. This means that $\lambda=\frac\pi2$ and the surface reduces to a stack of $m$ tangent $H$-spheres.
  \item Likewise, if $\frac{2H}{\sqrt{\kappa}}=\sqrt{m^2-1}$, then $\lambda=0$, and the surface is an $H$-torus.
\end{itemize}

\subsection{Compact $H$-surfaces of arbitrary genus in $\mathbb{S}^2(\kappa)\times\mathbb{R}$}\label{sec:genus}

The second conjugate construction we present in this survey concerns compact embedded $H$-surfaces with arbitrary genus $g\geq 0$ in $\mathbb{S}^2(\kappa)\times \mathbb{R}$ and mean curvature $H < \frac{\sqrt{\kappa}}{2}$. The idea is to produce a compact fundamental piece $\Sigma_{H,g}$ whose projection fits in a fundamental triangle of a regular tessellation of $\mathbb{S}^2(\kappa)$ by regular polygons, see~\S\ref{subsubsec:tessellations}. The complete $H$-surface $\Sigma_{H,g}^*$ after reflection about its symmetry planes inherits all the symmetries of the tessellation of $\mathbb{S}^2(\kappa)$, whence it is compact. In the case of genus $0$ or $1$, we already have the rotationally invariant $H$-spheres and $H$-tori, so our result is relevant for $g\geq 2$. Observe that the horizontal Delaunay tori given by Theorem~\eqref{thm:horizontal-Delaunay-embeddedness} satisfy the opposite inequality $H>\frac{\sqrt{\kappa}}{2}$.

Therefore, it is worth saying something about the assumption $H<\frac{\sqrt{\kappa}}{2}$, which initially showed up in the continuity argument used to adjust $\Sigma_{H,g}$ to the shape of a fundamental triangle, but then proved to be a natural constraint. The value $H=\frac{\sqrt{\kappa}}{2}$ is geometrically relevant in $\mathbb{S}^2(\kappa)\times\R$ because it is the value of $H$ for which $H$-spheres are bigraphs over a hemisphere, and hence two of them are tangent along a whole equator. As a matter of fact, we can think of the map $H\mapsto\Sigma_{H,g}^*$, for a fixed $g\geq 2$, as a desingularization of two such tangent spheres as $H\to\frac{\sqrt{\kappa}}{2}$. This number is also natural in the proof of embeddedness, which uses the convexity of the boundary components in~\cite[Cor.~3.5]{Man13}. Recall that in the proof of embeddedness in Theorem~\eqref{thm:horizontal-Delaunay-embeddedness}, the opposite condition $H>\frac{\sqrt{\kappa}}{2}$ has also appeared as a natural constraint that prevents the surfaces to surpass the north pole and ensures that the $H$-tori close one of their periods.

The main result of this section is the following theorem.

\begin{theorem}\label{thm:arbitrary-genus-H-surface-embeddedness}
Let $0 < H < \frac{\sqrt{\kappa}}{2}$ and an integer $g\geq 0$. There exists a compact embedded $H$-surface $\Sigma_{H,g}^*$ of genus $g$ in $\mathbb{S}^2(\kappa)\times \mathbb{R}$ that is a bigraph over a slice and inherits all the symmetries of a $(2,g+1)$-tessellation (see \S\ref{subsubsec:tessellations}). Furthermore, if $g\geq 2$, 
\begin{itemize}
  \item the limit of $\Sigma_{H,g}^*$ as $H\to\frac{\sqrt{\kappa}}{2}$ is a pair of $\frac{\sqrt{\kappa}}{2}$-spheres tangent along an equator; 
  \item the limit of $\Sigma_{H,g}^*$ as $H\to 0$ is a double cover of $\mathbb{S}^2(\kappa)\times \{0\}$ with singularities at $g+1$ points evenly distributed along an equator. 
\end{itemize}
\end{theorem}

\begin{figure}[htbp]
\centering
\includegraphics{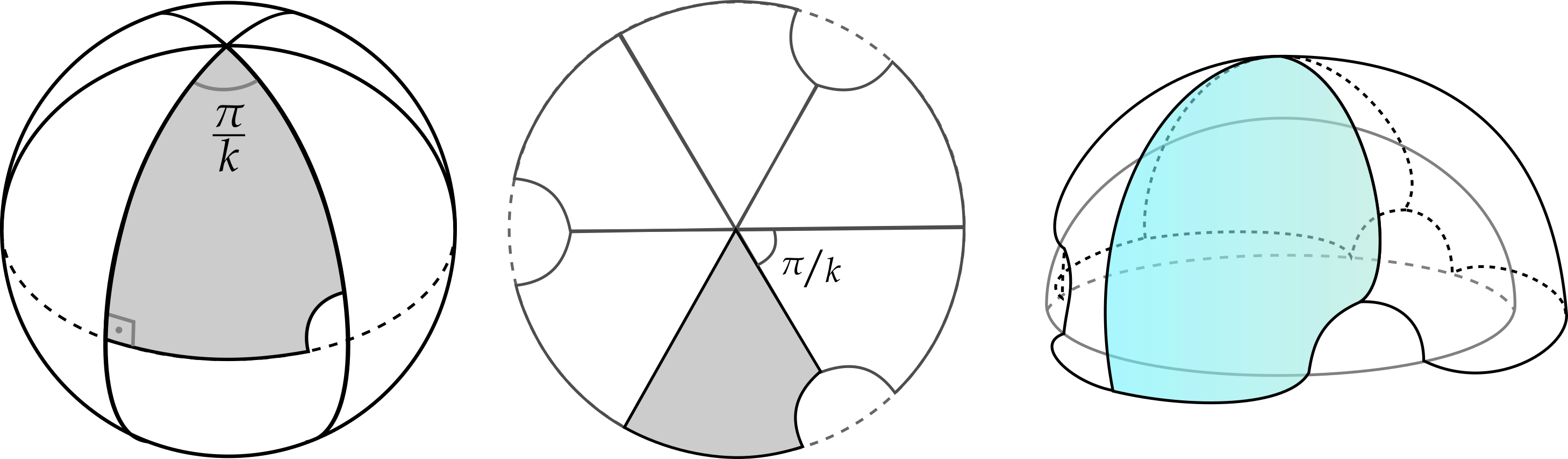}
\caption{From left to right for $k = 3$: $(2,k)$-tessellation of the sphere (see \S\ref{subsubsec:tessellations}), view of the $(2,k)$-tessellation from the north pole, and a sketch of $\Sigma_{H,g}$ ($g = 2$) in the conformal model $\mathbb{R}^3\setminus\{0\}$ of $\mathbb{S}^2(\kappa)\times\mathbb{R}$ (right). The shaded area in the first two figures is the projection of $\Sigma_{H,g}$ over its slice of symmetry $\mathbb{S}^2(\kappa) \times \{0\}$.}

\label{fig:arbitrary-genus-H-surface-fundamental-patch}
\end{figure}

\begin{remark}\label{rmk:arbitrary-genus-H-surface-tessellation-platonic-solids}
It is possible to obtain a similar result for the rest of tessellations of the sphere associated with regular polyhedra. However, just a few genera can be recovered in this way, plus different restrictions for the mean curvature appear depending on the tessellation (see Table~\ref{tb:tessellations-sphere} and \cite[Thm.~1.1]{MT20a}).

The same construction can be carried out for $H$-surfaces in $\mathbb{M}^2(\kappa)\times \mathbb{R}$ with $\kappa\leq 0$ and an arbitrary regular tessellation of $\mathbb{M}^2(\kappa)\times\mathbb{R}$ by regular polygons~\cite[Thm.~1.1]{MT20a}. In the case $\kappa=0$, our construction reduces to Lawson's doubly periodic $1$-surfaces in Euclidean space $\mathbb{R}^3$~\cite[Thm.~9]{Law} (see also~\cite[\S3]{G}). In the case $\kappa<0$, we obtain new properly immersed surfaces of subcritical, critical and supercritical constant mean curvature in a slab of $\mathbb{H}^2(\kappa)\times\mathbb{R}$, though we have not been able to analyze their embeddedness.
\end{remark}

\subsubsection{Regular tessellations}\label{subsubsec:tessellations}
Given integers $m,k\geq 2$, a regular $(m,k)$-tessellation is a tiling of $\mathbb{M}^2(\kappa)$ by regular $m$-gons such that $k$ of them meet at each vertex, see Figure~\ref{fig:arbitrary-genus-H-surface-fundamental-patch} for an example of a $(2,3)$-tessellation of the sphere. The centers and vertexes of the $m$-gons will be called the centers and vertexes of the tessellation. An straightforward application of the Gau\ss--Bonnet formula to one of the $m$-gons reveals that the sign of $\frac{1}{k} + \frac{1}{m} - \frac{1}{2}$ agrees with the sign of $\kappa$, in which case the $(m,k)$-tessellation of $\mathbb{M}^2(\kappa)$ actually exists. 

In the case of $\mathbb{S}^2(\kappa)$, the inequality $\frac{1}{k} + \frac{1}{m}>\frac{1}{2}$ is quite restrictive, for it only allows the tessellations associated to the Platonic solids and also two infinite families: the \emph{beach ball} tessellations (with $m = 2$ and arbitrary $k$) and the \emph{degenerated} case for $k = 2$ and arbitrary $m$). The possible configurations are shown in Table~\ref{tb:tessellations-sphere}. Notice that the $(m,k)$ and the $(k,m)$-tessellation are dual by swapping centers and vertexes and hence they have the same isometry group.

\begin{table}[htbp]
\begin{center}
\begin{tabular}{cccc} \toprule
Initial tessellation&$(m,k)$&$\alpha(m,k)$&Genus of $\Sigma$ \\\midrule
\emph{Beach ball}&$(2,g+1)$&$1$&$g$\\
\emph{Degenerated}&$(g+1,2)$&$\cot^2(\frac{\pi}{2+2g})$&$1$\\\midrule
Tetrahedron&$(3,3)$&$2+\sqrt{3}$&$3$\\
Hexahedron&$(4,3)$&$5+2\sqrt{6}$&$5$\\
Octahedron&$(3,4$)&$3+2\sqrt{2}$&$7$\\
Dodecahedron&$(5,3)$&$8+4 \sqrt{3}+3 \sqrt{5}+2 \sqrt{15}$&$11$\\
Icosahedron&$(3,5)$&$4+\sqrt{5}+2 \sqrt{5+2 \sqrt{5}}$&$19$\\\bottomrule
\end{tabular}
\end{center}
\caption{For each $(m,k)$-tessellation,~\cite[Thm.~1.1]{MT20a} gives compact orientable $H$-surfaces in $\mathbb{S}^2(\kappa)\times\mathbb{R}$ provided that $0<\frac{4H^2}{\kappa} < \alpha(m,k)$, see also Remark~\ref{rmk:arbitrary-genus-H-surface-tessellation-platonic-solids}. Theorem~\ref{thm:arbitrary-genus-H-surface-embeddedness} corresponds to the case of a \emph{beach ball} tessellation.}\label{tb:tessellations-sphere}
\end{table}

In a $(m,k)$-tessellation of $\mathbb{S}^2(\kappa)$, each $m$-gon can be decomposed into $2m$ congruent triangles by joining the center with the vertexes and the midpoints of the sides. Each of these triangles has angles $\frac{\pi}{k}$, $\frac{\pi}{m}$ and $\frac{\pi}{2}$ and will be called a \emph{fundamental triangle} because the whole tessellation can be recovered by symmetries about its sides. The idea is to construct an $H$-bigraph in $\mathbb{S}^2(\kappa)\times\mathbb{R}$ symmetric with respect to $\mathbb{S}^2(\kappa)\times\{0\}$ and with a curve at height zero around each of the centers, see Figure~\ref{fig:arbitrary-genus-H-surface-fundamental-patch}. This gives a surface of genus the number of polygons of the tessellation minus one, see the last column of Table~\ref{tb:tessellations-sphere}, so we will will focus on the $(2,g+1)$-tessellation of $\mathbb{S}^2(\kappa)$ for $g \geq 2$, which leads to the proof of Theorem~\ref{thm:arbitrary-genus-H-surface-embeddedness} that will be sketched throughout this section.

\subsubsection{Construction of the minimal surface in $\Sb(4H^2+\kappa,H)$}\label{subsubsec:arbitrary-genus-H-surface-conjugate}

We will fix positive real numbers $H$ and $\kappa$, and the target genus $g\geq 2$ in what follows so we will omit the dependence on these data. Instead, we will take a parameter $0<\rho \leq \frac{\pi}{\sqrt{4H^2 + \kappa}}$ that will provide one degree of freedom to be used later in a continuity argument. The dependence on $\rho$ will be written in functional notation to make it clear that it represents an auxiliary parameter.

Consider a convex spherical triangle $\widetilde\Delta(\rho)\subset\mathbb{S}^2(4H^2+\kappa)$ with two angles $\frac{\pi}{g+1}$ and $\frac{\pi}{2}$ adjacent to a side of length $\rho$. This triangle defines a geodesic quadrilateral $\widetilde\Gamma(\rho)\subset\Sb(4H^2+\kappa,H)$ with three horizontal sides $\widetilde h_1$, $\widetilde h_2$ and $\widetilde h_3$ projecting to the sides of $\widetilde\Delta(\rho)$ (being $\rho$ the length of $\widetilde{h}_2$) and a vertical segment $\widetilde v$ joining the ends of $\widetilde{h}_1$ and $\widetilde{h}_3$. The vertexes of $\widetilde\Gamma(\rho)$ will be denoted by $\widetilde 1$, $\widetilde 2$, $\widetilde 3$ and $\widetilde 4$, as shown in Figure~\ref{fig:arbitrary-genus-H-surface-polygon} (top left).

The pair $(\widetilde\Delta(\rho),\widetilde\Gamma(\rho))$ is a Nitsche graph (see Definition~\ref{def:nitsche}) such that $W(\rho)=\pi^{-1}(\widetilde\Delta(\rho))$ is a mean-convex set, so Proposition~\ref{prop:plateau-existencia} ensures the existence and uniqueness of a minimal disk $\widetilde{\Sigma}(\rho)\subset W(\rho)$ with boundary $\widetilde{\Gamma}(\rho)$ for any $0 < \rho \leq \frac{\pi}{\sqrt{4H^2 + \kappa}}$, which is also a graph over the interior of $\widetilde\Delta(\rho)$. We will assume without loss of generality that the angle function $\nu_\rho$ of $\widetilde\Delta(\rho)$ is negative over the interior of $\widetilde\Delta(\rho)$. We can analyze $\nu_\rho$ by means of the boundary maximum principle with respect to $\partial W$ and by Strategy 2, discussed in \S\ref{subsubsec:angle-control}.

\begin{proposition}[{\cite[\S3.1]{MT20a}}]\label{prop:arbitrary-genus-angle}
Let $\nu_{\rho}\leq 0$ be the angle function of the compact minimal disk $\widetilde{\Sigma}(\rho)$ spanning $\widetilde{\Gamma}(\rho)$.
\begin{enumerate}[label=(\alph*)]
  \item The only points at which $\nu_{\rho}$ vanishes are those in the curve $\widetilde v$. 

  \item  The only points at which $\nu_{\rho}$ takes the value $-1$ are $\widetilde 2$ and $\widetilde 3$.
\end{enumerate}
\end{proposition}

\subsubsection{The conjugate $H$-immersion}\label{subsubsec:arbitrary-genus-H-surface-compactness}

Thanks to Lemmas~\ref{lem:horizontal-geodesics} and~\ref{lem:vertical-geodesics}, the conjugate $H$-surface $\Sigma(\rho)$ in $\mathbb{S}^2(\kappa)\times\mathbb{R}$ is bounded by a contour $\Gamma(\rho)$ formed by the curves $h_1$, $h_2$, $h_3$ (corresponding to $\widetilde h_1$, $\widetilde h_2$, $\widetilde h_3$) contained in vertical planes, and $v$ (corresponding to $\widetilde v$) contained in a horizontal slice, which will be assumed to be $\mathbb{S}^2(\kappa)\times\{0\}$ after a vertical translation (see Figure~\ref{fig:arbitrary-genus-H-surface-polygon} top right). Their endpoints will be denoted by $1$--$4$, in correspondence with $\widetilde 1$--$\widetilde 4$.

Write $h_i=(\beta_i,z_i)\in\mathbb S^2(\kappa)\times\mathbb{R}$ for $i\in\{1,2,3\}$ as in Lemma~\ref{lem:horizontal-geodesics}. From the properties of $\nu_\rho$ in Proposition~\ref{prop:arbitrary-genus-angle}, it follows that the curves $\beta_i$ are one-to-one, and the height components $z_i$ are strictly monotonic. On the one hand, it is easy to show that $\beta_1$, $\beta_2$ and $\beta_3$ are part of the geodesics containing the sides of a spherical triangle $\Delta(\rho)\subset\mathbb{S}^2(\kappa)$ with two angles equal to $\frac{\pi}{g+1}$ and $\frac{\pi}{2}$, which coincide with the angles made by $h_1$ and $h_2$, and by $h_2$ and $h_3$, respectively, see Figure~\ref{fig:arbitrary-genus-H-surface-polygon} bottom right. On the other hand, $v$ is strictly convex as a curve of $\mathbb{S}^2(\kappa)\times\{0\}$ with respect to $-N$ as a conormal along $v$ by Lemma~\ref{lem:vertical-geodesics}. As a consequence, the curve $\pi\circ v$ is embedded and contained in $\Delta(\rho)$. Were it not the case, $\pi\circ v$ would intersect itself or other points of some $\beta_i$ producing a convex loop in $\pi(\Gamma(\rho))$. However, $\pi(\Sigma(\rho))$ must Alexandrov-embedded as a domain of $\mathbb{S}^2(\kappa)$ since $\pi$ restricted to $\Sigma(\rho)$ is an immersion (its Jacobian is the angle function $\nu_\rho$, which does not vanish in the interior). In \S\ref{subsubsec:arbitrary-genus-H-surface-embeddedness}, we will prove that $\Sigma(\rho)$ is a graph and confirm that Figure~\ref{fig:arbitrary-genus-H-surface-polygon} is a faithful picture of $\Gamma(\rho)$.

\begin{figure}[htbp]
\centering
\includegraphics{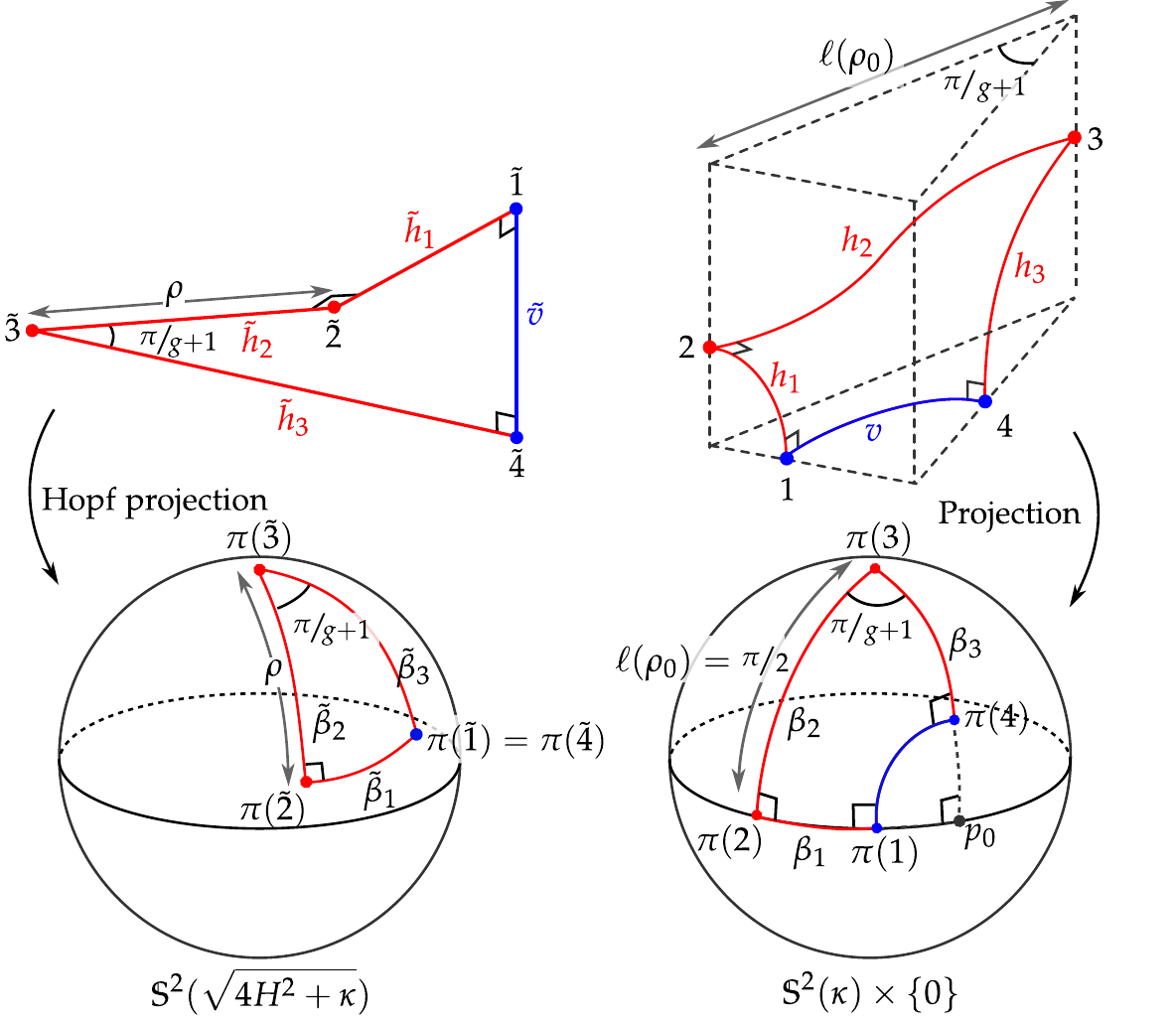}
\caption{Geodesic polygon $\widetilde\Gamma(\rho)$ (top left), its sister contour $\Gamma(\rho)$ (top right) in the target case $\ell(\rho_0)=\frac{\pi}{2\sqrt{\kappa}}$ (see \S\ref{subsubsec:arbitrary-genus-H-surface-compactness}) and their projections to $\mathbb{S}^2(4H^2+\kappa)$ via the Hopf fibration (bottom left) and to $\mathbb{S}^2(\kappa)$ (bottom right).}
\label{fig:arbitrary-genus-H-surface-polygon}
\end{figure}

Proposition~\ref{prop:conjugation-completion-by-simmetries} allows us to produce a complete $H$-surface $\Sigma^*(\rho)$ by successive mirror symmetries about the totally geodesic surfaces containing $h_1$, $h_2$, $h_3$ and $v$. The symmetry group of $\Sigma^*(\rho)$ depends on the shape of the triangle $\Delta(\rho)$, which is determined by the distance $\ell(\rho)$ from $\pi(2)$ to $\pi(3)$ given by
\begin{equation}\label{eq:arbitrary-genus-H-surface-ell}
  \ell(\rho) = - \int_{h_2} \nu_\rho = -\int_{\widetilde{h}_2} \nu_\rho,
\end{equation}
since the angle function $\nu_\rho$ has been chosen negative. Observe that $\widetilde{\Gamma}(\rho)$ depends continuously on $\rho$, so the uniqueness in the Plateau problem (solved by $\widetilde{\Sigma}(\rho)$) and Proposition~\ref{prop:continuity} yield that $\rho\mapsto\ell(\rho)$ is a continuous function. This enables a continuity argument to prove that there always exists a value of $\rho_0$ such that $\ell(\rho_0) = \frac{\pi}{2\sqrt{\kappa}}$ (see Figure~\ref{fig:arbitrary-genus-H-surface-polygon} bottom right), in which case $\Sigma^*(\rho_0)$ is compact since it has the symmetries of the $(2,g+1)$-tessellation. Let us analyze the limit cases:
\begin{itemize}
  \item As $\rho \to 0$, the length of $\widetilde{h}_2$ converges to zero, so $\ell(\rho)$ gets arbitrarily close to zero by Equation~\eqref{eq:arbitrary-genus-H-surface-ell}.
  
  \item As $\rho \to \frac{\pi}{\sqrt{4H^2 + \kappa}}$, the surface $\widetilde{\Sigma}(\rho)$ converges to $\frac{1}{4(g+1)}$ of the horizontal umbrella centered at $\widetilde{3}$, whence $\Sigma(\rho)$ is a sector of angle $\frac{\pi}{g+1}$ of the upper half of an $H$-sphere $S_{H,\kappa,0}$ in $\mathbb{S}^2(\kappa)\times \mathbb{R}$, see Example~\ref{ex:umbrellas}. This means that $\ell(\frac{\pi}{\sqrt{4H^2 + \kappa}})$ is the radius of the spherical circle of $\mathbb{S}^2(\kappa)$ over which $S_{H,\kappa,0}$ is a bigraph, that is, $\ell(\frac{\pi}{\sqrt{4H^2 + \kappa}})=\frac{2}{\sqrt{\kappa}} \arctan\frac{\sqrt{\kappa}}{2H}$, cp.\ Equation~\eqref{thm:embeddedness:eqn1}.
\end{itemize}
By the intermediate value theorem, $\ell(\rho)$ takes all values in $(0, \frac{2}{\sqrt{\kappa}}\arctan\tfrac{\sqrt{\kappa}}{2H})$, though it might take each value more than once. We finish the argument by realizing that this interval contains the target value $\frac{\pi}{2\sqrt{\kappa}}$ if and only if $H < \tfrac{\sqrt{\kappa}}{2}$. Observe that we can discuss the topology of the complete surface $\Sigma^*(\rho_0)$ analytically by means of Gau\ss--Bonnet formula if $\ell(\rho_0)=\frac{\pi}{2\sqrt{\kappa}}$: since $8(g+1)$ copies of $\Sigma(\rho_0)$ are needed to get a compact surface and the total curvature of each piece is $\int_{\Sigma(\rho_0)} K = \frac{\pi}{g+1} - \frac{\pi}{2}$, we easily deduce that the genus of $\Sigma^*(\rho_0)$ is $g$.

If $H \to \tfrac{\sqrt{\kappa}}{2}$, then $\frac{2}{\sqrt{\kappa}}\arctan\tfrac{\sqrt{\kappa}}{2H}\to\frac{\pi}{2\sqrt{\kappa}}$, which forces $\rho_0\to\frac{\pi}{\sqrt{4H^2+\kappa}}$ and the constructed surface $\Sigma(\rho_0)$ converges to a subset of an $\frac{\sqrt{\kappa}}{2}$-sphere, so $\Sigma^*(\rho_0)$ becomes a pair of tangent $\frac{\sqrt{\kappa}}{2}$-spheres. If $H \to 0$, then we will immerse $\widetilde\Sigma_{\rho_0}$ immersed in the local model $M(4H^2 + \kappa, H)$ via the isometry in~\eqref{eq:local-isometry-Daniel-Berger}. As the bundle curvature of $M(4H^2+\kappa,H)$ tends to zero, so does $\Length(\widetilde{v}) = 2H \mathrm{Area}(\widetilde{\Delta}(\rho))$ (see the discussion about the geometric meaning of the bundle curvature in \S\ref{subsec:working-coordinates}). The maximum principle and the stability of the piece $\widetilde{\Sigma}(\rho_0)$ imply that its angle function $\nu_{\rho_0}$ converges uniformly to $-1$, so the conjugate piece $\Sigma(\rho_0)$ becomes in the limit a slice (with singularities at the vertexes of the tessellation). We remark that we have used the model $M(4H^2 + \kappa, H)$ instead of the Berger sphere $\Sb(4H^2 + \kappa, H)$ to study the limit due to the fact that, as $H\to 0$, the Berger sphere collapses onto $\mathbb{S}^2(\kappa)$ whilst the Cartan model smoothly converges to $\mathbb{S}^2(\kappa) \times \mathbb{R}$.

\subsubsection{Embeddedness}\label{subsubsec:arbitrary-genus-H-surface-embeddedness}

Finally, we will show that the constructed fundamental piece $\Sigma(\rho_0)$ is actually a graph over some domain of $\mathbb{S}^2(\kappa)$ (so there are no self-intersect\-ions of type I, see \S\ref{sec:completion-embeddedness}) and lies in the prism $\Delta(\rho_0) \times \mathbb{R}$ (so there are no self-intersections of type II, see \S\ref{sec:completion-embeddedness}), where $\Delta(\rho_0)$ is the desired triangle with angles $\frac{\pi}{g+1}$, $\frac\pi2$ and $\frac{\pi}{2}$ . This clearly implies that the completion $\Sigma^*(\rho_0)$ is embedded, see Figure~\ref{fig:arbitrary-genus-H-surface-polygon} right. To this end, a standard application of the maximum principle reveals that it suffices to show that $\Gamma(\rho_0)$ is a graph and lies in $\Delta(\rho_0)\times \mathbb{R}$. After our discussion in~\S\ref{subsubsec:arbitrary-genus-H-surface-conjugate}, it will be enough to prove that $\pi(1)$ and $\pi(4)$ lie in $\Delta(\rho_0)$, i.e., the geodesics $\beta_1$ and $\beta_3$ do not reach the point $p_0$ shown in Figure~\ref{fig:arbitrary-genus-H-surface-polygon}. 

In the case of $\pi(1)$, this is a consequence of the fact that the distance from $\pi(2)$ to $p_0$ is $\frac{\pi}{(g+1)\sqrt{\kappa}}$, whereas the length of $\beta_1$ can be estimated as:
\begin{equation}\label{eq:arbitrary-genus-H-surface-estimate-pi1}
  \mathrm{Length}(\beta_1)=-\int_{\widetilde{h}_1}\nu\leq\mathrm{Length}(\widetilde h_1)\leq\frac{\pi}{(g+1)\sqrt{4H^2+\kappa}}<\frac{\pi}{(g+1)\sqrt{\kappa}}.
\end{equation}
The last inequality in~\eqref{eq:arbitrary-genus-H-surface-estimate-pi1} follows from the fact that $\widetilde h_1$ has maximum length when $\widetilde h_2$ is a quarter of a great circle of $\mathbb{S}^2(4H^2+\kappa)$. As for $\pi(4)$, since varying $\rho$ in the construction produces a foliation of a region of the Berger sphere $\Sb(4H^2+\kappa,4H)$, this implies that $\nu_\rho$ along $\beta_3$ depends monotonically on $\rho$. The maximum value of $\mathrm{Length}(\beta_3)=-\int_{\widetilde{h}_3}\nu$ is thus attained when $\rho=\frac{\pi}{\sqrt{4H^2+\kappa}}$, since for this value we integrate the largest function on the largest interval. In particular, $\Length(\beta_3)$ is less than $\frac{\pi}{2\sqrt{\kappa}}$, the radius of the disk over which the $\frac{\sqrt{\kappa}}{2}$-sphere is a bigraph. This means that $\pi(4)$ also lies in the boundary of $\Delta(\rho_0)$.

\subsection{Compact minimal surfaces in $\mathbb{S}^2(\kappa)\times\mathbb{S}^1(\eta)$}\label{sec:periodic}

The last construction of this section concerns periodic minimal surfaces in $\mathbb{S}^2(\kappa)\times\mathbb{R}$ that are compact in the quotient by a vertical translation of a certain length $2h$, i.e., in the homogeneous $3$-manifold $\mathbb{S}^2(\kappa) \times \mathbb{S}^1(\eta)$, where $\mathbb{S}^1(\eta)$ is a circle of curvature $\eta=\frac{\pi}{h}$. 

Hoffman, Traizet and White~\cite[Thm.~1]{HTW} obtained a class of properly embedded minimal surfaces in $\s^2(\kappa)\times\R$ called \emph{periodic genus $g$ helicoids} as well as infinitely-many non-congruent compact orientable embedded minimal surfaces in $\s^2(\kappa)\times\s^1(\eta)$ with arbitrary genus $g\geq 2$ and arbitrary $\eta>0$~\cite[Thm.~2]{HTW}. Note that, as we are dealing with minimal surfaces, non-orientable examples might also exist. Rosenberg~\cite[\S4]{Rosenberg} constructed non-orientable compact minimal surfaces in $\s^2(\kappa)\times\s^1(\eta)$, for all $\eta>0$, with even Euler characteristic, and in~\cite{MPT} it is proved that there cannot be examples with odd Euler characteristic.

We will obtain other compact minimal surfaces of arbitrary genus $g\geq 3$ that can be thought of as Schwarz P-surfaces in $\mathbb{S}^2(\kappa)\times \mathbb{S}^1(\eta)$, see Figure~\ref{fig:minimal-S2xS1-arbitrary-genus}. 
As in the constructions in~\S\ref{sec:delaunay} and~\S\ref{sec:genus}, we will produce a fundamental piece fitting a tile of a tessellation of $\mathbb{S}^2(\kappa)\times \mathbb{R}$ giving the desired genus (in the quotient) after extending it by symmetries about its boundary components. The conjugate technique starts with another minimal surface in $\mathbb{S}^2(\kappa)\times\R$ and involves a continuity argument that only allows us to get the result for $\eta$ large enough (see \S\ref{subsubsec:P-Schwarz-compactness}). Moreover, in principle this family might fail to be continuous as in~\S\ref{sec:delaunay}.

The main result is the following theorem (see also Figure~\ref{fig:minimal-S2xS1-arbitrary-genus}).

\begin{theorem}[{\cite[Prop.~3]{MPT}}]\label{thm:orientable-minimal-arbitrary-genus-examples-S2xS1}
For any integer $g\geq 3$ and $\eta > 2\sqrt{\kappa}$ big enough depending on $g$ and $\kappa$, there exists a compact embedded orientable minimal surface $\Sigma_{g,\eta}$ with genus $g$ in $\mathbb{S}^2(\kappa)\times\mathbb{S}^1(\eta)$ that is a bigraph over a slice and inherits all the symmetries of a $(2,g-1)$-tessellation of $\mathbb{S}^2(\kappa)$ (see~\S\ref{subsubsec:tessellations}).
\end{theorem}

\begin{figure}[htbp]
  \centering
  \includegraphics{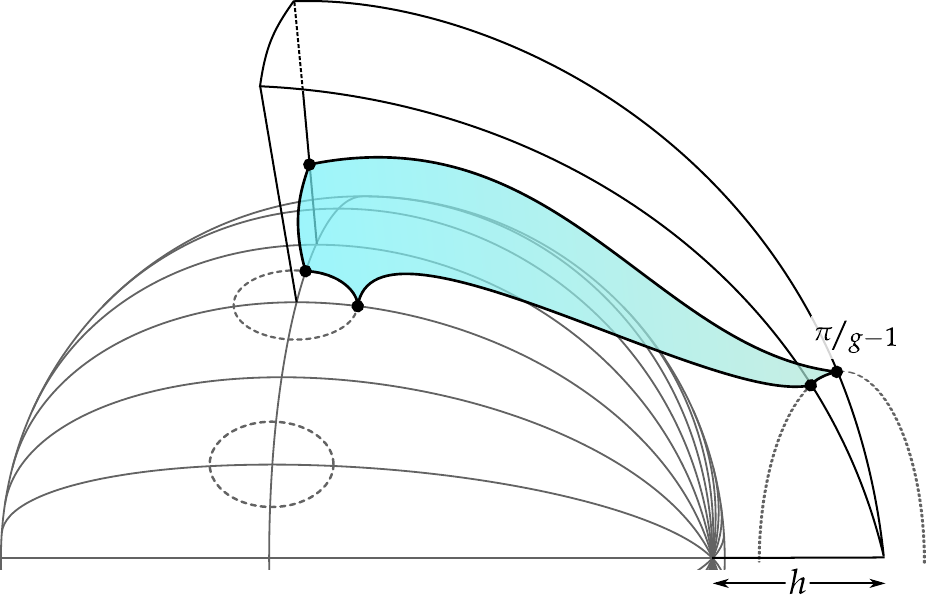}
  \caption{A sketch of the fundamental piece of the minimal surface $\Sigma_{g,\eta}$ of Theorem~\ref{thm:orientable-minimal-arbitrary-genus-examples-S2xS1} represented in the conformal model $\mathbb{R}^3-\{(0,0,0)\}$ of $\mathbb{S}^2(\kappa) \times \mathbb{R}$. See also the numerical approximations in Figure~\ref{fig:evolver-genus}.}
  \label{fig:minimal-S2xS1-arbitrary-genus}
\end{figure}

\begin{remark}
In the cases $g=0$ and $g=1$, there is no such a Schwarz P-surface, but in these cases we can easily produce a compact minimal surface of genus $g$ by considering the $\mathbb{S}^2(\kappa)\times \{0\}$ or a vertical torus. The case $g=2$ is more subtle, since the arguments in the proof of Theorem~\ref{thm:orientable-minimal-arbitrary-genus-examples-S2xS1} fail, and we believe that no such Schwarz P-surface with these prescribed symmetries will exist. Recall that Hofmann, Traizet and White have provided compact embedded examples of all genera. We also presented a similar but more symmetric construction in~\cite[\S3.2]{MPT} to produce compact minimal surfaces in $\mathbb{S}^2(\kappa)\times \mathbb{S}^1(\eta)$ with arbitrary odd genus $2k-1$ for $k \geq 2$ and $\eta$ large enough, which inherit the symmetries of a $(2,k)$-tessellation but they are not congruent those in Theorem~\ref{thm:orientable-minimal-arbitrary-genus-examples-S2xS1}. Besides, a genus $7$ compact minimal example in $\mathbb{S}^2(\kappa)\times \mathbb{S}^1(\eta)$ with the symmetries of the $(4,3)$-tessellation (that is, the hexahedron, see Table~\ref{tb:tessellations-sphere}) is also constructed.

On the other hand, the strange condition $\eta>2\sqrt{\kappa}$ shows up in the spherical geometry when one tries to produce the initial geodesic polygon, so there cannot be a Schwarz P-surface in $\mathbb{S}^2(\kappa)\times \mathbb{S}^1(\eta)$ with $\eta\leq2\sqrt{\kappa}$ enjoying the symmetries we have prescribed.
\end{remark}

\subsubsection{Construction of the minimal surface in $\mathbb{S}^2(\kappa)\times\R$}\label{subsubsec:P-Schwarz-conjugate-construction} We will fix $\eta>0$ and the genus $g\geq 3$ in the sequel, not writing the dependence on these variables. Consider three real parameters $0\leq\widetilde{a}, \widetilde{b} \leq \frac{\pi}{2\sqrt{\kappa}}$ and $\rho>0$ and define the geodesic triangle $\widetilde\Delta(\widetilde a,\widetilde b) \subset \mathbb{S}^2(\kappa) \times \{0\}$ with two sides of lengths $\widetilde{a}$ and $\widetilde{b}$ meeting at an angle of $\frac{\pi}{2}$. The opposite angles will be denoted by $\widetilde\alpha$ and $\widetilde\beta$. We can produce a geodesic polygon $\widetilde{\Gamma}(\widetilde a,\widetilde b,\rho)$ by adding two vertical geodesic segments of length $\rho$ at the vertexes $\widetilde\alpha$ and $\widetilde\beta$. Therefore, $\widetilde{\Gamma}(\widetilde a,\widetilde b,\rho)$ is a closed curve in $\mathbb{S}^2(\kappa)\times\mathbb{R}$ whose vertexes will be denoted by $\widetilde{1}$, $\widetilde{2}$, $\widetilde{3}$, $\widetilde{4}$ and $\widetilde{5}$ as shown in Figure~\ref{fig:P-Schwarz-initial} left. Notice that $\widetilde{\Gamma}(\widetilde a,\widetilde b,\rho)$ is still well defined for $0\leq \widetilde{a}, \widetilde{b} \leq \frac{\pi}{\sqrt{\kappa}}$ not both of them equal to $\frac{\pi}{\sqrt{\kappa}}$ (a half of the length of a great circle of $\mathbb{S}^2(\kappa)$), though imposing the restriction $0\leq\widetilde{a}, \widetilde{b} \leq \frac{\pi}{2\sqrt{\kappa}}$ is essential for some of the arguments in the construction (observe that this restriction on $\widetilde a$ and $\widetilde b$ implies that the angles $\widetilde{\alpha}$ and $\widetilde{\beta}$ are at most $\frac{\pi}{2}$).

\begin{remark} 
One can also consider the boundary curve $\widetilde{\Gamma}(\widetilde a,\widetilde b,\rho)$ constructed likewise using horizontal and vertical geodesics in $\R^3$. This leads, by conjugation, to the classical Schwarz P-surface, whose name is motivated by the fact that it is invariant under a primitive cubic lattice $C$, when choosing $\widetilde a=\widetilde b=\rho$ and $\widetilde\alpha=\widetilde\beta=\frac{\pi}{4}$ (e.g., see~\cite[\S1.6.2]{K89a}). The quotient of the conjugate surface under the lattice $C$ has genus $3$ and consists of $16$ copies of the fundamental piece with $\alpha=\beta=\frac{\pi}{4}$ and $a=b$ (cp.\ Figure~\ref{fig:P-Schwarz-initial} right). Here, the actual lenghts of $\widetilde a$ and $a$ are not relevant, since we can change them by homotheties of $\R^3$. Note also that this construction in $\R^3$ can be seen as a limit of our construction as $\kappa\to 0$.
\end{remark}

By construction, the triangle $\widetilde\Delta(\widetilde{a}, \widetilde{b})$ and the polygon $\widetilde{\Gamma}(\widetilde{a}, \widetilde{b},\rho)$ form a Nitsche contour, so Proposition~\ref{prop:plateau-existencia} ensures the existence and uniqueness of a minimal disk $\widetilde{\Sigma}(\widetilde{a}, \widetilde{b},\rho)\subset\widetilde\Delta(\widetilde{a}, \widetilde{b})\times\R$ with boundary $\widetilde{\Gamma}(\widetilde{a}, \widetilde{b},\rho)$, whose interior is a graph over $\widetilde{\Gamma}(\widetilde{a}, \widetilde{b},\rho)$. We will assume that its angle function $\nu_{\widetilde{a}, \widetilde{b},\rho}$ is positive in the interior without loss of generality. By the Strategy 2 in~\S\ref{subsubsec:angle-control} and the boundary maximum principle we deduce that the angle function only vanishes along the vertical segments $\widetilde{12}$ and $\widetilde{34}$, whereas it only takes the value $1$ at the vertex $\widetilde{5}$.

\subsubsection{The conjugate minimal surface}\label{subsubsec:P-Schwarz-embeddedness}
Let $\Sigma(\widetilde{a}, \widetilde{b},\rho)$ be the conjugate minimal surface in $\mathbb{S}^2(\kappa)\times\mathbb{R}$ whose boundary $\Gamma(\widetilde{a}, \widetilde{b},\rho)$ consists of curves contained in vertical or horizontal totally geodesic surfaces that $\Sigma$ meets orthogonally, see Lemmas~\ref{lem:horizontal-geodesics} and~\ref{lem:vertical-geodesics}. We will denote by $1$, $2$, $3$, $4$ and $5$ the corresponding vertexes of $\Gamma(\widetilde{a}, \widetilde{b},\rho)$ and by $P_{ij}$ the vertical plane or horizontal slice containing $i$ and $j$. Up to a vertical translation we will assume that $P_{34} = \mathbb{S}^2(\kappa)\times \{0\}$. 

We will sketch the proof that $\Sigma(\widetilde a,\widetilde b,\rho)$ is embedded and contained in a triangular prism of $\mathbb{S}^2(\kappa)\times \mathbb{R}$ (see~\cite[Lem.~3]{MPT} for further details). To prove this we will additionally assume that $\rho\leq\frac\pi{2\sqrt{\kappa}}$.

\begin{figure}[htbp]
  \centering
  \includegraphics[width=0.7\textwidth]{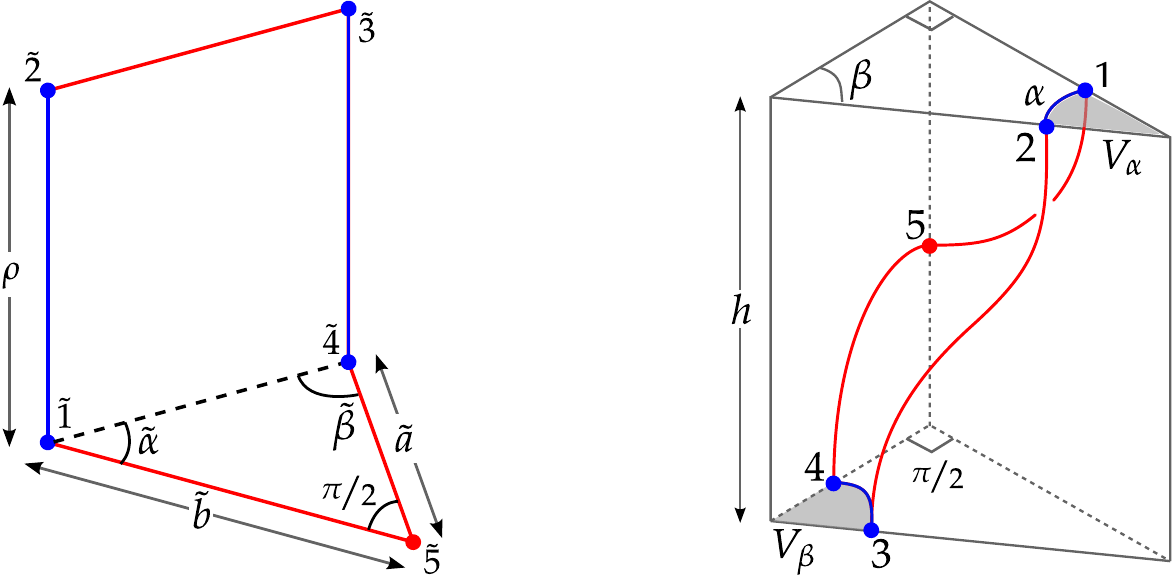}
\caption{The polygon $\widetilde{\Gamma}(\widetilde{a}, \widetilde{b}, \rho)$ (left) and its conjugate $\Gamma(\widetilde{a}, \widetilde{b},\rho)$ (right) that lies inside the prism demarcated by the symmetry planes. Note that by construction $h \leq \ell(23) = \ell(\widetilde{23}) \leq \tfrac{\pi}{2\sqrt{\kappa}}$ so we get the restriction $\eta = \frac{\pi}{h} \geq 2\sqrt{\kappa}$ in Theorem~\ref{thm:orientable-minimal-arbitrary-genus-examples-S2xS1}. The restriction $\rho \leq \frac{\pi}{2\sqrt{\kappa}}$ is unrelated to this and will be used to prove embeddedness. Moreover, $\alpha = \mathrm{area}(V_\alpha) + \widetilde{\alpha}$ and $\beta = \mathrm{area}(V_\beta) + \widetilde{\beta}$.}
  \label{fig:P-Schwarz-initial}
\end{figure}

Observe that $P_{45}$ and $P_{51}$ are different since they form an angle of $\frac{\pi}{2}$. Moreover, $P_{12}$ and $P_{34}$ do not coincide (so they are at a certain distance $h > 0$) and the height of the point $5$ is between $0$ and $h$; otherwise, we would find a contradiction with the maximum principle when comparing with slices $\mathbb{S}^2(\kappa)\times \{t\}$. Due to the behavior of the angle function, the curves ${23}$, ${45}$ and ${51}$ project one-to-one to $\mathbb{S}^2(\kappa)\times \{0\}$. The normal to $\widetilde\Sigma(\widetilde a,\widetilde b,\rho)$ has different directions of rotation along $\widetilde{12}$ and $\widetilde{34}$, so the curves ${12}$ and $34$ are convex arcs that $\Sigma(\widetilde a,\widetilde b,\widetilde h)$ meets from above and below, respectively, by Lemma~\ref{lem:vertical-geodesics}. This easily leads to the fact that $P_{45}$, $P_{23}$ and $P_{15}$ demarcate a spherical triangle $\Delta(\widetilde a,\widetilde b,\rho)$ in the projection to $\mathbb{S}^2(\kappa)$, such that $\Sigma(\widetilde a,\widetilde b,\rho)\subset\Delta(\widetilde a,\widetilde b,\rho)\times\R$. We will call $\alpha$ and $\beta$ the angles of this triangle, as indicated in Figure~\ref{fig:P-Schwarz-initial}. 

It remains to prove that ${12}$ and ${34}$ are embedded and $P_{23}$, $P_{15}$ and $P_{45}$ are pairwise distinct. We will explain the argument for ${12}$, being the case ${34}$ analogous. The total curvature of ${12}$ is $\int_{{12}} \kappa_g=-\widetilde{\alpha}$, i.e., minus the total rotation of the normal along $\widetilde{12}$ (note that $\kappa_g$ is computed with respect to the normal $N$ such that $\nu>0$ in the interior). If ${12}$ were not embedded then, by convexity, it would contain a loop enclosing a domain $D$, where Gau\ss--Bonnet formula yields $\kappa\Area(D) \geq \pi - \widetilde{\alpha} \geq \frac{\pi}{2}$. On the other hand, $\Length(\partial D)\leq\rho \leq \frac{\pi}{2\sqrt{\kappa}}$ by assumption, which contradicts the isoperimetric inequality in $\s^2(\kappa)$ that reads $4\pi\Area(D)-\kappa\Area(D)^2\leq\Length(\partial D)^2$. 

Once we know that both curves ${12}$ and ${34}$ are embedded, they must enclose some domains $V_\alpha$ and $V_\beta$ as in Figure~\ref{fig:P-Schwarz-initial}, demarcated by the horizontal geodesics orthogonal at their endpoints. Gau\ss--Bonnet formula applied to these domains $V_\alpha$ and $V_\beta$ gives the relations $\alpha = \mathrm{area}(V_\alpha) + \widetilde{\alpha}$ and $\beta = \mathrm{area}(V_\beta) + \widetilde{\beta}$.  The same argument as above shows that the angles $\alpha$ and $\beta$ do not exceed $\pi$ whenever $\rho \leq \frac{\pi}{2\sqrt{\kappa}}$, so the planes $P_{23}$, $P_{45}$ and $P_{51}$ are pairwise different.

\subsubsection{Compactness}\label{subsubsec:P-Schwarz-compactness}

As we have mentioned in the introduction, our goal to prove Theorem~\ref{thm:orientable-minimal-arbitrary-genus-examples-S2xS1} is to fit $\Sigma(\widetilde a,\widetilde b,\rho)$ in the prism with base half of a $2$-gon in the $(2,g-1)$-tessellation of the sphere $\mathbb{S}^2(\kappa)\times \{0\}$ (see Figure~\ref{fig:minimal-S2xS1-arbitrary-genus}). This amounts to saying that $\alpha = \frac{\pi}{g-1}$ and $\beta = \frac{\pi}{2}$ (see Figure~\ref{fig:P-Schwarz-initial} right). To prove that these conditions are satisfied, we will use a degree argument inspired by the work of Karcher, Pinkall and Sterling~\cite{KPS88}. 

We will assume that $0<\rho<\frac{\pi}{2\sqrt{\kappa}}$ is fixed, so the conjugate minimal surface $\Sigma(\widetilde a,\widetilde b,\rho)$ continuously depends on $\widetilde{a}$ and $\widetilde{b}$ (see Proposition~\ref{prop:continuity}), whence there exists a continuous function $f_{\rho}: (0, \frac{\pi}{2\sqrt{\kappa}}]\times (0, \frac{\pi}{2\sqrt{\kappa}}]\to \mathbb{R}^2$ such that $f_{\rho}(\widetilde{a}, \widetilde{b})=(\alpha, \beta)$. Let $c_1, c_2, c_3$ and $c_4$ be four straight segments parametrized by
\begin{align*}
  c_1(t) &=\tfrac{1}{\sqrt{\kappa}}(t,\tfrac{\pi}{2}),\quad t\in[\tfrac{1}{2(g-1)},\tfrac{\pi}{g-1}],
        &
  c_2(t) &= \tfrac{1}{\sqrt{\kappa}}(\tfrac{\pi}{g-1},t),\quad t\in[\tfrac12,\tfrac\pi2], \\
  c_3(t) &= \tfrac{1}{\sqrt{\kappa}}(t,\tfrac12),\quad t\in[\tfrac{1}{2(g-1)},\tfrac\pi{g-1}],
        &
  c_4(t) &= \tfrac{1}{\sqrt{\kappa}}(\tfrac1{2(g-1)},t),\quad t\in[\tfrac12,\tfrac\pi2],
\end{align*}
which form a closed rectangle $R\subset (0, \frac{\pi}{2\sqrt{\kappa}}]\times (0, \frac{\pi}{2\sqrt{\kappa}}]$. We claim that there exists $\rho>0$ such that the image of $f_{\rho}(R)$ is a closed curve around $(\frac{\pi}{g-1}, \frac{\pi}{2})$. 

By spherical trigonometry, it is easy to obtain that $\widetilde{\beta} = \frac{\pi}{2}$ along $c_1$, $\widetilde{\alpha} > \frac{\pi}{g-1}$ along $c_2$, $\widetilde{\beta} < \frac{\pi}{2}$ along $c_3$, and $\widetilde{\alpha} < \frac{\pi}{g-1}$ along $c_4$. Also, we know that $\alpha > \widetilde{\alpha}$ and $\beta > \widetilde{\beta}$ (see \S\ref{subsubsec:P-Schwarz-embeddedness}), so the image of $f_{\rho}\circ c_1$ is above the horizontal line at height $\frac{\pi}{2}$ and the image of $f_{\rho}\circ c_2$ is to the right of the vertical line at $\frac{\pi}{g-1}$. Finally, since $\alpha \to \widetilde{\alpha}$ and $\beta \to \widetilde{\beta}$ for $\rho\to 0$, we deduce that there exists $\rho_0$ small enough so that the image of $f_{\rho_0} \circ c_3$ is below the horizontal line at height $\frac{\pi}{2}$ and the image of $f_{\rho_0}\circ c_4$ is to the left of the vertical line at $\frac{\pi}{k}$ and the claim is verified. Observe that this argument holds true for any $\rho \leq \rho_0$.

By continuity, for any $\rho \leq \rho_0$, there exists $\widetilde{a}$ and $\widetilde{b}$ (that depend on $\rho$) such that $f_{\rho}(\widetilde{a}, \widetilde{b}) = (\frac{\pi}{g-1}, \frac{\pi}{2})$ so the polygon $\Gamma(\widetilde a,\widetilde b,\rho)$ fits in the $(2,g-1)$-tessellation. Then, the complete surface $\Sigma^*(\widetilde a,\widetilde b,\rho)$ obtained by successive reflections of $\Sigma$ inherits the desired symmetries, in particular it is compact in the quotient. Finally, a similar computation using the Gau\ss--Bonnet theorem as in \S\ref{subsubsec:arbitrary-genus-H-surface-compactness} ensures that quotient minimal surface in $\mathbb{S}^2(\kappa)\times \mathbb{S}^1(\frac{h}{\pi})$ has genus precisely $g$, $g \geq 3$. By the continuous dependence of $h$ on the parameters and taking into account that both $\widetilde\Sigma(\widetilde a,\widetilde b,\rho)$ and $\Sigma(\widetilde a,\widetilde b,\rho)$ converge to a horizontal slice as $\rho\to 0$, the desired compact surfaces exist for all $h$ small enough, i.e. for $\eta$ large enough.


\section{Complete $H$-surfaces in $\mathbb{H}^2(\kappa)\times\mathbb{R}$}\label{sec:non-compact}

This section is devoted to the construction of minimal and constant mean curvature $k$-noids and $k$-nodoids in $\mathbb{H}^2\times\mathbb{R}$ by conjugating the solution of a Jenkins--Serrin problem. The main difference with respect to the above constructions in \S\ref{sec:compact} is that additionally we have to deal with the asymptotic behaviour of the surface via the analysis of ideal horizontal and vertical geodesics in the initial minimal surface (see Proposition~\ref{prop:conjugation-JS}).

\subsection{Genus zero constant mean curvature $k$-noids and $k$-nodoids in $\mathbb{H}^2(\kappa)\times\mathbb{R}$}\label{sec:knoids}

The first construction concerns $H$-surfaces in $\mathbb{H}^2(\kappa)\times\R$ of surfaces of genus $0$ with an arbitrary number of ends $k\geq 3$ asymptotic to vertical $H$-cylinders and $4H^2+\kappa\leq 0$. Before stating the main theorem, it is worth saying that these surfaces, in the minimal case, reduce to the minimal $k$-noids and saddle towers constructed by Morabito and Rodríguez~\cite{MR} and also by Pyo~\cite{P} independently, whose names are inspired by their counterparts in $\R^3$ given by Jorge and Meeks~\cite{JM83} and Karcher~\cite{K88}. The $(H,k)$-noids in Theorem~\ref{th:knoids-genus-0} were constructed first by Plehnert in~\cite{Ple14} and by Daniel and Hauswirth~\cite{DH} for $k=2$. However, the main contribution of our theorem lies in the new family of $(H,k)$-nodoids, with a similar behaviour at infinity but approaching the asymptotic $H$-cylinders from the convex side. Some of them have self-intersections and give rise to counterexamples to the Krust property for subcritical $H$-surfaces.

The term catenodoid is inspired by the fact that the $H$-nodoids we have constructed in~\S\ref{sec:delaunay} seem to converge to a catenodoid of critical mean curvature as $H$ decreases to $\frac{\sqrt{-\kappa}}{2}$. In this limit we choose $\widetilde 2$ (see Figure~\ref{fig:horizontal-Delaunay-polygon-Berger}) as accumulation point and the parameter $\lambda>\frac\pi2$ should also be chosen appropriately along the sequence. Likewise, the limit of the $H$-unduloids as $H$ decreases to $\frac{\sqrt{-\kappa}}{2}$ seems to be a catenoid of critical mean curvature. The situation is similar if we look at rotationally invariant $H$-surfaces with subcritical and supercritical mean curvature (see~\cite[Fig.~A and~B]{OM}, noticing that Montaldo and Onnis choose $\kappa=-2$ and the mean curvature as the sum of the principal curvatures, not their average).

\begin{theorem}[{\cite[Thm.~1.2]{CMR}}]\label{th:knoids-genus-0}
Assume that $4H^2+\kappa\leq 0$. For each integer $k\geq 2$, there is a continuous $2$-parameter family of proper Alexandrov-embedded $H$-surfaces $\Sigma_{a,b}^*\subset\mathbb{H}^2(\kappa)\times\mathbb{R}$, $a,b\in(0,\infty]$ not both of them equal to $\infty$. These surfaces are invariant by mirror symmetries about a horizontal plane and about $k$ equiangular vertical planes.
\begin{enumerate}[label=\emph{(\alph*)}]
  \item If $a,b<\infty$, then $\Sigma_{a,b}^*$ are called \emph{saddle towers}. They are singly periodic, having genus $0$ and $2k$ ends in the quotient of $\h^2(\kappa)\times\R$ by a vertical translation, and each end is asymptotic in the quotient to a half of a $H$-cylinder. 
    
  \item If $a=\infty$ and $b<\infty$, then $\Sigma_{\infty,b}^*$ are called \emph{$(H,k)$-noids} (or $H$-catenoids if $k=2$). They have genus $0$ and $k$ ends. If $4H^2+\kappa<0$, then each end is embedded and contained in the \emph{concave} side of a $H$-cylinder to which it is asymptotic. If $4H^2+\kappa=0$, then each end is tangent to the asymptotic boundary $\partial_\infty\mathbb{H}^2(\kappa)\times\R$ along a vertical ideal geodesic. 
    
  \item If $a<\infty$ and $b=\infty$, then $\Sigma_{a,\infty}^*$ are called \emph{$(H,k)$-nodoids} (or $H$-catenodoids if $k=2$). They have genus $0$ and $k$ ends.  If $4H^2+\kappa<0$, then each end is embedded and contained in the \emph{convex} side of a $H$-cylinder to which it is asymptotic. If $4H^2+\kappa=0$, then such $H$-cylinders disappear at infinity. 
\end{enumerate}
\end{theorem} 

In the case $H=0$, the $(H,k)$-nodoids are congruent to the $(H,k)$-nodoids, but the subtle difference between these two families in the case $H>0$ will be transparent via conjugation. It relies on the key role of the orientation in $\E(4H^2+\kappa,H)$, whose bundle curvature is non-zero.

\subsubsection{The construction of the minimal surface in $\E(4H^2+\kappa,H)$}

We will fix $\kappa$ and $H$ such that $4H^2+\kappa\leq 0$, as well as the integer $k\geq 2$ in the sequel, so we will omit the dependence on these parameters as in previous constructions. Also, we will work on the global Cartan model $M(4H^2+\kappa,H)$ given in~\S\ref{subsec:working-coordinates}.

Given $a,b>0$, let $\widetilde T_{a,b}\subset\mathbb{M}^2(4H^2+\kappa)$ be a geodesic triangle of vertexes $\widetilde p_0=(0,0)$, $\widetilde p_1$ and $\widetilde p_2$ labeled counterclockwise such that the angle at $p_0$ is equal to $\frac{\pi}{k}$ and the geodesic segments $\widetilde{p}_0\widetilde{p}_1$ and $\widetilde{p}_0\widetilde{p}_2$ have lengths $a$ and $b$, respectively. After an orientation-preserving isometry, we will assume that $\widetilde p_1$ lies on the $x$-axis. We call $\ell$ the length of the side $\widetilde{p}_1\widetilde{p}_2$. We can extend this triangle to the case $a=\infty$ or $b=\infty$ by setting $\widetilde T_{\infty,b}=\mathrm{cl}(\cup_{a>0}\widetilde T_{a,b})$ and $\widetilde T_{a,\infty}=\mathrm{cl}(\cup_{b>0}\widetilde T_{a,b})$, respectively, where $\mathrm{cl}(G)$ denotes the topological closure of some subset $G\subset\M^2(4H^2+\kappa)$. Note that $\widetilde p_1$ or $\widetilde p_2$ become ideal if $4H^2+\kappa<0$ or just disappear if $4H^2+\kappa=0$ since $\Nil$ has not a notion of ideal boundary, and the triangle just becomes a truncated strip. Either way, we can lift the vertexes of $\widetilde T_{a,b}$ by means of the zero section as $\widetilde q_i=F_0(\widetilde p_i)=(\widetilde p_i,0)\in\mathbb{E}(4H^2+\kappa,H)$ for $i\in\{0,1,2\}$, and we denote by $\widetilde{q}_0\widetilde{q}_i$ the horizontal geodesic in $\mathbb E(4H^2+\kappa,H)$ joining $\widetilde q_0$ and $\widetilde q_i$.

\begin{figure}[htb]
\begin{center}
\includegraphics[width=\textwidth]{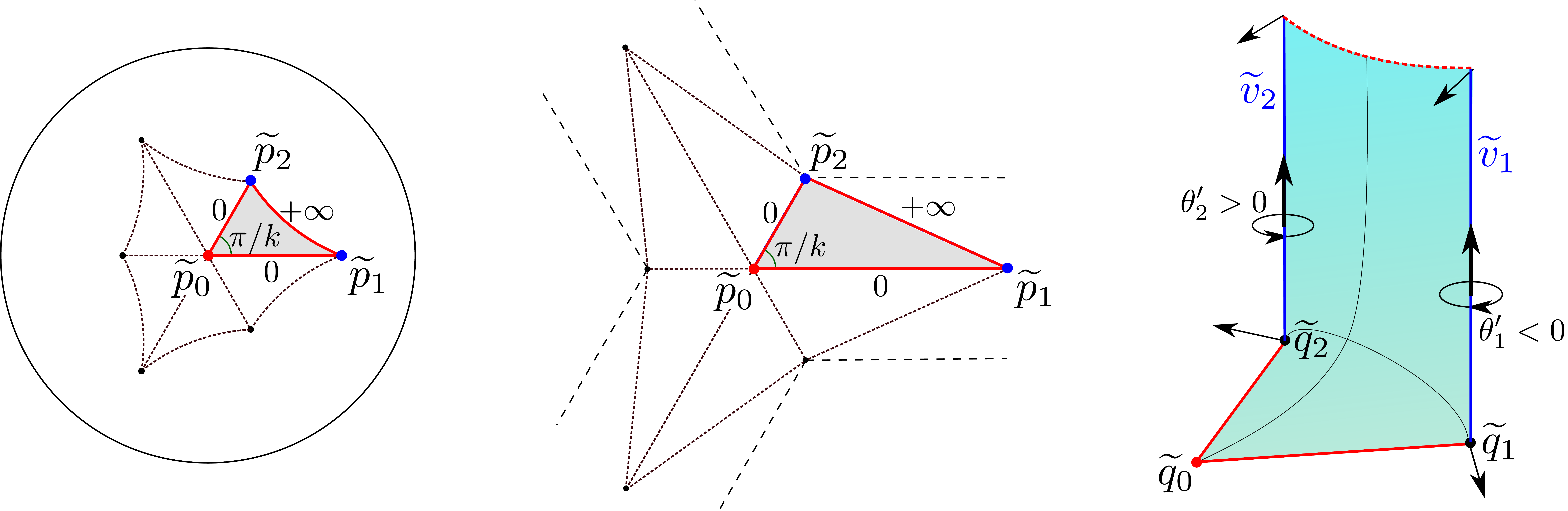}
\end{center}
\caption{The shaded triangle $\widetilde{T}_{a,b}$, $a,b<\infty$, and the Jenkins--Serrin boundary values in the subcritical case (left) and the critical case (center) for $k=3$. Sketch of the solution $\widetilde R_{a,b}$ and the rotation of the normal along $\widetilde v_1$ and $\widetilde v_2$ (right).}\label{fig:Hnoids-fundamental-domain}
\end{figure}

Our minimal fundamental piece is the vertical minimal graph $\widetilde R_{a,b}$ that solves the Jenkins--Serrin problem with prescribed boundary values $0$ over $\widetilde{p}_0\widetilde{p}_1$ and $\widetilde p_{0}\widetilde p_{2}$ and $+\infty$ over $\widetilde{p}_1\widetilde{p}_2$ (see Figure~\ref{fig:Hnoids-fundamental-domain}). We call $\widetilde\Sigma_{a,b}$ the vertical minimal graph obtained after successive reflections about the horizontal geodesics $\widetilde{q}_0\widetilde{q}_1$ and $\widetilde q_{0}\widetilde q_{2}$, which are contained in $\widetilde R_{a,b}$. We call $\widetilde \Omega_{a,b}$ the domain of $\mathbb M^2(4H^2+\kappa)$ onto which $\widetilde \Sigma_{a,b}$ projects, and we denote by $\widetilde p_i$, for $i\in\{1,\dots 2k\}$ the vertexes of $\widetilde\Omega_{a,b}$ labeled counterclockwise. Observe that the vertexes $\widetilde p_{2i-1}$ (resp.\ $\widetilde p_{2i}$) are ideal if $a=\infty$ (resp.\ $b=\infty$) when $4H^2+\kappa<0$ or disappear if $4H^2+\kappa=0$. The existence of the solution $\widetilde \Sigma_{a,b}$ is guaranteed by~\cite[Lem.~3.2 and~3.6]{CMR} (see also~\S\ref{sec:JS}).

\begin{remark}
  If $4H^2+\kappa=0$, the uniqueness of solution for the Jenkins--Serrin problem in the cases $a=\infty$ or $b=\infty$ is not clear. In that case, we take the graphs $\widetilde \Sigma_{\infty,b}$ and $\widetilde \Sigma_{a,\infty}$ as limit of the graphs $\widetilde \Sigma_{a,b}$ making the family continuous. The case $k=2$ is specially relevant, since the solutions $\widetilde\Sigma_{a,\infty}$ and $\widetilde\Sigma_{\infty,b}$ are explicit. They belong to a broader family $\mathcal H_\mu$ of minimal surfaces that are foliated by non-geodesic straight lines in $\Nil(\frac{\sqrt{-\kappa}}{2})$ orthogonal to a horizontal geodesic, having different directions of rotation if $\mu<\frac{-1}{2}$ or $\mu>\frac{1}{2}$, see Figure~\ref{fig:horizontal-helicoids}. These surfaces are described in~\cite[Lem.~3.3]{CMR} and resemble \emph{horizontal helicoids} with arbitrary distance between two consecutive vertical geodesics, so they can be used as barrier to solve the Jenkins--Serrin problem in a limit of a double sequence of minimal surfaces for $k>3$. The surfaces $\mathcal H_\mu$ with $\mu<-\frac{1}{2}$ solve the case $a=\infty$ and have been previously found by Daniel and Hauswirth~\cite[\S7]{DH} by other methods.
\end{remark}

\begin{figure}[htb]
  \includegraphics[width=0.9\textwidth]{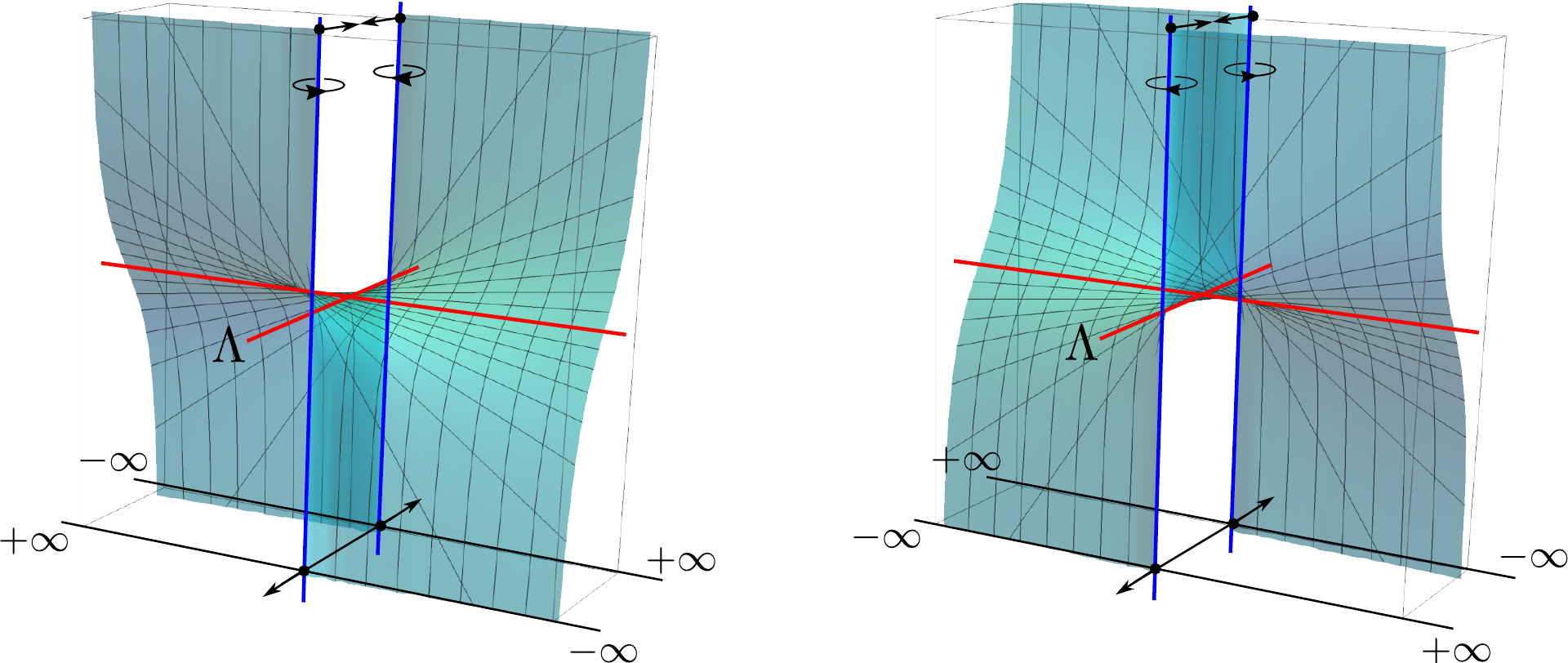}
  \caption{The surfaces $\widetilde\Sigma_{\infty,b}$ (left) and $\widetilde\Sigma_{a,\infty}$ (right) in the case $4H^2+\kappa=0$ and $k=2$ are horizontal helicoids in $\Nil(\frac{\sqrt{-\kappa}}{2})$ foliated by straight lines orthogonal to $\Lambda$.}\label{fig:horizontal-helicoids}
\end{figure} 

We will now analyze the angle function $\nu_{a,b}$ of $\widetilde\Sigma_{a,b}$. To this end, we can restrict to the fundamental piece $\widetilde R_{a,b}$, in whose interior $\nu_{a,b}$ will be assumed positive. We will only consider the case $a,b<\infty$ because, if $a=\infty$ or $b=\infty$, the surface $\widetilde\Sigma_{a,b}$ can be analyzed as a limit of the finite case. However, the following lemma is also expected to hold true in the cases $a=\infty$ or $b=\infty$. It is important to remark that the finite boundary of $\widetilde R_{a,b}$ consists of the vertical geodesics $\widetilde v_i=\pi^{-1}(\widetilde p_i)$ for $i\in\{1,\ldots,2k\}$, which are the only points where $\nu_{a,b}$ vanishes.

\begin{proposition}[{\cite[Lem.~4.4]{CMR}}]\label{prop:angle-H-knoids} 
  Let $a,b\in(0,\infty)$ and $H>0$, and define $\nu_{a,b}$ as the angle function of $\widetilde R_{a,b}\subset\mathbb{E}(4H^2+\kappa,H)$ that is positive in the interior of $\widetilde R_{a,b}$. 
  \begin{enumerate}[label=\emph{(\alph*)}]
    \item If $k=2$, then $\nu_{a,b}$ takes the value $1$ only at $\widetilde q_0$.
    
    \item If $k\geq 3$, then $\nu_{a,b}$ only takes the value $1$ at $\widetilde q_0$ and at other some point $\widehat q_1\in\widetilde q_0 \widetilde q_{1}$.
  \end{enumerate}
\end{proposition}

\subsubsection{The conjugate $H$-immersion}

Let $\Sigma_{a,b}$ be the conjugate $H$-surface in $\mathbb{H}^2(\kappa)\times\R$, which is a multigraph over a (possibly non-embedded, see Figure~\ref{fig:non-emb} right) domain $\Omega_{a,b}\subset\mathbb{H}^2(\kappa)$. Since $\widetilde \Sigma_{a,b}$ is invariant by axial symmetries about the geodesics $\widetilde{q}_i\widetilde{q}_{k+i}$ for $i\in\{1,\ldots,k\}$, Lemma~\ref{lem:horizontal-geodesics} says that $\Sigma_{a,b}$ has mirror symmetry with respect to $k$ vertical planes meeting at a common vertical line, say the $z$-axis, arranged symmetrically. Moreover, Lemma~\ref{lem:vertical-geodesics} and Proposition~\ref{prop:conjugation-JS} show that the boundary components of $\Sigma_{a,b}$ are the $2k$ complete (possibly ideal) horizontal curves $v_1,\dots,v_{2k}$ along with $2k$ ideal vertical geodesics joining the endpoints of $v_i$ and $v_{i+1}$ for $i\in\{1,\ldots,2k\}$. We will assume hereafter that $v_2\subset\h^2(\kappa)\times\{0\}$ (resp.\ $v_1\subset\h^2(\kappa)\times\{0\}$) if $a=\infty$ (resp.\ $b=\infty$), see Figures~\ref{fig:H-noids} and~\ref{fig:H-nodoids}.

\begin{enumerate}
  \item If $\widetilde p_i$ is not ideal, then let $\theta_i$ be the angle of rotation of $\widetilde N$ along the vertical geodesic $\widetilde v_i$. It satisfies $\theta_i'<0$ (resp.\ $\theta_i'>0$) if $i$ is odd (resp.\ even), see Figures~\ref{fig:Hnoids-fundamental-domain},~\ref{fig:H-noids} and~\ref{fig:H-nodoids}. Lemma~\ref{lem:vertical-geodesics} implies that $\kappa_g>2H$ (resp.\ $\kappa_g<2H$), being $\kappa_g$ the geodesic curvature of $v_i$ as a curve in a horizontal plane with respect to the unit normal $N$ of $\Sigma_{a,b}$. Lemma~\ref{lem:vertical-geodesics} also shows that the projection of $N$ points to the exterior (resp.\ interior) of $\Omega_{a,b}\subset\mathbb{H}^2(\kappa)$ if $i$ is odd (resp.\ even).
  
  \item If $\widetilde p_i$ is ideal (only possible in the case  $4H^2+\kappa<0$), then we can reason similarly for a sequence of graphs over finite triangles $\widetilde T_{a_n,b}$ or $\widetilde T_{a,b_n}$, with $a_n,b_n<\infty$, converging to $T_{a,b}$. Since $\theta_i'$ converges uniformly to zero as $n\to\infty$, Lemma~\ref{lem:vertical-geodesics} reveals that the geodesic curvature of $v_i$ is $2H$ with respect to $N$. As a limit of curves in the assumption of item (1), we infer that the projection of $N$ points to the exterior (resp.\ interior) of the domain $\Omega_{a,b}$ along $v_i$ if $i$ is odd (resp.\ even). Hence, the geodesic curvature of $v_i$ with respect to an inner conormal to $\Omega_{a,b}$ is $-2H$ (resp.\ $2H$) if $i$ is odd (resp.\ even). This confirms that $(H,k)$-noids (resp.\ $(H,k)$-nodoids) are asymptotic to the $H$-cylinders from the concave side (resp.\ convex side), see Figures~\ref{fig:H-noids} and~\ref{fig:H-nodoids}. 
\end{enumerate}

The complete surfaces $\Sigma_{a,b}^*$ are obtained after reflecting $\Sigma_{a,b}$ about the horizontal plane of symmetry. From the properties sketched so far, we deduce the description of the surfaces given by Theorem~\ref{th:knoids-genus-0}.

\begin{figure}[htb]
  \includegraphics[width=\textwidth]{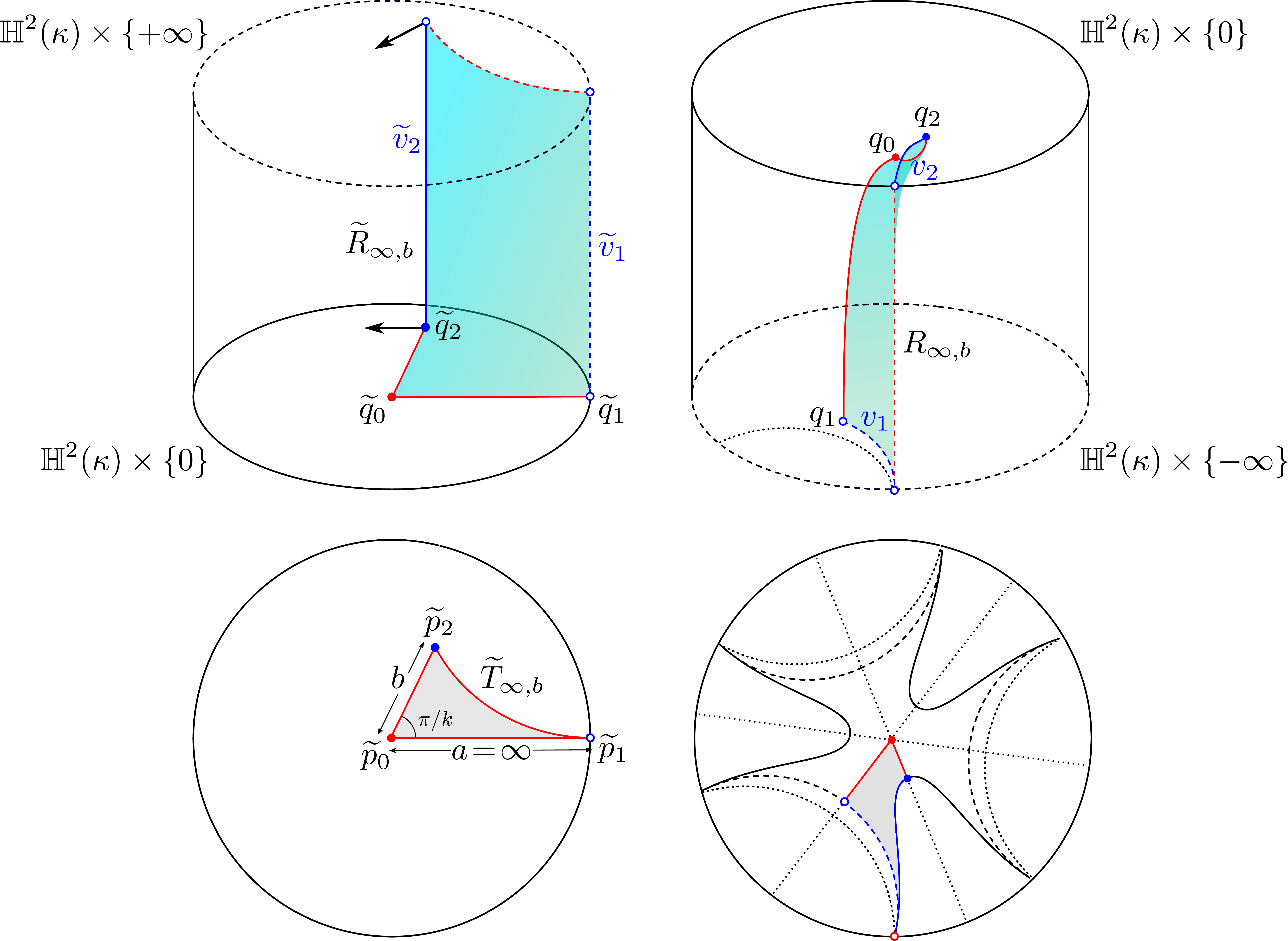}
  \caption{The fundamental piece $\widetilde R_{\infty,b}$ and its conjugate $R_{\infty,b}$ of a $(H,k)$-noid for $k=3$ and $4H^2+\kappa<0$, and their projections to $\mathbb{H}^2(\kappa)$. Observe that the asymptotic $H$-cylinders are approached from the concave side. Dotted lines are geodesics. }\label{fig:H-noids}
\end{figure}

\begin{figure}[htb]
  \includegraphics[width=\textwidth]{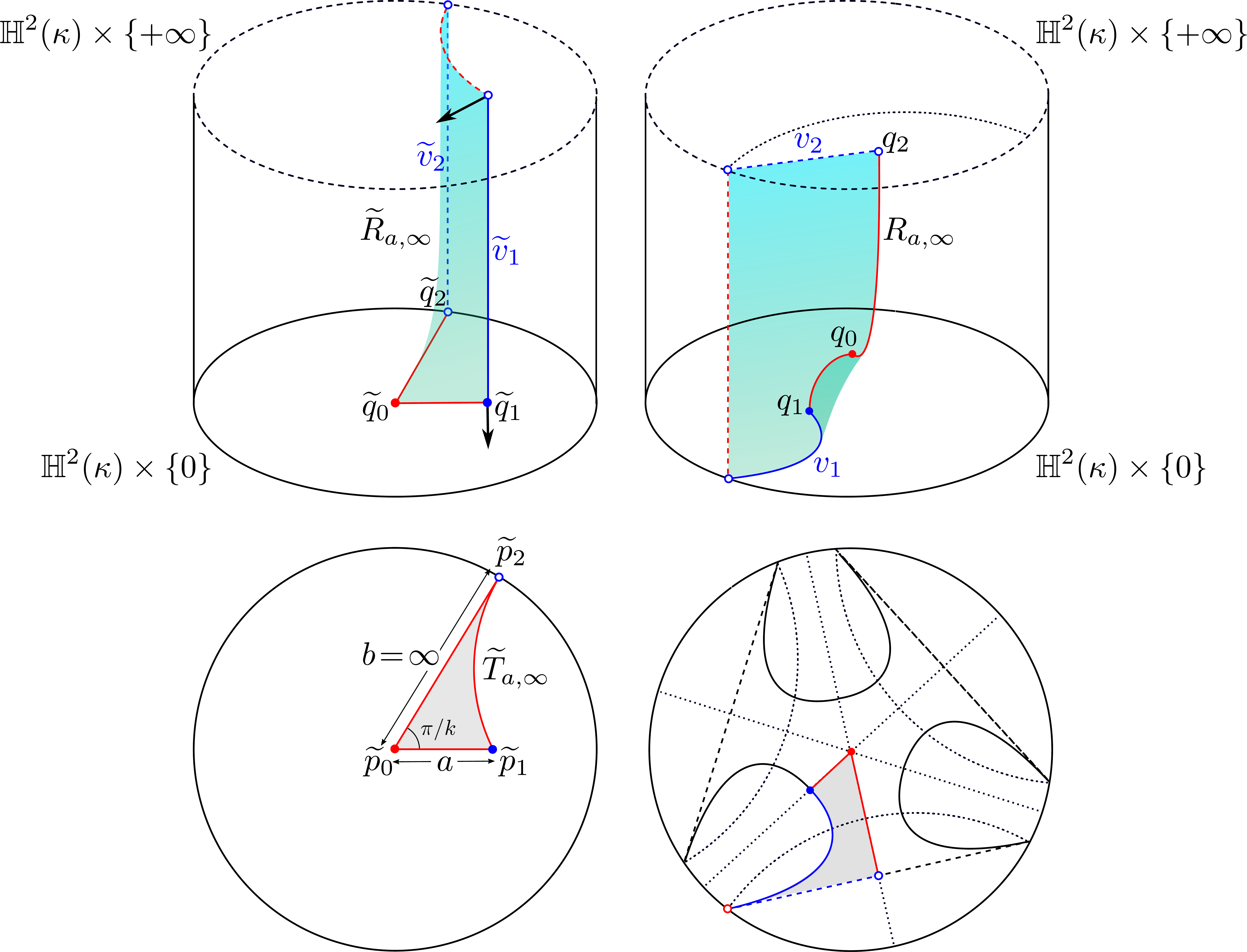} 
  \caption{The fundamental piece $\widetilde R_{a,\infty}$ and its conjugate $R_{a,\infty}$ of a $(H,k)$-nodoid for $k=3$ and $4H^2+\kappa<0$, and their projections to $\mathbb{H}^2(\kappa)$. Observe that the asymptotic $H$-cylinders are approached from the convex side. Dotted lines are geodesics.}\label{fig:H-nodoids}
\end{figure}

\begin{remark}\label{rmk:H-noids-selfintersections}
  Proposition~\ref{prop:angle-H-knoids} shows that $\Sigma_{a,b}$ behaves differently when $H=0$ and $H>0$. In the former case, the angle function $\nu_{a,b}$ only takes the value $1$ at $\widetilde q_0$ for all $k\geq 2$. However, in the later case, the height function of the conjugate surface (i.e., the projection to the factor $\R$) has a local minimum at the conjugate of $\widetilde q_0$ and saddle points at the conjugate point of $\widehat q_1$ and its symmetric points for all $k\geq 3$. So, the shape of $\Sigma_{a,b}$ is somewhat different in the cases $H=0$ and $H>0$. For instance, this situation leads to type II self-intersections for certain values $a,b<\infty$ (i.e., in the case of saddle towers) and $H>0$ because the fundamental piece escapes the slab where the boundary is contained (see~\cite[Prop.~4.6]{CMR}). 
\end{remark}

\begin{remark}\label{remark:horo-dis} To see the differences between the cases $4H^2+\kappa<0$ and $4H^2+\kappa=0$, we define the function $\rho=\rho(a,b)$ (resp.\ $d=d(a,b)$) as the distance in $\mathbb{H}^2(\kappa)$ from the center $\pi(q_0)$ of $\Omega_{a,b}$, where $q_0$ is the conjugate point of $\widetilde q_0$, to the curves $\pi(v_1)$ (resp.\ $\pi(v_2)$) in the projection of the (possibly asymptotic) boundary of $\Sigma_{a,b}$. In~\cite[Lem.~4.1]{CMR} it is shown that for $4H^2+\kappa<0$ the functions $a\mapsto d(a,\infty)$ and $b\mapsto d(\infty,b)$ are strictly increasing and range from $0$ to a finite number $d_\infty$. If $4H^2+\kappa=0$, then $d(a,\infty)=\rho(\infty,b)=\infty$ for all $a,b\in(0,\infty)$, so the $H$-cylinders asymptotic to $\Sigma_{a,\infty}$ and $\Sigma_{\infty,b}$ disappear at infinity when $4H^2+\kappa=0$. 

We also remark that, if such horocylinders did not disappear, then we would have encountered a contradiction to the halfspace theorem for surfaces of critical mean curvature given by Hauswirth, Rosenberg and Spruck~\cite{HRS}.
\end{remark}

\subsubsection{Embeddedness}\label{subsubsec:knoids-embeddedness}
We begin by observing that the embeddedness of $\Sigma_{a,b}^*$ in case $H=0$ easily follows from the Krust property (see Proposition~\ref{prop:Krust}), so we will restrict to the case $H>0$. We have already pointed out that type II self-intersections may occur even if $\Sigma_{a,b}$ is embedded (see Remark~\ref{rmk:H-noids-selfintersections}). Using Proposition~\ref{prop:angle-H-knoids} and the maximum principle with respect to horizontal planes, it is not difficult to see that this behavior is not possible for $k=2$, see~\cite[Prop.~4.6]{CMR} for the details. Said this, we cannot discard type I self-intersections either, because the counterexamples to the Krust property will actually come from this type of self-intersections. We will analyze the cases $a=\infty$ and $b=\infty$ separately.

\begin{enumerate}
  \item The embeddedness of $(H,k)$-noids: As the geodesics $\widetilde v_{2i-1}$ are ideal and the angle of rotation of the normal along $\widetilde v_{2i}$ satisfies $\theta_{2i}'>0$ for $i\in\{1,\ldots,k\}$, Lemma~\ref{lem:vertical-geodesics} and the maximum principle with respect to horizontal planes give the inclusion $\Sigma_{a,b}\subset\mathbb H^2(\kappa)\times (-\infty,0]$, where the symmetry curves of $\Sigma_{a,b}$ are contained in $\mathbb H^2(\kappa)\times\{0\}$. By Proposition~\ref{prop:rotation-curvature}, if $\int_{\widetilde v_{2i}}\theta'_{2i}\leq \pi$, then the conjugate curves $v_{2i}$ are embedded and by the symmetries of the surface the same happens to $\Sigma_{a,b}$. Therefore, the embeddedness of the $(H,k)$-noids is proved when $\int_{\widetilde v_{2i}}\theta'_{2i}\leq \pi$, i.e., when $\widetilde\Omega_{a,b}$ is a convex domain of $\mathbb M^2(4H^2+\kappa)$. This condition is always satisfies for $k=2$, so all $H$-catenoids are embedded.
  
  \item The embeddedness of $(H,k)$-nodoids: If $k\geq 3$, we do not have in general the inclusion $\Sigma_{a,b}\subset \mathbb H^2(\kappa)\times[0,+\infty)$ if we assume that $v_1$ lies in $\mathbb H^2(\kappa)\times\{0\}$. The maximum principle with respect to horizontal slices does not apply. However, for $4H^2+\kappa<0$ and $a$ large enough that inclusion can be proved using the continuity of the conjugation (Proposition~\ref{prop:continuity}), as well as the fact that $\Sigma_{a,\infty}$ converges to a Scherk $H$-graph $\Sigma_{\infty,\infty}$ as $a\to\infty$ after appropriate vertical translations (see Figure~\ref{fig:non-emb}) and the fact that the geodesic curvatures of $v_{2i-1}$ in the horizontal plane of symmetry converge to $2H$ with respect to the exterior conormal. In other words, $\Sigma_{a,\infty}$ is embedded for $a$ large enough.
  
  On the other hand, the hyperbolic distance from $\pi(v_{2})$ to the origin $\pi(q_0)$ ranges from $0$ to $d_{\infty}$ as $a$ runs from $0$ to $\infty$, see Remark~\ref{remark:horo-dis}. Since $\Sigma_{a,\infty}$ lies in the convex side of the cylinders over $v_{2i}$, there exists $a_1>0$, depending on $k$ and $H$, such that the asymptotic equidistant curves $v_2$ and $v_4$ in $\Sigma_{a_1,\infty}$ share at least one endpoint at infinity. This implies that the endpoints of the curve $v_3$ coincide when $a=a_1$ and there are type I self-intersections if and only if $a<a_1$, see Figure~\ref{fig:non-emb}. By symmetry, the same happens for any $v_{2i-1}$, so $\Sigma_{a,\infty}$ (and hence $\Sigma_{a,\infty}^*$) is not embedded in that case. If $4H^2+\kappa=0$, the curves $v_{2i}\subset\Sigma_{a,b}^*$ get close to horocycles at the same time that they diverge to infinity, see Remark~\ref{remark:horo-dis}. As $\Sigma_{a,\infty}^*$ lies locally in the convex side of the horocylinders, the curves $v_{2i-1}$ cannot be embedded. 
\end{enumerate}

All in all, if $4H^2+\kappa<0$, then $\Sigma_{a,\infty}$ is not embedded if $a<a_1$; if $4H^2+\kappa=0$, then $\Sigma_{a,\infty}$ is never embedded. In the case $k=2$, this non-embeddedness yields the counterexamples to the Krust property with $4H^2+\kappa\leq 0$, since the original minimal surface $\widetilde\Sigma_{a,b}$ is a vertical graph over a convex quadrilateral. In the case $4H^2+\kappa>0$, there are also counterexamples to the Krust property constructed by similar methods, see~\cite[\S5]{CMR}.

\begin{figure}[htb]
  \includegraphics[width=0.9\textwidth]{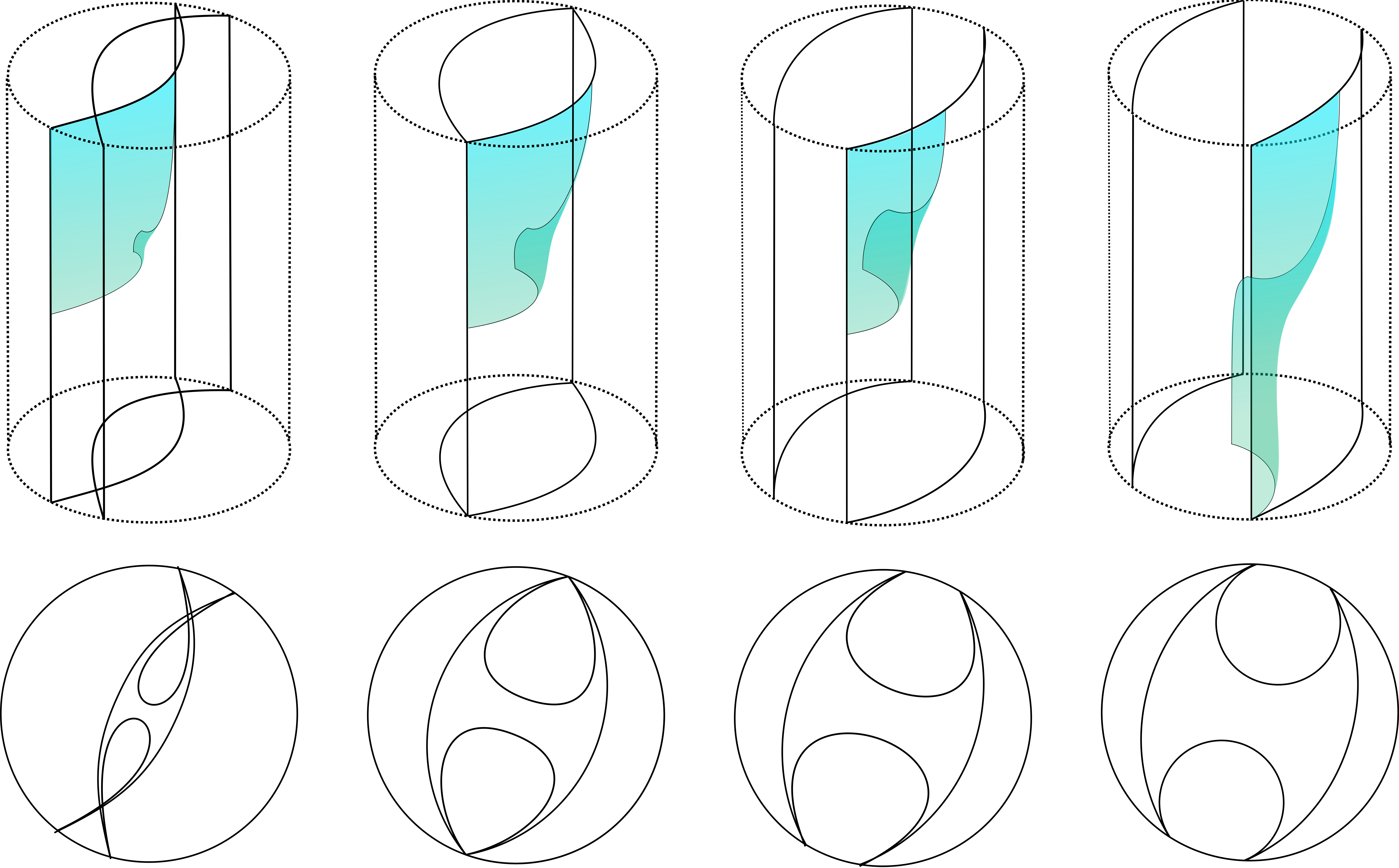}
  \caption{The fundamental piece $R_{a,\infty}$ of an $H$-catenodoid and its projection to $\mathbb{H}^2(\kappa)$. From left to right: $a<a_1$ (non-embedded), $a=a_1$ (vertical planes touching at infinity), $a>a_1$ (embedded), $a=\infty$ (limit Scherk $H$-graph after suitable vertical translations). }\label{fig:non-emb} 
\end{figure}

\begin{remark}
  It is expected that $(H,k)$-noids and $(H,k)$-nodois have finite total curvature as in the minimal case. By the properties of conjugation this problem is equivalent that the solutions of the Jenkins--Serrin problem $\widetilde R_{\infty,b}$ and $\widetilde R_{a,\infty}$ have finite total curvature in $\widetilde{\mathrm{SL}}_2(\mathbb{R})$. However, this problem is much more involved than in $\mathbb{H}^2(\kappa)\times\R$ because the Gauss curvature of minimal surfaces in $\widetilde{\mathrm{SL}}_2(\mathbb{R})$ may change sign, see the Gauss equation in~\S\ref{subsec:fundamental-data}.
\end{remark}

\subsection{Genus one minimal $k$-noids in $\mathbb{H}^2(\kappa)\times\mathbb{R}$}\label{sec:genusone}

The last application we present in this survey is a construction of minimal surfaces in $\mathbb H^2(\kappa)\times\mathbb R$ with finite total curvature by means of conjugating a solution to a Jenkins--Serrin problem in $\mathbb H^2(\kappa)\times\mathbb R$. They are analogous to the genus $1$ minimal $k$-noids in $\mathbb{R}^3$ obtained by Mazet~\cite{Mazet} and their construction is inspired by the work of Plehnert~\cite{Ple12}, who obtained similar surfaces in $\mathbb H^2(\kappa)\times\mathbb R$ with critical mean curvature. 

In the minimal case, we take advantage of the Krust property (see Proposition~\ref{prop:Krust}) as well as of the fact that our surfaces have finite total curvature, which enables a finer control of the asymptotic behaviour. Hauswirth, Nelli, Sa Earp and Toubiana~\cite{HNST}, and Hauswirth, Menezes and Rodríguez~\cite{HMR} proved that a complete minimal surface immersed in $\mathbb{H}^2(\kappa)\times\mathbb{R}$ has finite total curvature if and only if it is proper, has finite topology and each of its ends is asymptotic to an admissible polygon consisting of finitely-many complete vertical and horizontal ideal geodesics, in which case the total curvature must be a negative multiple of $2\pi$, see~\cite[Thm.~3.1]{HR}. Such a polygon is particularly well controlled when we want to obtain the surface by conjugation in view of Proposition~\ref{prop:conjugation-JS}. Also,~\cite[Thm.\ 6]{HMR} says that we have to prescribe a symmetry with respect to a horizontal slice if we want our surfaces to have finite total curvature and embedded ends.

It is worth pointing out that the literature on this topic does not contain many examples of minimal surfaces with finite total curvature in $\h^2(\kappa)\times\R$, and the surfaces given here are the first examples with genus $1$ and an arbitrary number of ends $k\geq 3$. There are two important remarks to this claim. On the one hand, our result cannot be extended to the case $k=2$ since it would contradict the uniqueness of the horizontal catenoids given by Hauswirth, Nelli, Sa Earp and Toubiana~\cite{HNST} (as usual, the condition $k\geq 3$ will appear as a natural restriction in the conjugate construction). On the other hand, Martín, Mazzeo and Rodríguez~\cite{MMR}, by means of gluing methods, obtained properly embedded minimal surfaces with finite total curvature in $\mathbb{H}^2(\kappa)\times\mathbb{R}$ of genus $g$ and $k$ ends asymptotic to vertical planes, for arbitrary genus $g\geq 0$. Nonetheless, in their result, $k$ is not arbitrary in principle but sufficiently large depending on $g$ and $\kappa$. 

\begin{theorem}[{\cite[Thm.~1]{CM}}]\label{th:knoids-genus-1}
  For each $k\geq 3$, there exists a $1$-parameter family $\Sigma_\varphi^*$, with $\frac\pi k\leq\varphi\leq\frac\pi 2$, of properly Alexandrov-embedded minimal surfaces in $\mathbb{H}^2(\kappa)\times\mathbb{R}$ with genus $1$ and $k$ ends. They are invariant by mirror symmetries about a horizontal plane and about $k$ equiangular vertical planes, and have finite total curvature $-4k\pi$. Moreover, each of their ends is embedded and asymptotic to a vertical plane. 
\end{theorem}

This construction can be adapted to produce minimal surfaces in $\h^2(\kappa)\times\R$ invariant by an arbitrary vertical translation, with genus $1$ and finite total curvature in the quotient of $\h^2(\kappa)\times\R$ by the vertical translation, see~\cite[Thm.~2]{CM}. These are the genus-$1$ counterparts of Morabito and Rodríguez' saddle towers~\cite{MR}.

\subsubsection{The minimal surface in $\mathbb{H}^2(\kappa)\times\R$}
Given $a>0$ and $0<\varphi<\tfrac{\pi}{2}$, consider the triangle $\widetilde \Delta(a,\varphi)\subset\h^2(\kappa)$ with one ideal vertex $\widetilde p_1$ and two interior vertexes $\widetilde p_2$ and $\widetilde p_3$, such that the finite edge $\widetilde p_2\widetilde p_3$ has length $a$ and the angle in the vertex $\widetilde p_2=(0,0)$ is equal to $\varphi$ (see Figure~\ref{fig:conjugate-genus-1} bottom left). We will work in the global Cartan model given in \S\ref{subsec:working-coordinates} assuming that $\widetilde p_2=(0,0)$ and $\widetilde p_1=(\frac{2}{\sqrt{-\kappa}},0)$.

Our initial minimal piece is the unique minimal vertical graph $\widetilde \Sigma(a,\varphi,b)$ in $\h^2(\kappa)\times\R$ that solves the Jenkins--Serrin problem over $\widetilde \Delta(a,\varphi)$ with boundary data $b$ over $\widetilde p_2\widetilde p_3$, $+\infty$ over $\widetilde p_1\widetilde p_3$ and $0$ over $\widetilde p_1\widetilde p_2$. The existence and uniqueness of solution is guaranteed by Theorem~\ref{thm:general-JS} and Proposition~\ref{prop:conjugation-JS}. The finite boundary of $\widetilde \Sigma(a,\varphi,b)$ consists of the vertical segment $\widetilde v_2$ of length $b$ projecting to $\widetilde p_2$, the vertical half-line $\widetilde v_3$ projecting to $\widetilde p_3$, and the horizontal geodesics $\widetilde h_1$ and $\widetilde h_3$ contained in $\h^2(\kappa)\times\{b\}$ and $\h^2(\kappa)\times\{0\}$ that project to $\widetilde p_2\widetilde p_3$ and $\widetilde p_1\widetilde p_2$, respectively. Moreover, the asymptotic boundary of $\widetilde \Sigma(a,\varphi,b)$ consists of a vertical half-line $\widetilde v_1$ projecting to $\widetilde p_1$ and the horizontal geodesic $\widetilde h_2$ contained in $\h^2(\kappa)\times\{+\infty\}$ that projects to ${\widetilde p_1\widetilde p_3}$, see Figure~\ref{fig:conjugate-genus-1}. The interior of $\widetilde\Sigma(a,\varphi,b)$ is a graph, where we will assume that the angle function is positive. Next, we analyze the horizontal and vertical points as in the previous constructions (again, this analysis follows from the boundary maximum principle together with Strategy 2 in~\S\ref{subsubsec:angle-control}).

\begin{proposition}[{\cite[Lem.~2]{CM}}]\label{prop:angle01:genus1}
  Let $\nu_{a,\varphi,b}$ be the angle function of the minimal surface $\widetilde \Sigma(a,\varphi,b)$, which will be assumed positive in the interior.
  \begin{enumerate}[label=\emph{(\alph*)}]
    \item The points with $\nu_{a,\varphi,b}=0$ are precisely those at $\widetilde v_2\cup \widetilde v_3$.
    
    \item  There is exactly one point with $\nu_{a,\varphi,b}=1$ and it belongs to $\widetilde h_1$.
  \end{enumerate}
\end{proposition}

We also remark here that, if $b>0$, the angles of rotation of $\theta_2$ and $\theta_3$ of the normal $\widetilde N$ along $\widetilde v_2$ and $\widetilde v_3$ satisfy $\theta_2'>0$ and $\theta_3'>0$, respectively, see Figure~\ref{fig:conjugate-genus-1}.

\subsubsection{The conjugate minimal surface}\label{subsubsec:genus-1-conjugate}

Since $\widetilde\Delta(a,\varphi)$ is convex, the Krust property (Proposition~\ref{prop:Krust}) implies that the conjugate surface $\Sigma(a,\varphi,b)$ is also a vertical minimal graph in $\mathbb{H}^2(\kappa)\times\R$ over some domain $\Delta(a,\varphi,b)\subset\h^2(\kappa)$. By Lemma~\ref{lem:vertical-geodesics}, the curves $v_2 $ and $v_3$ in $\partial\Sigma(a,\varphi,b)$ are contained in horizontal slices, being their projections convex with respect to the inner-pointing conormal to $\Delta(a,\varphi,b)$ along its boundary. Moreover, $\Sigma(a,\varphi,b)$ lies locally above the horizontal planes containing $v_2$ and $v_3$. By Lemma~\ref{lem:horizontal-geodesics}, the conjugate curves $h_1$ and $h_3$ lie in vertical planes, and it can be shown that the component of $h_1$ in the factor $\mathbb R$ has a minimum at the unique point where $\nu=1$. The asymptotic boundary of $\Sigma(a,\varphi,b)$ is composed of the half horizontal geodesics $v_1$ in $\h^2(\kappa)\times\{-\infty\}$ and of the ideal vertical half-line $h_2$, see Figure~\ref{fig:conjugate-genus-1}.

\begin{figure} 
  \includegraphics[width=0.9\textwidth]{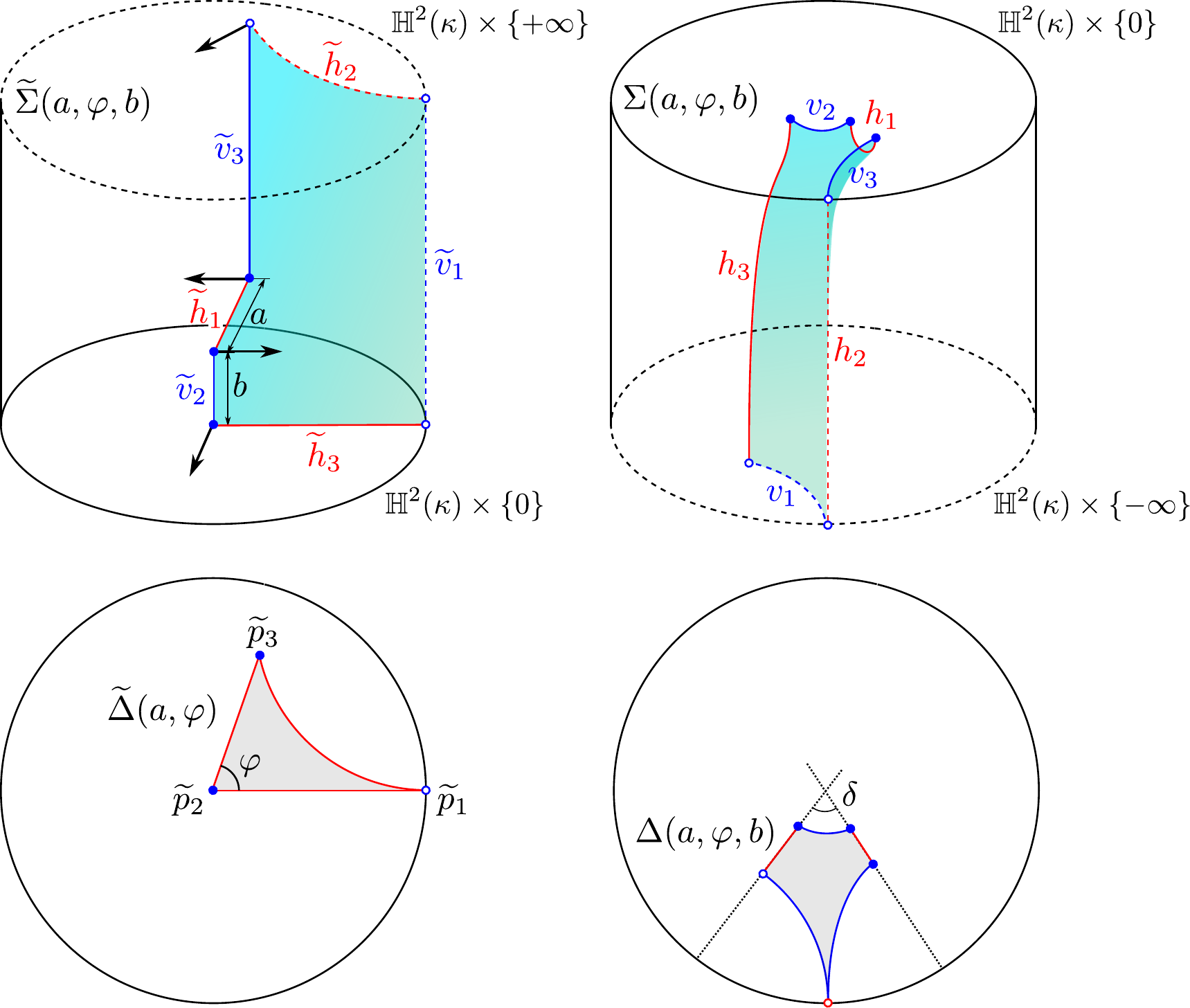}
  \caption{The conjugate surfaces $\widetilde \Sigma(a,\varphi,b)$ and $\Sigma(a,\varphi,b)$ and their projections to $\mathbb H^2(\kappa)$. The arrows represent the normal $\widetilde N$ at the endpoints of $\widetilde v_2$ and $\widetilde v_3$. The period $\mathcal{P}_1$ is zero when $v_2$ and $v_3$ are at the same height; the period $\mathcal P_2$ is the cosine of $\delta$, if such an angle exists.}\label{fig:conjugate-genus-1}
\end{figure}

We aim at showing that appropriate parameters give rise to a complete minimal surface $\Sigma^*(a,\varphi,b)$ of genus $1$ in $\h^2(\kappa)\times \R$ after extending $\Sigma(a,\varphi,b)$ by mirror symmetries over the horizontal and vertical planes. We will sketch the proof that for each $\frac\pi k<\varphi<\frac\pi2$, there exist $a_\varphi$ and $b_\varphi$ such that $\Sigma_\varphi^*=\Sigma^*(a_\varphi,\varphi,b_\varphi)$ has the desired properties. This will be accomplished if the following two periods are closed, inspired by~\cite[\S6.3]{Ple12}.
\begin{enumerate}
  \item \textbf{First period problem.} We call $\mathcal{P}_1(a,\varphi,b)$ the difference of heights of the endpoints of $h_1$, which must be zero so each end of $\Sigma(a,\varphi,b)$ is an annulus. Parametrizing $h_1:[0,a]\to\mathbb{H}^2(\kappa)\times\mathbb{R}$ with endpoints $h_1(0)\in v_2$, $ h_1(a)\in v_3$ and unit speed, the properties of the conjugation yield
  \begin{equation}\label{eqn:p1}
    \mathcal{P}_1(a,\varphi,b)=\int_{ h_1}\langle h_1',\xi\rangle=\int_{\widetilde h_1}\langle \eta,\xi\rangle,
  \end{equation}
  where $\eta=-J\widetilde h_1'$ is the unit inward conormal vector to $\widetilde \Sigma(a,\varphi,b)$ along $\widetilde h_1$. 
  
  \item \textbf{Second period problem.} Assume that the vertical planes containing the symmetry curves $h_1$ and $h_3$ intersect each other at a non-oriented angle $\delta$, and call $\mathcal P_2$ the cosine of the angle $\delta$. To give an analytic expression for $\mathcal P_2$, we will consider the half-space model, see~\S\ref{subsubsec:halfspace-model}.
  
  Parametrize $v_2(t)=(x(t),y(t),0)$ for $t\in[0,b]$ and assume, up to an ambient isometry, that $h_3$ and $v_2$ lie in the vertical plane $\{x=0\}$ and the horizontal plane $\{z=0\}$, respectively, and also $x(0)=0$, $y(0)=1$, and $x(t)<0$ when $t$ is close to $0$, see Figure~\ref{fig:horocycle-foliation}. Let $\psi\in\mathcal C^\infty[0,b]$ be the angle of rotation of $v_2$ with respect to the horocycle foliation in the sense of Remark~\ref{rmk:horocycle-rotation}, where we choose the initial angle $\psi(0)=\pi$. The second period is given by
  \begin{equation}\label{eqn:p2}
    \mathcal P_2(a,\varphi,b)=\cos(\delta)=\frac{x(b)\sin(\psi(b))}{y(b)}-\cos(\psi(b)).
  \end{equation}
  Incidentally, the right-hand side of~\eqref{eqn:p2} is well defined even when the  vertical planes containing the symmetry curves $h_1$ and $h_3$ do not intersect. However, provided that the first period is solved, it can be shown that the vertical planes containing the symmetry curves $h_1$ and $h_3$ intersect each other with an angle $\delta$ if and only if $\mathcal P_2(a,\varphi,b)=\cos(\delta)$, see~\cite[Lem.~6]{CM}.
\end{enumerate}

\begin{figure}[htb]
\includegraphics[height=5.8cm]{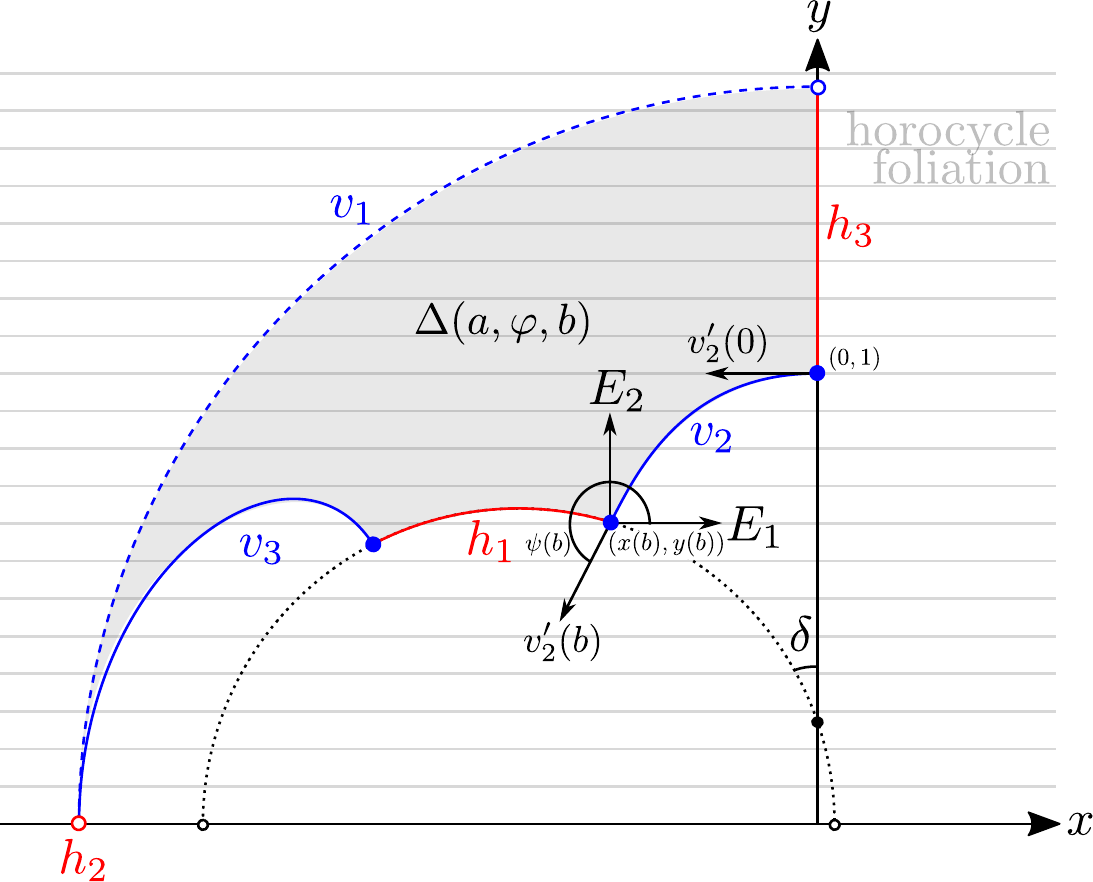}
\caption{The angle $\psi(b)$ of rotation of $v_2$ with respect to the horocycle foliation at $v_2(b)$, where we identify $\mathbb{H}^2\times\{0\}\equiv\mathbb H^2$. The surface $\Sigma(a,\varphi,b)$ projects onto the shaded region $\Delta(a,\varphi,b)$ with boundary. The complete geodesic  containing the projection of $h_1$ appears in dotted line in the case it intersects the $y$-axis with (non-oriented) angle $\delta$.}\label{fig:horocycle-foliation}
\end{figure}

\subsubsection{Solving the period problems}
We will start by discussing the first period. We will assume that $b$ is any non-negative real number and we have the restrictions $0<\varphi<\tfrac{\pi}{2}$ and $0<a<a_{\mathrm{max}}(\varphi)=2\arctanh\left(\cos(\varphi)\right)$. By hyperbolic trigonometry, this means that the angle of $\widetilde\Delta(a,\varphi)$ at $\widetilde p_3$ is strictly less than $\varphi$. This assumption is natural since a simple comparison argument shows that the first period problem cannot be solved if $a>a_{\mathrm{max}}(\varphi)$, see~\cite[Rmk.~1]{CM}. The following result guarantees that both period problems can be solved simultaneously.

\begin{lemma}[{\cite[Lem.~5 and~6]{CM}}]\label{lem:periods-problems}
  Let $\Omega=\left\{(a,\varphi)\in\mathbb{R}^2:0<\varphi<\tfrac{\pi}{2},0<a<a_{\mathrm{max}}(\varphi)\right\}$.
  \begin{enumerate}[label=\emph{(\alph*)}]      
    \item There exists a unique function $f:\Omega\to\mathbb{R}_+$ such that $\mathcal P_1(a,\varphi,f(a,\varphi))=0$ for all $(a,\varphi)\in\Omega$, which is continuous. For a fixed $\varphi_0\in(0,\frac\pi2)$, it has limits
    \[\lim\limits_{a\to a_{\mathrm{max}}(\varphi_0)}f(a,\varphi_0)=+\infty,\quad\lim\limits_{(a,\varphi)\to (0,\varphi_0)}f(a,\varphi)=0.\] 
    
    \item If $\varphi_0\in(0,\frac\pi2)$ and $b=f(a,\varphi_0)$, then the inequalities $x(t)<0$ and $\pi<\psi(t)<2\pi$ hold true for all $t\in(0,b]$ (that is, along the curve $v_2$, see~\S\ref{subsubsec:genus-1-conjugate}). We have the limits
    \[\lim_{a\to 0}\mathcal P_2(a,\varphi_0,f(a,\varphi_0))=\cos(\varphi_0),\quad 
    \lim_{a\to a_{\mathrm{max}}(\varphi_0)}\mathcal P_2(a,\varphi_0,f(a,\varphi_0))=+\infty.\]
  \end{enumerate} 
\end{lemma} 

We will highlight some ideas in the proof but we refer to~\cite{CM} for further details. In the proof of item (a), the existence of the function $f$ is based on a comparison along the boundary using Strategy 3 in~\S\ref{subsubsec:angle-control}. We show that the first period function $\mathcal P_1$ is strictly decreasing with respect to the third argument since we can compare $\widetilde \Sigma(a,\varphi,b_1)$ and $\widetilde \Sigma(a,\varphi,b_2)$ for $0<b_1<b_2$ and their conormals along the horizontal geodesic $\widetilde h_1$ by the boundary maximum principle. Moreover, we show that $\mathcal P_1(a,\varphi,0)>0$ and $\mathcal P_1(a,\varphi,b)<0$ for $b>0$ large enough. The monotonicity of $\mathcal P_1$ implies that there is a unique $b_0>0$ such that $\mathcal P_1(a,\varphi,b_0)=0$, and this uniqueness implies in turn the uniqueness and continuity of $f$.

The computation of the limits in (a) is also based on the same comparison idea with the limit surfaces. For instance, in the limit as $a\to 0$, we use Proposition~\ref{prop:continuity} and Remark~\ref{rmk:rescaling} to rescale the space while keeping $a$ constant. The limit surface lies in Euclidean space $\R^3$ and is a minimal graph over a truncated strip to which we can also apply a similar comparison argument.

As for item (b), the result is based of a careful use of the formula $\theta'=-\kappa_g^P=\psi'+\sqrt{-\kappa}\cos(\psi)$ given by Lemma~\ref{lem:vertical-geodesics} and Remark~\ref{rmk:horocycle-rotation}, where $\theta$ is the angle of rotation of the normal $\widetilde N$ along $\widetilde v_2$ and $\psi$ is the angle of rotation of $v_2$ with respect the horocycle foliation. On the one hand, we can integrate $\theta'=-\kappa_g^P$ to get $\int_0^b\kappa_g(t)dt=-\int_0^b \theta'(t)dt=-\varphi$. This gives estimates for the total geodesic curvature of subsets of $v_2$, that can be used to prove the inequalities $x(t)<0$ and $\pi<\psi(t)<2\pi$ via the Gau\ss--Bonnet formula (applied to appropriate domains, see~\cite[Fig.~5]{CM}). On the other hand, we can integrate $\theta'=\psi'+\sqrt{-\kappa}\cos(\psi)$ to obtain $\varphi=\psi(b)-\pi+\sqrt{-\kappa}\int_0^b\cos(\psi(t))\df t$. In particular, $\psi(b)\to\varphi+\pi$ and $(x(b),y(b))\to(0,1)$ as $b\to 0$. Using this and the limits in item (a), the first limit in item (b) can be easily deduced. We will skip the proof of the second one, which is more technical and involves the limit of surfaces.

In view of Lemma~\ref{lem:periods-problems}, it is not difficult to see how to finish the proof of Theorem~\ref{th:knoids-genus-1}. Given $k\geq 3$, for each $\frac{\pi}{k}<\varphi<\frac{\pi}{2}$, we choose $b=f(a,\varphi)$ to solve the first period problem. Observe that $\mathcal P_2(a,\varphi,f(a,\varphi))$ tends to $\cos(\varphi)$ when $a\to 0$ and tends to $+\infty$ when $a\to a_{\mathrm{max}}(\varphi)$; since $\cos(\varphi)<\cos(\frac{\pi}{k})$ and $\mathcal P_2$ is continuous, there exists some $a_\varphi\in(0,a_{\mathrm{max}}(\varphi))$ such that $\mathcal P_2(a,\varphi,f(a_\varphi,\varphi))=\cos(\frac{\pi}{k})$, though it might not be unique. Therefore, we choose  $b_\varphi=f(a_\varphi,\varphi)$ so that $\Sigma_\varphi^*=\Sigma^*(a_\varphi,\varphi,b_\varphi)$ solves both period problems. By a similar argument to that of Collin and Rosenberg in~\cite[Rmk.~7]{CR} using Fatou's Lemma, it follows that $\Sigma_\varphi^*$ has finite total curvature. This is also a consequence of the characterization of minimal surfaces with finite total curvature of Hauswirth, Menezes and Rodríguez in~\cite{HMR}, since $\Sigma_\varphi^*$ is proper, has finite topology and each of its end is asymptotic to a vertical plane, in particular $\Sigma_\varphi^*$ is asymptotic to an admissible polygon at infinity. Either way, the fact that $\Sigma_\varphi^*$ has finite total curvature enables a better understanding of the asymptotic behavior. For instance, it follows from~\cite{HMR} that, if an end of $\Sigma_\varphi^*$ is embedded, then it is a horizontal graph in the sense of~\cite[Def.~7 and~8]{HMR}. 

\subsubsection{Embeddedness}
The analysis of the second period function does not allow us to prove the uniqueness of  $a_\varphi$ since we have not been able to control the dependence of the second problem with respect the parameter $a$. Nevertheless, we expect that there do exist values of $\varphi$ for which the complete surface $\Sigma_\varphi^*$ will be embedded. Observe that the fundamental piece $\Sigma(a_\varphi,\varphi,b_\varphi)$ is a vertical graph contained in the half-space $\h^2(\kappa)\times(-\infty,0]$, but we can find self-intersections of type II after reflecting it about the vertical planes of symmetry, even if all periods are closed. This actually happens if $\varphi\to\frac{\pi}{k}$ because it implies that $a_\varphi\to 0$, and the surfaces $\Sigma_\varphi^*$ converge, after rescaling, to a genus $1$ minimal $k$-noid in $\R^3$, which is not globally embedded.

We can ensure that $\Sigma_\varphi^*$ is embedded if the value $a_\varphi$ that solves both period problem is bigger than the quantity $a_{\text{emb}}(\varphi)=\text{arcsinh}(\cot(\varphi))$. By a simple application of hyperbolic trigonometry, the inequality $a_\varphi\geq a_{\text{emb}}(\varphi)$ amounts to saying that the angle of $\widetilde\Delta(a_\varphi,\varphi)$ at $\widetilde p_3$ is at most $\frac\pi2$, so that the fundamental piece is still a vertical graph over a convex domain after extending it by axial symmetry about $\widetilde h_1$. Although we are not able to prove that there are values of $a_\varphi$ that satisfy this inequality, we expect that the surface $\Sigma_\varphi^*$ is embedded if $\varphi$ is close to $\frac\pi 2$.

The fact that the ends of $\Sigma_\varphi^*$ are embedded is a consequence of the fact that each of them is contained in four copies of the fundamental piece that form a symmetric embedded bigraph. This claim follows from the fact that two of these four pieces come from the fundamental piece extended by axial symmetric about $\widetilde h_2$, and the extended surface projects to a convex quadrilateral of $\mathbb{H}^2(\kappa)$. The Krust property guarantees that the conjugate surface is graph, and the other two copies needed to produce the aforesaid bigraph are their symmetric ones with respect to the slice $\h^2(\kappa)\times\{0\}$ containing $\widetilde v_2$ and $\widetilde v_3$.

\begin{remark}  
  If $\mathcal P_2(a,\varphi,f(a,\varphi))\geq 1$, then the completion $\Sigma^*(a,\varphi,f(a,\varphi))$ is a surface invariant by a discrete group of parabolic or hyperbolic translations (depending on whether $\mathcal P_2=1$ or $\mathcal P_2>1$, respectively), instead of a discrete group of rotations. We call these examples \emph{parabolic} and \emph{hyperbolic $\infty$-noids}, respectively, see Figure~\ref{fig:infty-noids}. The former are obtained when the vertical planes of symmetry of $\Sigma(a,\varphi_0,f(a,\varphi_0))$ are asymptotic, whereas in the latter these planes lie at positive distance. Parabolic and hyperbolic $\infty$-noids have genus $0$ and infinitely many ends, each of them asymptotic to a vertical plane and having finite total curvature. We also remark that these surfaces induce surfaces with finite total curvature in some quotients of $\mathbb{H}^2(\kappa)\times\mathbb{R}$ in which minimal surfaces of finite total curvature have been described by Hauswirth and Menezes~\cite{HM}. In the case of hyperbolic $\infty$-noids, we can always find values of the parameters such that $a>a_{\rm emb}(\varphi)$, so that family always contain embedded examples.
\end{remark}

\begin{figure}[t]
  \includegraphics[width=\textwidth]{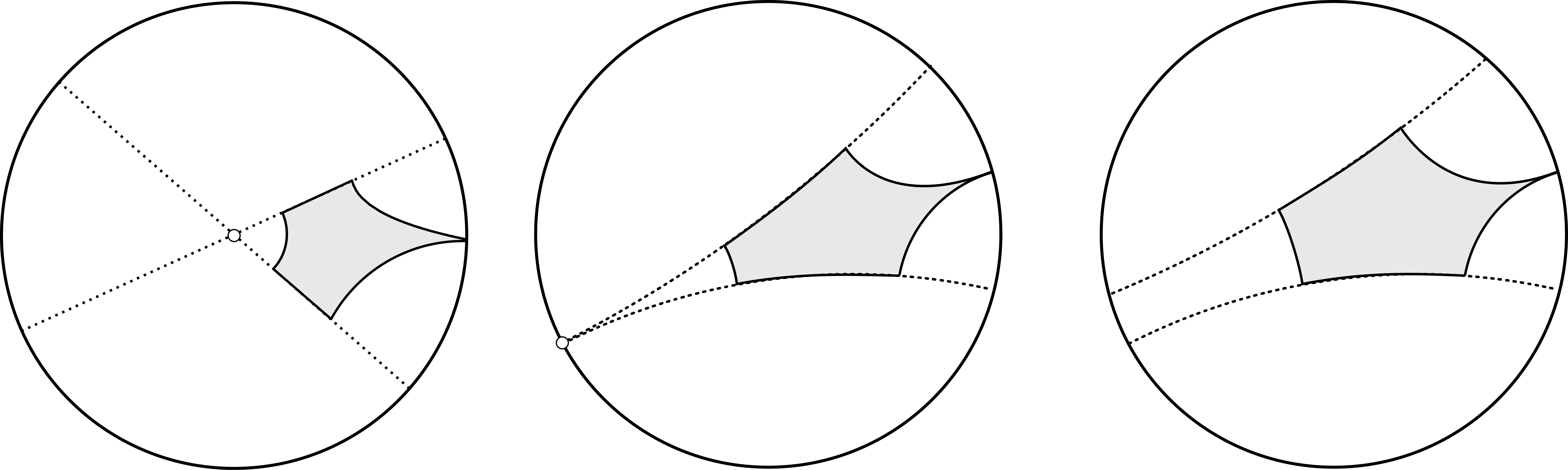}
  \caption{The fundamental domains of a $3$-noid (left), a parabolic $\infty$-noid (center), and a hyperbolic $\infty$-noid (right).} \label{fig:infty-noids}
\end{figure}


\section[Numerical examples]{Numerical examples of minimal surfaces in product spaces}
\label{sec:numerical-examples}

The aim of this last section is to present some numerical experiments that help us to visualize the minimal surfaces constructed in~\S\ref{sec:periodic}. To this end, we will use Kenneth Brakke's \emph{Surface Evolver} \cite{Brakke1992} (version 2.70), which is publicly available in \url{https://facstaff.susqu.edu/brakke/evolver/evolver.html}. This software has been successfully used to approximate both minimal and non-zero constant mean curvature surfaces in the Euclidean space, see for example~\cite{HKS1992, Grosse-Brauckmann1997} and Brakke's gallery of triply periodic minimal surfaces~\cite{BrakkeLib}. Surface Evolver is also able to perform the computation of the adjoint\footnote{See \url{http://facstaff.susqu.edu/brakke/evolver/html/scripts.htm\#adjoint.cmd}} of a minimal surface in $\mathbb{R}^3$ as well as the conjugate\footnote{See \url{http://facstaff.susqu.edu/brakke/evolver/html/scripts.htm\#cmccousin.cmd}} discrete constant mean curvature surface in $\mathbb{R}^3$ of a minimal surface in $\mathbb{S}^3$. The scripts are implemented following the algorithm developed by Pinkall and Polthier~\cite{PP1993} and Oberknapp and Polthier~\cite{OP1997}. These procedures are based on the computation of the discrete conjugate harmonic map, which is possible in space forms thanks to the close relation between harmonic maps and minimal surfaces. However, this approach is not available in $\mathbb{E}(\kappa,\tau)$-spaces.

Surface Evolver is an interactive program that minimizes energies of triangulated surfaces subject to constraints and boundary conditions. The default energy is the surface tension or the area functional, but Surface Evolver is able to operate with many other quantities like gravitation or even user-defined ones. A \emph{surface} is implemented as a triangulation, initially defined by the user in an input \emph{datafile} by prescribing the vertexes and the incidence relations between edges and faces of the triangulation. The program iteration \emph{evolves} the initial surface by minimizing the energy towards a possible local minimum close to the initial configuration by a gradient descent method. However, the numerical algorithm can also find critical saddle points.

To avoid problems with the triangulation in the evolution process, Surface Evolver provides commands for vertex averaging (\texttt{V}) and equitriangulation (\texttt{u}) as well as commands to modify the triangulation by eliminating elongated triangles (\texttt{K}), small edges (\texttt{l} and \texttt{t}) or faces (\texttt{w}). The usual evolution consists in iterating by gradient descent with the \texttt{g} command, refining the triangulation when necessary with the \texttt{r} command, and using the previous commands to keep the triangulation in good shape throughout the process. In general, after a reasonable amount of iterations the energy stalls and the resulting surface is usually near a critical point of the energy. However, some subtleties might be in place (see, for example, the evolution of the unstable catenoid in the Surface Evolver manual).

For our purposes, a key feature of Surface Evolver is its ability to operate with any Riemannian metric in the Euclidean space. However, according to its \href{http://facstaff.susqu.edu/brakke/evolver/html/model.htm#Riemannian-metric}{manual}\footnote{See \url{http://facstaff.susqu.edu/brakke/evolver/html/model.htm\#Riemannian-metric}} \emph{``the metric is used solely to calculate lengths and areas''}. For instance, it is not used for computing the enclosed volume so in order to get a volume constraint the user needs to define his own \emph{named quantity} (see~\S\ref{subsec:evolver-final-remarks} for further details).

We consider in the punctured Euclidean space $\mathbb{R}^3_*=\R^3-\{(0,0,0)\}$ the conformal metric $g = \frac{1}{x^2+y^2+z^2}g_0$, where $g_0$ stands for the usual inner product and $(x,y,z)$ are the standard coordinates. It follows that $\mathbb{S}^2\times \mathbb{R}$ is isometric to $(\mathbb{R}^3_*, g)$ via the map $F(p,t)=e^t p$, i.e., $\mathbb{S}^2\times \mathbb{R}$ is conformally flat. We will work with this identification from now on.
\begin{itemize}
  \item Horizontal slices $\mathbb{S}^2 \times \{t_0\}$, $t_0 \in \mathbb{R}$, are in correspondence with spheres $S$ of radius $e^{t_0}$ centered at the origin. Moreover, the reflection about the slice $\mathbb{S}^2\times \{t_0\}$ corresponds to an inversion in $\mathbb{R}^3$ with respect to the sphere $S$.
  \item Vertical cylinders $\gamma\times\mathbb{R}$, where $\gamma$ is a geodesic of $\mathbb{S}^2$, correspond to linear planes $P \subset \mathbb{R}^3$ through the origin. Moreover, the reflection about $\gamma \times \mathbb{R}$ corresponds to the Euclidean reflection about $P$.
  \item Vertical geodesics $\{p\} \times \mathbb{R} \subset \mathbb{S}^2 \times \mathbb{R}$ corresponds to straight lines through the origin and rotations around them corresponds to linear rotations in the Euclidean space.
\end{itemize}

We present in the following sections two numerical experiments. The first one, concerning the minimal sphere $\mathbb{S}^2 \times \{t_0\} \subset \mathbb{S}^2 \times \mathbb{R}$, is a toy example that helps us to understand better how Surface Evolver operates with the new metric and to know its limitations. The aim of the second one is to get an approximation of a singly periodic  minimal surface that produces the compact genus $g \geq 3$ minimal surface $\Sigma_{g,\eta}$ in the quotient $\mathbb{S}^2 \times \mathbb{S}^1(\eta)$ for certain $\eta$ (see~\S\ref{sec:genus} and Theorem~\ref{thm:orientable-minimal-arbitrary-genus-examples-S2xS1}). 


\subsection{Evolution to the minimal sphere}\label{subsec:evolver-minimal-sphere}

Our first goal is to evolve an initial parallelepiped to a sphere centered at the origin that corresponds, via the isometry $F$, with a slice $\mathbb{S}^2\times \{t_0\}$. As we will see, the choice of initial parallelepiped will determine the slice (that is, the radius of the sphere) in the final evolution. The slices $\mathbb{S}^2\times \{t_0\}$, $t_0 \in \mathbb{R}$, are stable minimal surfaces~\cite{TU2014} so Surface Evolver is expected to approximate properly such a surface.

We start loading a datafile into Surface Evolver with a parallelepiped inscribed in the sphere of radius $1$ centered at the origin, where we specify the command \verb|conformal_metric 1/(x^2 + y^2 + z^2)| at the beginning of the file to set the aforesaid conformal metric. We first refine the rough initial triangulation (see Figure~\ref{fig:evolver-sphere} left) three times and then use the \verb|V| and \verb|u| commands (see Figure~\ref{fig:evolver-sphere} center). Finally, we evolve the surface 100 times (see Figure~\ref{fig:evolver-sphere}) using the \verb|g| command. After that, each step in the gradient descent method only decreases the area by approximately $10^{-4}$, giving a value about $12.6557$. This approximates the expected value $4\pi \approx 12.5664$ within an error of order $10^{-2}$. 

\begin{figure}[htbp]
  \centering
  \includegraphics[width=0.3\linewidth]{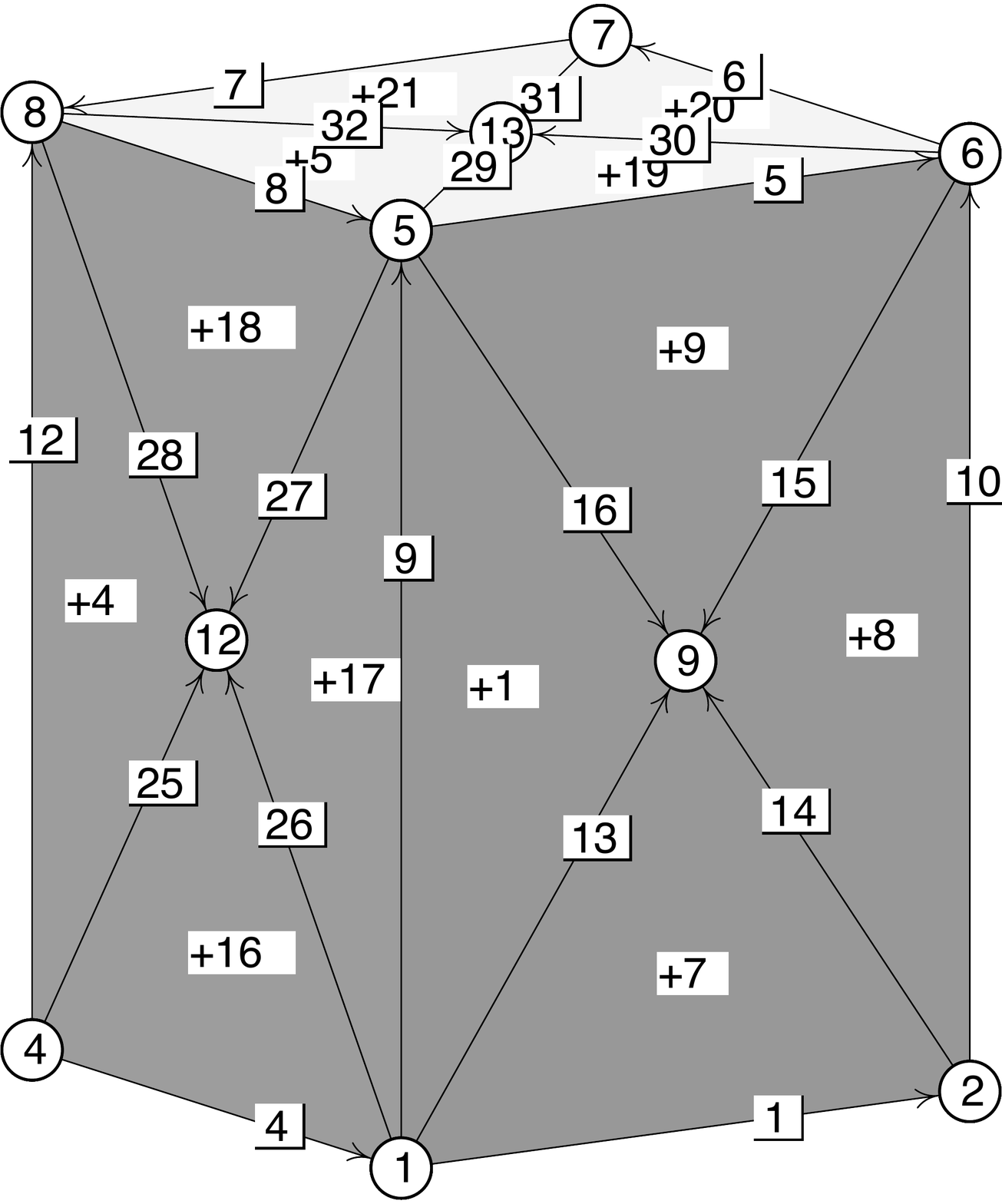}
  \quad
  \includegraphics[width=0.3\linewidth]{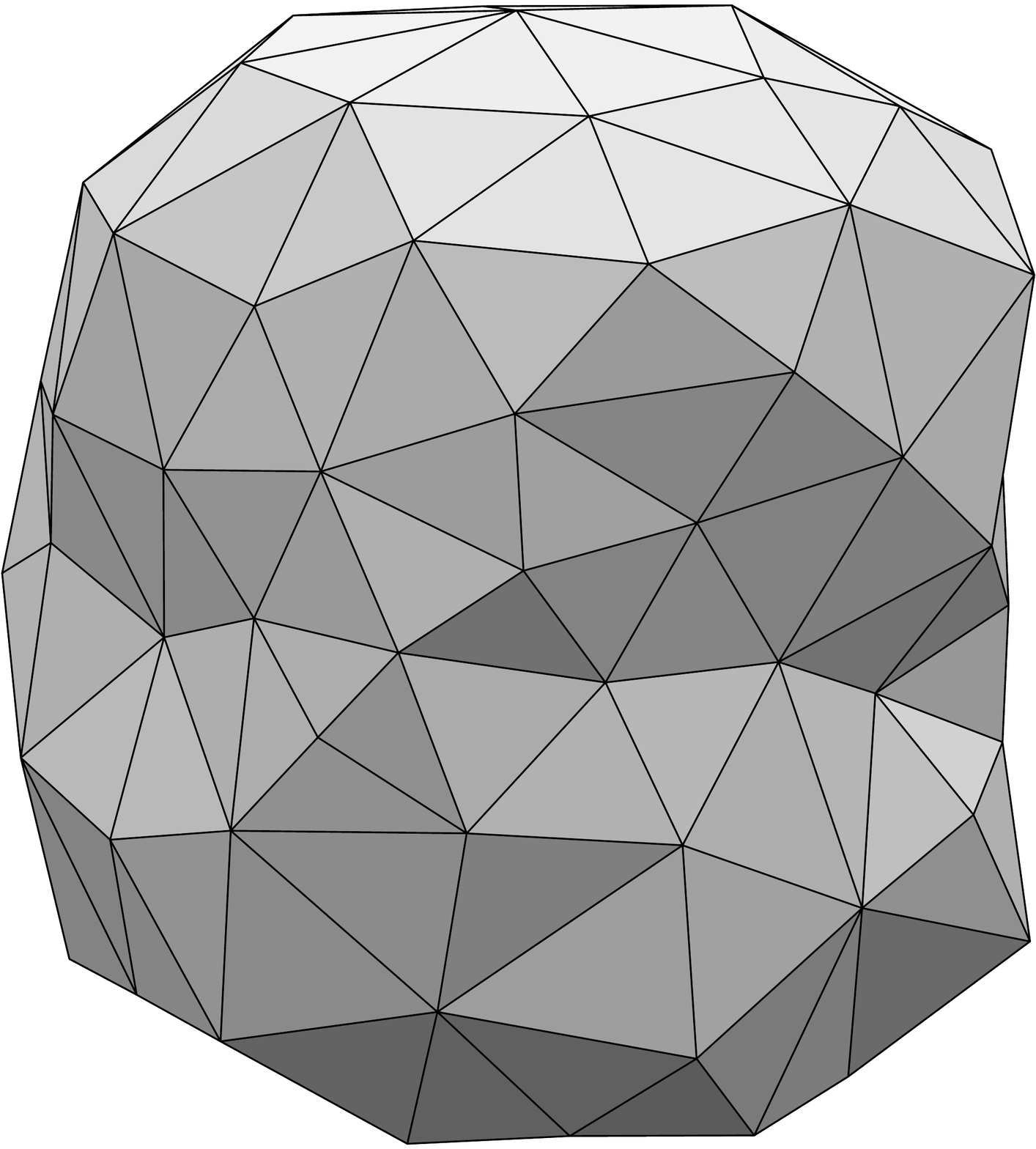}
  \quad
  \includegraphics[width=0.3\linewidth]{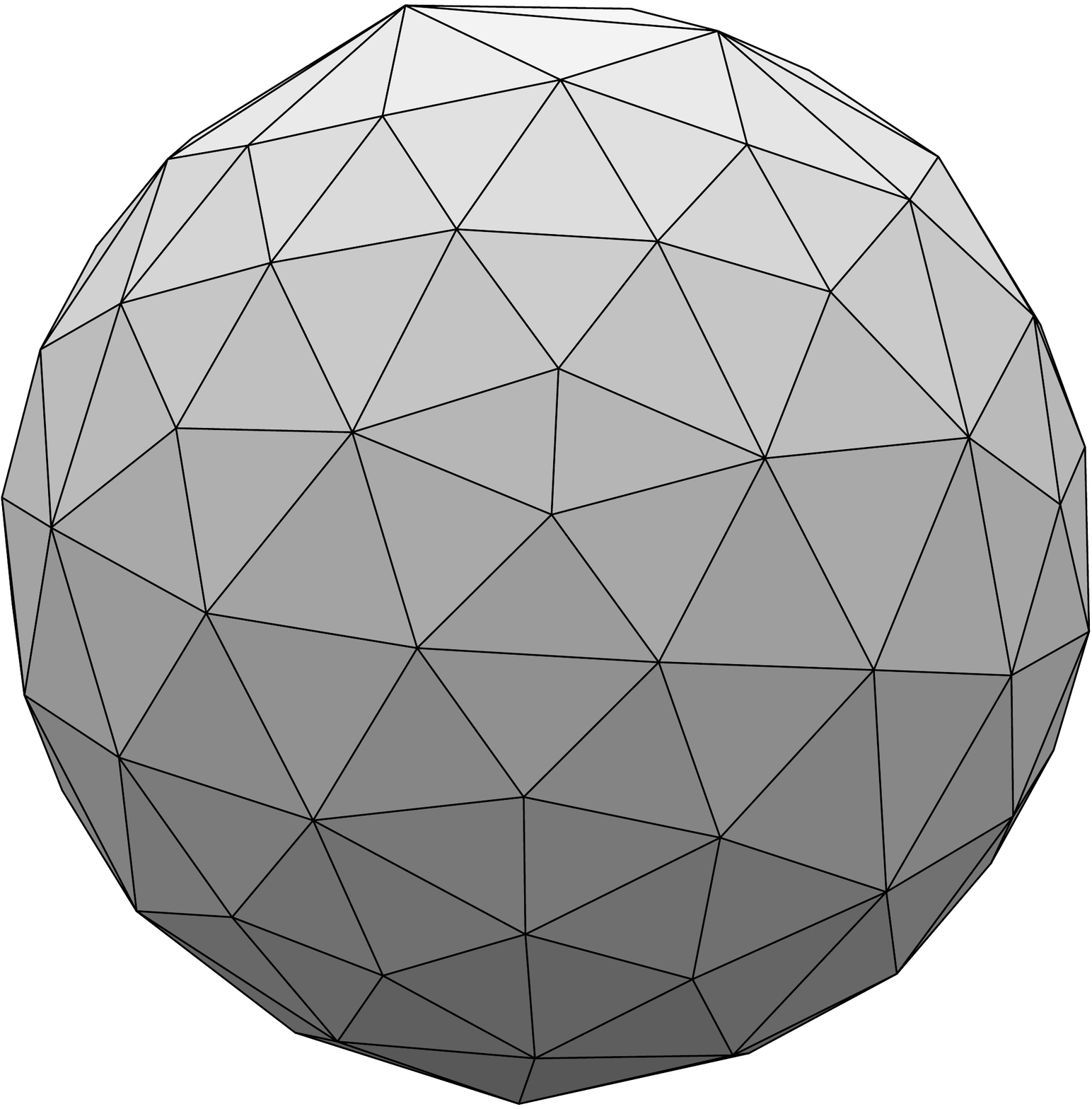}
  \caption{Initial triangulation (left) where vertexes are marked in circled numbers while edges in squared ones (faces have a plus or minus sign depending on its orientation). Triangulation after refining (center) using \texttt{r}, \texttt{u} and \texttt{V} commands. Final state of the evolution (right) after $100$ steps with the \texttt{g} command.}
  \label{fig:evolver-sphere}
\end{figure}

Surface Evolver can show the Euclidean volume enclosed by the surface. In the final step of evolution such enclosed volume is approximately $1.225$ so the surface approaches a sphere of radius $0.66$. If we start with the same parallelepiped but inscribed in a sphere of radius $2$, the same evolution yields a sphere of approximately the same area $12.6557$ (which is the expected behavior) but of Euclidean volume approximately $9.796$, i.e., the radius of the sphere is approximately $1.327$. Finally, Surface Evolver is able to compute the Euclidean discrete mean curvature of the triangulated surface giving an average of $1.517$ and $0.758$ in the first and second cases which approximately agrees with the computed radii.

\subsection{Singly periodic minimal surfaces in $\mathbb{S}^2 \times \mathbb{R}$}\label{subsec:evolver-single-periodic}

Now, we are interested in getting an approximation of the compact minimal surfaces in $\mathbb{S}^2\times \mathbb{S}^1(\eta)$ with arbitrary genus $g \geq 3$ obtained in~\S\ref{sec:compact}. We recall that $\eta \geq 2\sqrt{\kappa}$ has to be large enough to guarantee the existence of the surface (see Theorem~\ref{thm:orientable-minimal-arbitrary-genus-examples-S2xS1}). Those surfaces are obtained as the quotient of singly periodic (by a vertical translation) minimal surfaces in $\mathbb{S}^2 \times \mathbb{R}$. We will use Surface Evolver to approximate the latter (see Figures~\ref{fig:evolver-single-periodic-fundamental-piece} and~\ref{fig:evolver-genus}).

\begin{figure}[htbp]
  \centering
  \includegraphics[width=0.3\textwidth]{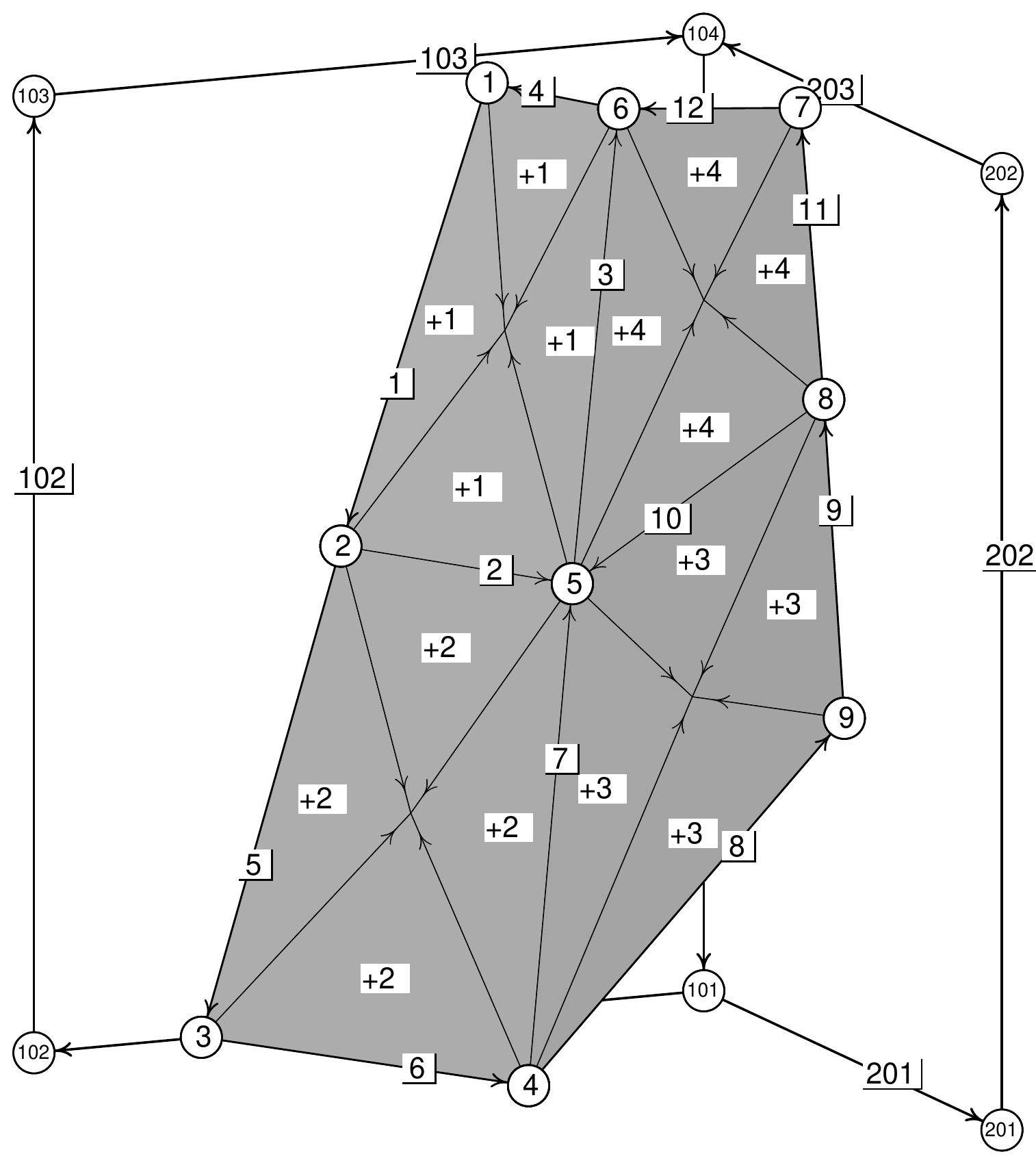}
  \quad
  \includegraphics[width=0.3\textwidth]{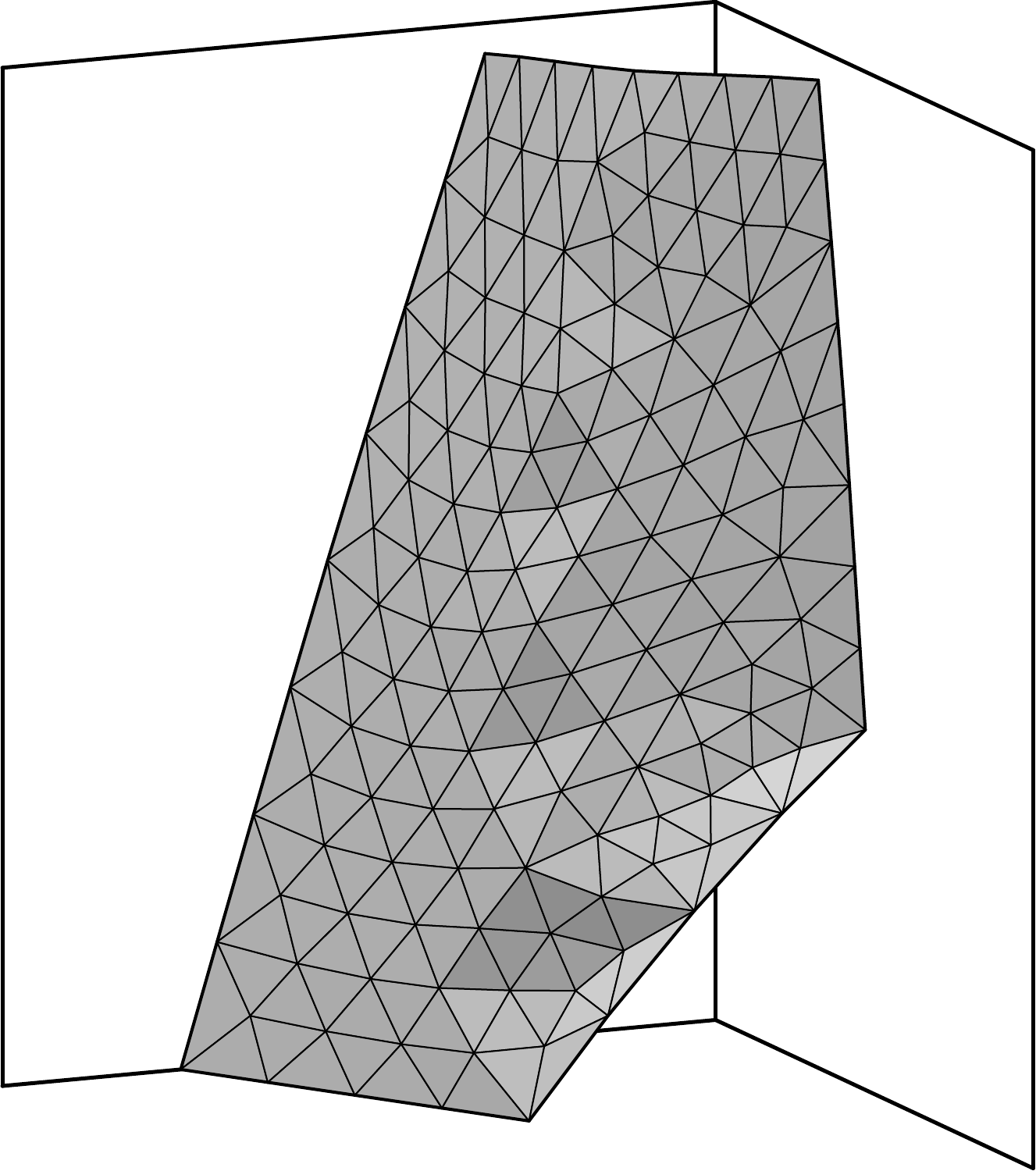}
  \quad
  \includegraphics[width=0.3\textwidth]{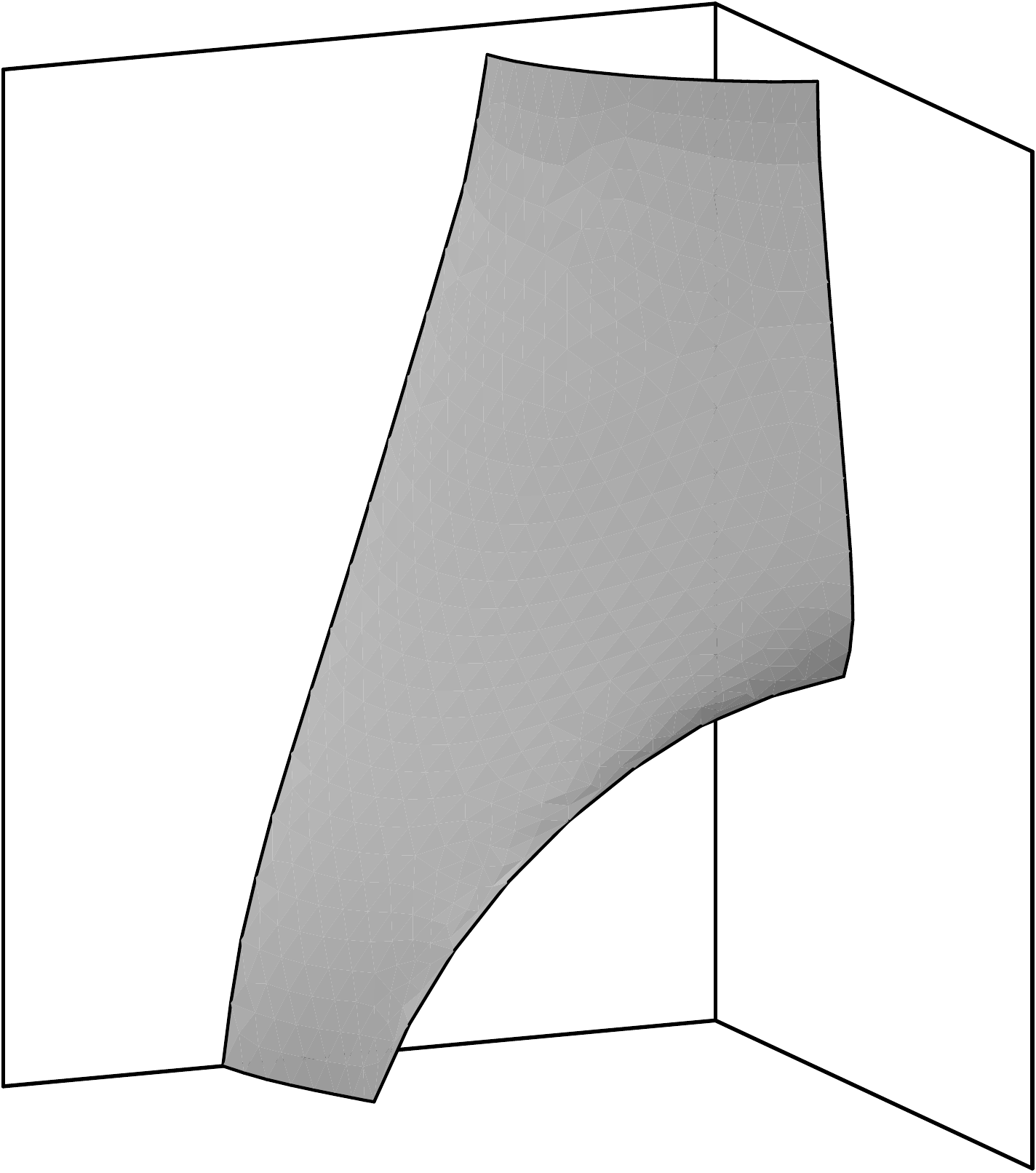}
  \caption{Initial triangulation (left) for genus $g = 3$: the edges $\overline{49}$ and $\overline{17}$ are respectively constrained to the slice $\mathbb{S}^2\times \{0\}$ (sphere of radius $1$) and the slice $\mathbb{S}^2\times \{h\}$ (sphere of radius $e^h$, $h \approx 0.7$ in the figure); the edges $\overline{13}$, $\overline{34}$ and $\overline{79}$ are constrained to the vertical planes of symmetry (planes $y = 0$, $z = 0$ and $-x \sin(\frac{\pi}{g-1}) + y \cos(\frac{\pi}{g-1}) = 0$ respectively). Triangulation after refining using \texttt{r}, \texttt{u} and \texttt{V} commands (center). Final state of evolution after $150$ steps using the \texttt{g} commands and taking care of the triangulation each $50$ steps (right). The symmetry planes are drawn to help visualization}
  \label{fig:evolver-single-periodic-fundamental-piece}
\end{figure}

The triangulation defined in the datafile used to generate Figures~\ref{fig:evolver-single-periodic-fundamental-piece} and~\ref{fig:evolver-genus} depends on several parameters that provide an easy way to test the evolution process for different initial configurations. It is possible to change the genus, that actually controls the angle between the symmetry planes drawn in Figure~\ref{fig:evolver-single-periodic-fundamental-piece}, as well as the \emph{height} $h$ of the fundamental piece (i.e.\ half the length of the vertical translation that leaves the surface invariant, see Figures~\ref{fig:minimal-S2xS1-arbitrary-genus} and~\ref{fig:P-Schwarz-initial} right). After loading the datafile in Surface Evolver we first improve the initial triangulation (with \texttt{u} and \texttt{V} commands) and then we evolve the surface 150 steps (with \texttt{g} command) taking care of the triangulation each 50 steps. At this stage, we observe that the area only decreases by $0.01$ each step. Further iteration with the \texttt{g} commands just decreases the area slightly, which insinuates  that the surface might be in a critical saddle point. We also notice that the \emph{scale factor} (a real number that controls the size of the motion at each step of the iteration) becomes small due to the fact that the area of some faces approaches zero as the triangulation accumulates around the boundary $\overline{49}$ (see Figure~\ref{fig:evolver-single-periodic-fundamental-piece}), which is constrained to the sphere of radius one. This suggests that there is an obstacle to the evolution\footnote{The same behaviour is observed in the free boundary example \texttt{free\_bdry.fe} (see \url{http://facstaff.susqu.edu/brakke/evolver/workshop/html/day2.htm}) where the scale factor converges to zero as we iterate with the \texttt{g} command.}. However, trying to overcome this issue either by activating the \emph{conjugate gradient method} (with \texttt{U} command) or by removing the small faces (with \texttt{w} command), as suggested in Surface Evolver manual, and then evolving the surface (with \texttt{g} command) produces a collapse near the boundary $\overline{49}$. 

\begin{figure}[htbp]
  \centering
  \includegraphics[height=0.25\textheight]{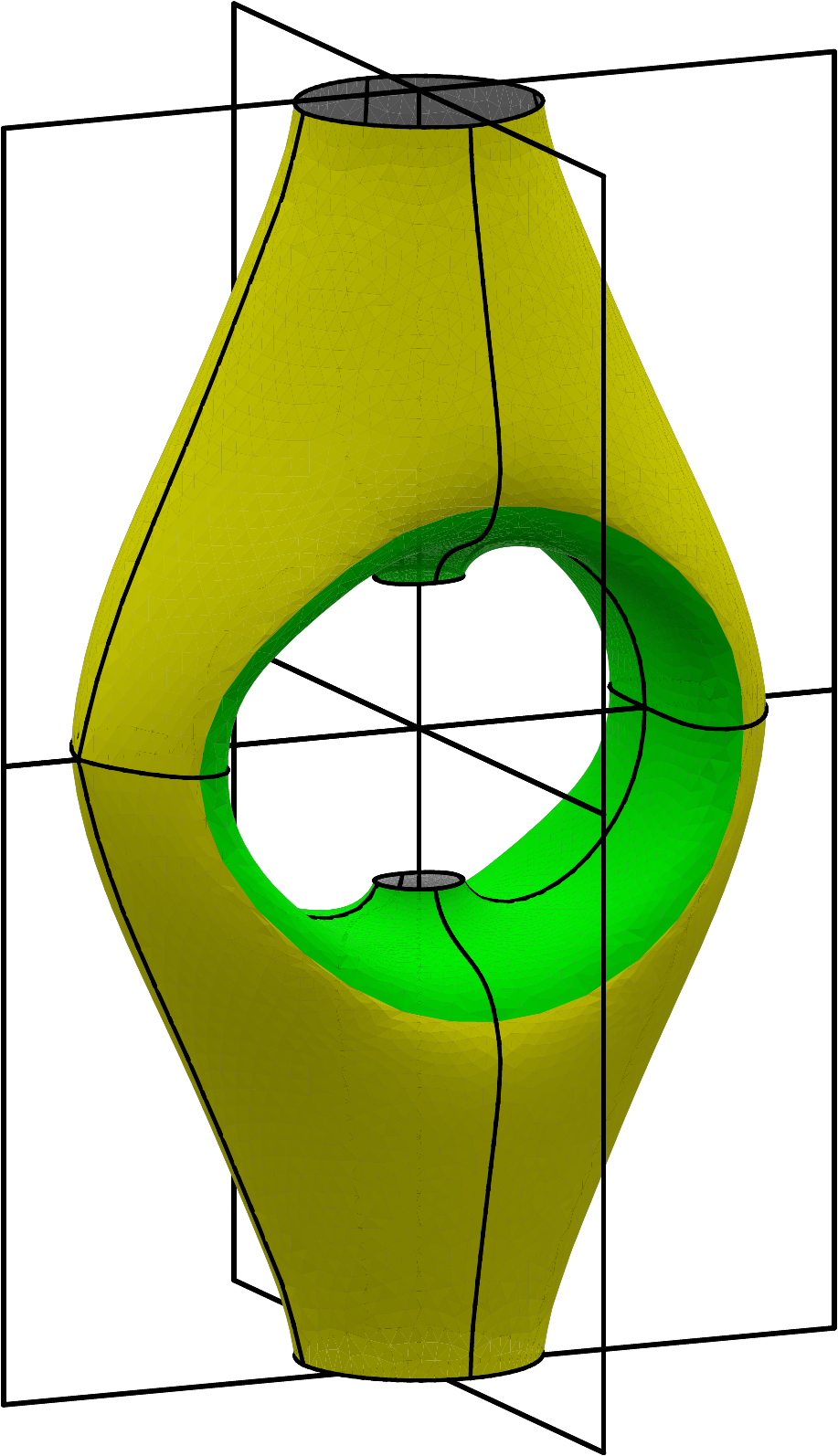}
  \hfill
  \includegraphics[height=0.25\textheight]{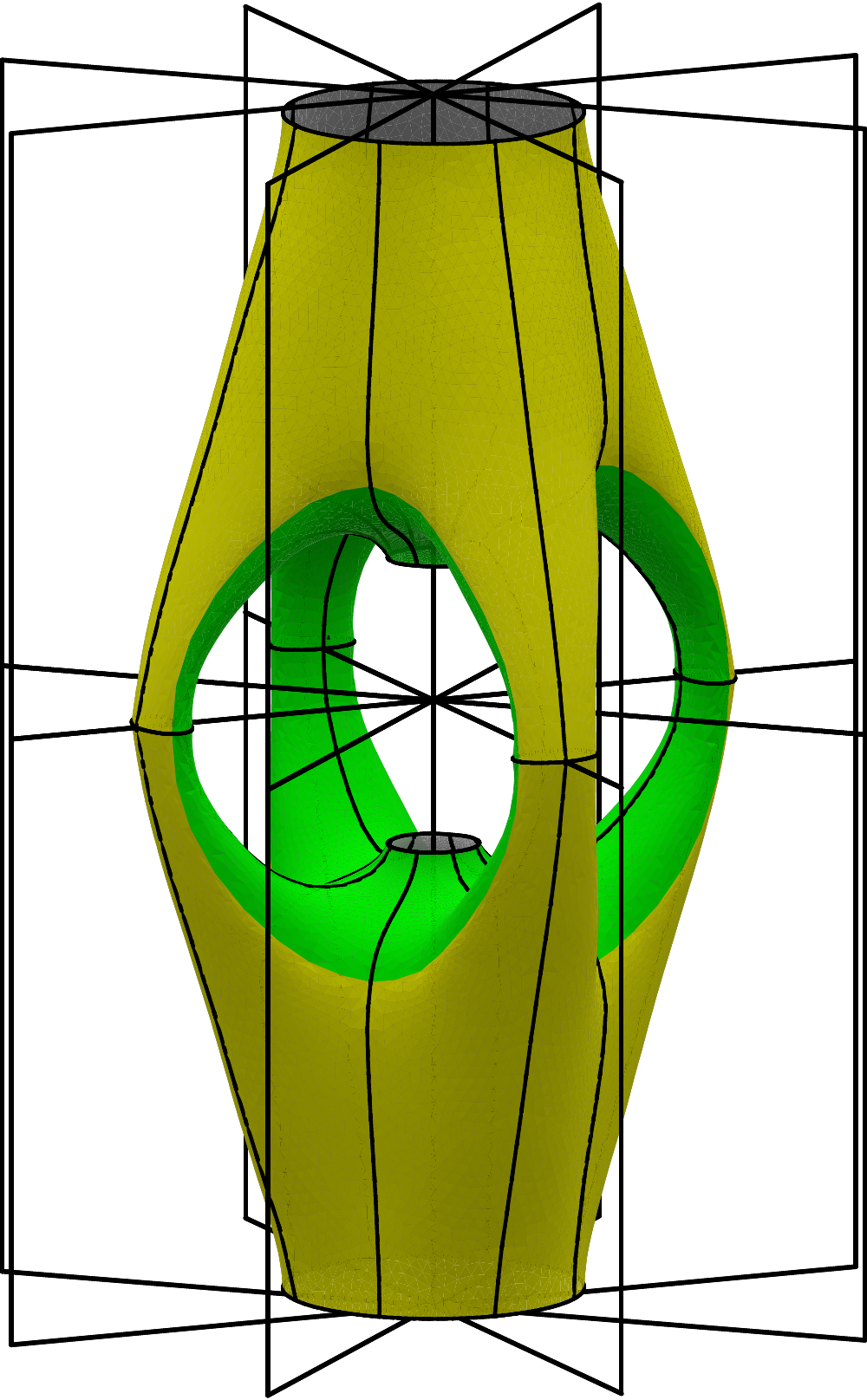}
  \hfill
  \includegraphics[height=0.25\textheight]{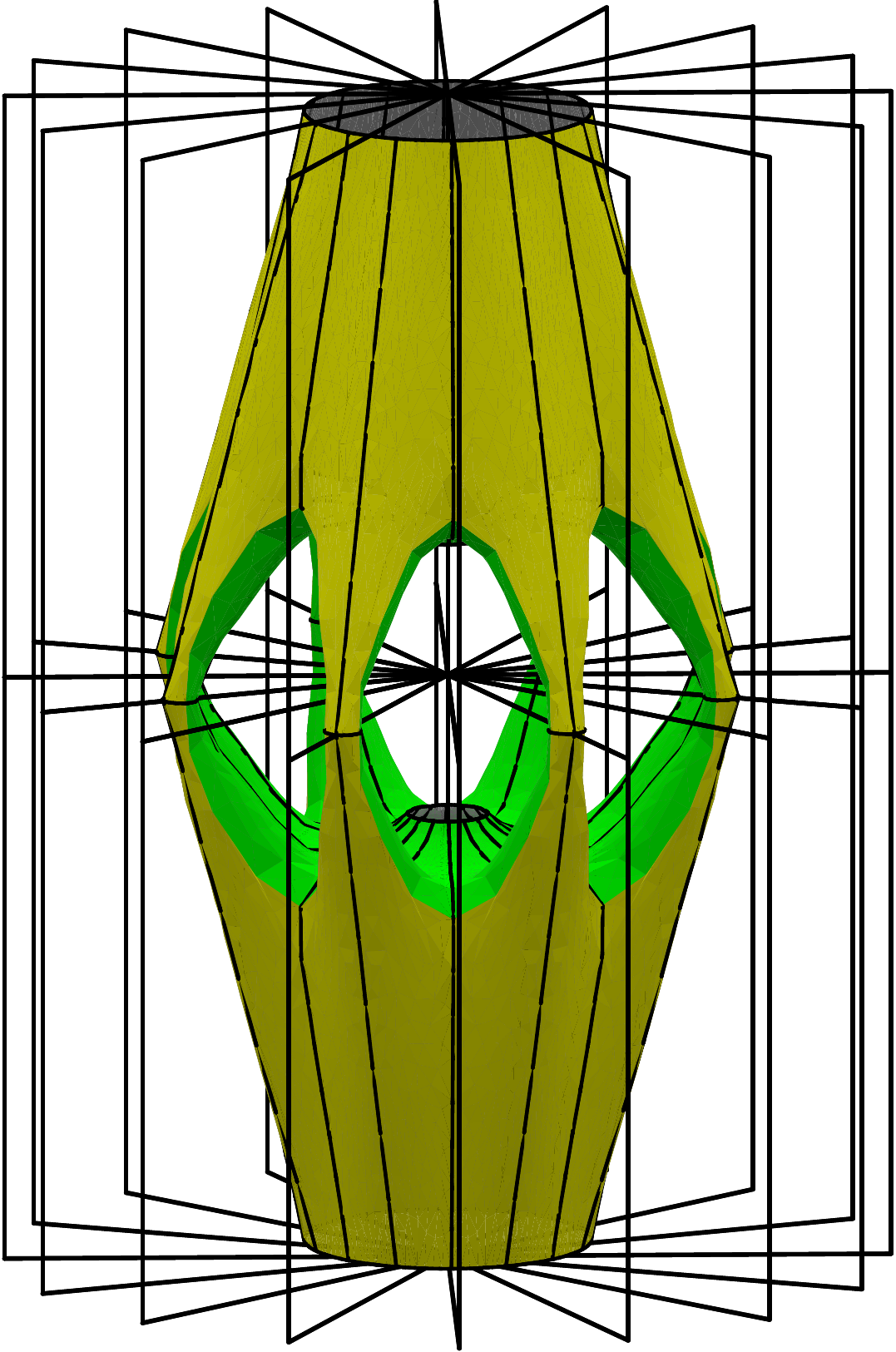}
  \caption{From left to right: visual approximation in $(\mathbb{R}^3_*, g)$ of three different singly periodic minimal surfaces of $\mathbb{S}^2 \times \mathbb{R}$. These surfaces, from left to right, produces compact minimal examples of genus $3$, $5$ and $9$ in the quotient $\mathbb{S}^2\times \mathbb{S}^1(\tfrac{\pi}{h})$ ($h \approx 0.7$ in the figure). The green and yellow parts are congruent in $\mathbb{S}^2 \times \mathbb{R}$, i.e.\ they are congruent in the figure by an inversion of the sphere of radius $1$.}
  \label{fig:evolver-genus}
\end{figure}

\subsection{Final remarks and future work}\label{subsec:evolver-final-remarks}

The experiments of the previous section have produced surfaces that resemble the theoretical ones (see the detailed description in \S\ref{subsubsec:P-Schwarz-embeddedness} and compare Figures~\ref{fig:P-Schwarz-initial} and~\ref{fig:evolver-single-periodic-fundamental-piece}). However, due to the change of metric, the program exhibits some limitations in order to check the precision of the approximation. On the one hand, even after prescribing the metric, Surface Evolver computes the discrete mean curvature at each vertex of the triangulation with respect to the Euclidean metric, as shown in~\S\ref{subsec:evolver-minimal-sphere}. As a consequence, we cannot check the precision of the approximation by computing the deviation of the discrete mean curvature from zero.  On the other hand, the change of metric does not affect other quantities like the volume (see next paragraph) or the implementation of the Willmore functional\footnote{\texttt{star\_perp\_sq\_mean\_curvature}, see \url{http://facstaff.susqu.edu/brakke/evolver/workshop/doc/quants.htm\#star_perp_sq_mean_curvature}}, both of which computed in the Euclidean metric. The latter could have been used to easily check if the evolution is near a minimal surface. Other problem we have found (see~\S\ref{subsec:evolver-single-periodic}) is that the scale factor approaches zero as we iterate using the command \texttt{g} so the conducted experiments invariably lead to a collapse near one of the necks of the surface. This could indicate some instability in the discrete free boundary problem.

However, Surface Evolver is flexible enough to allow the user to define his own \emph{quantities}. In this sense and as a future work, it will be interesting to extend Surface Evolver's functionality to overcome the mentioned issues. For instance, an implementation of the enclosed volume with respect the new metric will be extremely useful to get approximations of the $H$-surfaces constructed in~\S\ref{sec:genus} decreasing the area functional with a constrain on the volume. Another quite interesting solution is to look for an algorithm able to compute the discrete conjugate of a minimal surface in $\mathbb{E}(\kappa,\tau)$. We also propose a final approach, motivated by the fact that the initial minimal surface seems to be more tractable (Surface Evolver finds satisfactorily the unique solution to a Plateau problem given by Proposition~\ref{prop:plateau-existencia}). Therefore, it should be possible to get good approximations of the angle function (resp.\ rotation of the normal) along the horizontal (resp.\ vertical) geodesics of the boundary, so the different period problems that we have encountered in our constructions can be potentially solved numerically \emph{a priori}. This would provide us with precious information to be plugged into the initial configuration of the conjugate surface.


\end{document}